\documentclass[12pt,reqno]{amsart}
\def\draft{n}
\usepackage{
  fullpage,graphicx,dbnsymb,amsmath,amssymb,picins,multicol,mathtools,
  stmaryrd,xcolor,mathdots,mathabx,needspace,enumitem,tensor
}
\usepackage[all]{xy}
\usepackage{xr}
\usepackage{bigfoot}
\DeclareNewFootnote{T}

\SelectFootnoteRule[1]{ctalk}
\DeclareNewFootnote{C}

\usepackage[pagebackref,breaklinks=true]{hyperref}\hypersetup{colorlinks,
  linkcolor={green!50!black},
  citecolor={green!50!black},
  urlcolor=blue
}

\usepackage{chngcntr}
\counterwithin{figure}{section}

\if\draft y
  \usepackage[color]{showkeys}
\fi

\theoremstyle{plain}
\newtheorem{theorem}{Theorem}[section]

\newtheorem{proposition}[theorem]{Proposition}
\newtheorem{defprop}[theorem]{Definition-Proposition}

\theoremstyle{definition}
\newtheorem{definition}[theorem]{Definition}

\newtheorem{figcap}[theorem]{Figure}

\theoremstyle{remark}
\newtheorem{comment}[theorem]{Comment}
\newtheorem{comments}[theorem]{Comments}
\newtheorem{discussion}[theorem]{Discussion}

\newlength{\standardunitlength}
\newcommand{\ad}{\operatorname{ad}}

\newcommand{\im}{\operatorname{im}}

\newcommand{\tr}{\operatorname{tr}}

\def\qed{{\linebreak[1]\null\hfill\text{$\Box$}}}

\newlength{\globalparindent}
\newenvironment{myitemize}{
        \begin{list}{$\bullet$}{\setlength{\leftmargin}{16pt}
        \setlength{\labelwidth}{12pt}
        \setlength{\labelsep}{4pt}}
}{
        \end{list}
}

\def\arXiv#1{{\href{http://front.math.ucdavis.edu/#1}{arXiv:\linebreak[0]#1}}}

\catcode`\@=11
\long\def\@makecaption#1#2{%
    \vskip 10pt
    \setbox\@tempboxa\hbox{%\ifvoid\tinybox\else\box\tinybox\fi
      \small\sf{\bfcaptionfont #1. }\ignorespaces #2}%
    \ifdim \wd\@tempboxa >\captionwidth {%
        \rightskip=\@captionmargin\leftskip=\@captionmargin
        \unhbox\@tempboxa\par}%
      \else
        \hbox to\hsize{\hfil\box\@tempboxa\hfil}%
    \fi}
\font\bfcaptionfont=cmssbx10 scaled \magstephalf
\newdimen\@captionmargin\@captionmargin=2\parindent
\newdimen\captionwidth\captionwidth=\hsize
\catcode`\@=12

\def\web{http://drorbn.net/AcademicPensieve/Projects/WKO4}
\def\citeweb#1{{\cite{WKO4}/\href{\web/#1}{\tt #1}}}

\def\draftcut{\if\draft y \cleardoublepage \fi}

\definecolor{lightred}{RGB}{255, 217, 217}
\def\lightred#1{\setlength{\fboxsep}{0pt}\colorbox{lightred}{#1}}
\def\gray{\color{gray}}
\def\yellowm#1{{\setlength{\fboxsep}{0pt}\colorbox{yellow}{$#1$}}}
\def\yellowt#1{{\setlength{\fboxsep}{0pt}\colorbox{yellow}{#1}}}

\def\aAS{{\overrightarrow{AS}}}
\def\act{{\hspace{-1pt}\sslash\hspace{-0.75pt}}}
\def\aIHX{{\overrightarrow{IHX}}}
\def\aSTU{{\overrightarrow{STU}}}
\def\atder{\operatorname{\mathfrak{der}}}
\def\atdiv{\operatorname{div}}
\def\atsder{\operatorname{\mathfrak{sder}}}
\def\attder{\operatorname{\mathfrak{tder}}}
\def\attr{\operatorname{\mathfrak{tr}}}

\def\bbQ{{\mathbb Q}}
\def\bbR{{\mathbb R}}
\def\bbZ{{\mathbb Z}}
\def\BCH{\operatorname{BCH}}
\def\calA{{\mathcal A}}

\def\calE{{\mathcal E}}
\def\calK{{\mathcal K}}
\def\calP{{\mathcal P}}
\def\calS{{\mathcal S}}
\def\calU{{\mathcal U}}
\def\Cap{{\mathit C\!a\!p}}

\def\CW{\text{\it CW}}
\def\dA{\text{\it dA}}
\def\der{\operatorname{der}}
\def\dS{\text{\it dS}}
\def\endpar#1{~\hfill\fbox{\footnotesize\ref{#1}}}
\def\FA{\text{\it FA}}
\def\FG{\text{\it FG}}
\def\FL{\text{\it FL}}
\def\fraka{{\mathfrak a}}
\def\frakg{{\mathfrak g}}
\def\frakt{{\mathfrak t}}
\def\lie{\operatorname{\mathfrak{lie}}}
\def\PuB{{\mathit P\!u\!B}}
\def\PwB{{\mathit P\!w\!B}}
\def\remove{\setminus}

\def\SolKV{\text{SolKV}}
\def\sder{\operatorname{sder}}
\def\TAut{\operatorname{TAut}}
\def\tder{\operatorname{tder}}
\def\TW{\text{\it TW}}

\def\upcap{{\raisebox{-1mm}{\begin{picture}(0,0)%
\includegraphics{figs/SmallCap.pstex}%
\end{picture}%
\setlength{\unitlength}{3947sp}%
\begingroup\makeatletter\ifx\SetFigFont\undefined%
\gdef\SetFigFont#1#2#3#4#5{%
  \reset@font\fontsize{#1}{#2pt}%
  \fontfamily{#3}\fontseries{#4}\fontshape{#5}%
  \selectfont}%
\fi\endgroup%
\begin{picture}(64,201)(151,-1123)
\end{picture}%
}}}

\def\wTFo{{\mathit w\!T\!F^o}}

\def\imagetop#1{\vtop{\null\hbox{#1}}}

\def\face{\begin{picture}(0,0)%
\includegraphics{figs/face.pstex}%
\end{picture}%
\setlength{\unitlength}{789sp}%
\begingroup\makeatletter\ifx\SetFigFont\undefined%
\gdef\SetFigFont#1#2#3#4#5{%
  \reset@font\fontsize{#1}{#2pt}%
  \fontfamily{#3}\fontseries{#4}\fontshape{#5}%
  \selectfont}%
\fi\endgroup%
\begin{picture}(1240,1240)(-19,-381)
\end{picture}%
}
\def\human{\begin{picture}(0,0)%
\includegraphics{figs/human.pstex}%
\end{picture}%
\setlength{\unitlength}{789sp}%
\begingroup\makeatletter\ifx\SetFigFont\undefined%
\gdef\SetFigFont#1#2#3#4#5{%
  \reset@font\fontsize{#1}{#2pt}%
  \fontfamily{#3}\fontseries{#4}\fontshape{#5}%
  \selectfont}%
\fi\endgroup%
\begin{picture}(1244,2442)(-21,-1583)
\end{picture}%
}
\def\machine{\begin{picture}(0,0)%
\includegraphics{figs/machine.pstex}%
\end{picture}%
\setlength{\unitlength}{789sp}%
\begingroup\makeatletter\ifx\SetFigFont\undefined%
\gdef\SetFigFont#1#2#3#4#5{%
  \reset@font\fontsize{#1}{#2pt}%
  \fontfamily{#3}\fontseries{#4}\fontshape{#5}%
  \selectfont}%
\fi\endgroup%
\begin{picture}(1244,1394)(-21,-608)
\end{picture}%
}
\def\shortmathinclude#1{{%
  \newline\vspace{0mm}
  {\imagetop{\begin{picture}(0,0)%
\includegraphics{figs/face.pstex}%
\end{picture}%
\setlength{\unitlength}{789sp}%
\begingroup\makeatletter\ifx\SetFigFont\undefined%
\gdef\SetFigFont#1#2#3#4#5{%
  \reset@font\fontsize{#1}{#2pt}%
  \fontfamily{#3}\fontseries{#4}\fontshape{#5}%
  \selectfont}%
\fi\endgroup%
\begin{picture}(1240,1240)(-19,-381)
\end{picture}%
}\
  \imagetop{\includegraphics[scale=0.16]{ComputerTalk/#1.eps}}}
  \newline\vskip 1mm
}}
\def\mathinclude#1{{\label{C:#1}%
  \newline\vspace{0mm}
  {\imagetop{\begin{picture}(0,0)%
\includegraphics{figs/human.pstex}%
\end{picture}%
\setlength{\unitlength}{789sp}%
\begingroup\makeatletter\ifx\SetFigFont\undefined%
\gdef\SetFigFont#1#2#3#4#5{%
  \reset@font\fontsize{#1}{#2pt}%
  \fontfamily{#3}\fontseries{#4}\fontshape{#5}%
  \selectfont}%
\fi\endgroup%
\begin{picture}(1244,2442)(-21,-1583)
\end{picture}%
}\
  \imagetop{\includegraphics[scale=0.16]{ComputerTalk/#1.eps}}}
  \newline\vskip 1mm
}}
\def\shortdialoginclude#1{{\label{C:#1}%
  \noindent
  \imagetop{\begin{picture}(0,0)%
\includegraphics{figs/face.pstex}%
\end{picture}%
\setlength{\unitlength}{789sp}%
\begingroup\makeatletter\ifx\SetFigFont\undefined%
\gdef\SetFigFont#1#2#3#4#5{%
  \reset@font\fontsize{#1}{#2pt}%
  \fontfamily{#3}\fontseries{#4}\fontshape{#5}%
  \selectfont}%
\fi\endgroup%
\begin{picture}(1240,1240)(-19,-381)
\end{picture}%
}\
\imagetop{\includegraphics[scale=0.16]{ComputerTalk/Human-#1.eps}}
  \newline
  \vskip 1mm\noindent
  \imagetop{\begin{picture}(0,0)%
\includegraphics{figs/machine.pstex}%
\end{picture}%
\setlength{\unitlength}{789sp}%
\begingroup\makeatletter\ifx\SetFigFont\undefined%
\gdef\SetFigFont#1#2#3#4#5{%
  \reset@font\fontsize{#1}{#2pt}%
  \fontfamily{#3}\fontseries{#4}\fontshape{#5}%
  \selectfont}%
\fi\endgroup%
\begin{picture}(1244,1394)(-21,-608)
\end{picture}%
}\
\imagetop{\includegraphics[scale=0.16]{ComputerTalk/Machine-#1.eps}}
  \newline
}}
\def\shortdialogincludewithlink#1{{\label{C:#1}%
  \noindent
  \imagetop{\begin{picture}(0,0)%
\includegraphics{figs/face.pstex}%
\end{picture}%
\setlength{\unitlength}{789sp}%
\begingroup\makeatletter\ifx\SetFigFont\undefined%
\gdef\SetFigFont#1#2#3#4#5{%
  \reset@font\fontsize{#1}{#2pt}%
  \fontfamily{#3}\fontseries{#4}\fontshape{#5}%
  \selectfont}%
\fi\endgroup%
\begin{picture}(1240,1240)(-19,-381)
\end{picture}%
}\
\imagetop{\includegraphics[scale=0.16]{ComputerTalk/Human-#1.eps}}
  \newline
  \vskip 1mm\noindent
  \imagetop{\begin{picture}(0,0)%
\includegraphics{figs/machine.pstex}%
\end{picture}%
\setlength{\unitlength}{789sp}%
\begingroup\makeatletter\ifx\SetFigFont\undefined%
\gdef\SetFigFont#1#2#3#4#5{%
  \reset@font\fontsize{#1}{#2pt}%
  \fontfamily{#3}\fontseries{#4}\fontshape{#5}%
  \selectfont}%
\fi\endgroup%
\begin{picture}(1244,1394)(-21,-608)
\end{picture}%
}\
\imagetop{\includegraphics[scale=0.16]{ComputerTalk/Machine-#1.eps}}
  \hfill\parbox[t]{1in}{
    \rightline{Fuller output:}
    \rightline{\citeweb{#1.nb}}
  }
  \newline
}}
\def\dialoginclude#1{{\label{C:#1}%
  \noindent
  \imagetop{\begin{picture}(0,0)%
\includegraphics{figs/human.pstex}%
\end{picture}%
\setlength{\unitlength}{789sp}%
\begingroup\makeatletter\ifx\SetFigFont\undefined%
\gdef\SetFigFont#1#2#3#4#5{%
  \reset@font\fontsize{#1}{#2pt}%
  \fontfamily{#3}\fontseries{#4}\fontshape{#5}%
  \selectfont}%
\fi\endgroup%
\begin{picture}(1244,2442)(-21,-1583)
\end{picture}%
}\
\imagetop{\includegraphics[scale=0.16]{ComputerTalk/Human-#1.eps}}
  \newline
  \vskip 1mm\noindent
  \imagetop{\begin{picture}(0,0)%
\includegraphics{figs/machine.pstex}%
\end{picture}%
\setlength{\unitlength}{789sp}%
\begingroup\makeatletter\ifx\SetFigFont\undefined%
\gdef\SetFigFont#1#2#3#4#5{%
  \reset@font\fontsize{#1}{#2pt}%
  \fontfamily{#3}\fontseries{#4}\fontshape{#5}%
  \selectfont}%
\fi\endgroup%
\begin{picture}(1244,1394)(-21,-608)
\end{picture}%
}\
\imagetop{\includegraphics[scale=0.16]{ComputerTalk/Machine-#1.eps}}
  \newline
}}
\def\dialogincludewithlink#1{{\label{C:#1}%
  \noindent
  \imagetop{\begin{picture}(0,0)%
\includegraphics{figs/human.pstex}%
\end{picture}%
\setlength{\unitlength}{789sp}%
\begingroup\makeatletter\ifx\SetFigFont\undefined%
\gdef\SetFigFont#1#2#3#4#5{%
  \reset@font\fontsize{#1}{#2pt}%
  \fontfamily{#3}\fontseries{#4}\fontshape{#5}%
  \selectfont}%
\fi\endgroup%
\begin{picture}(1244,2442)(-21,-1583)
\end{picture}%
}\
\imagetop{\includegraphics[scale=0.16]{ComputerTalk/Human-#1.eps}}
  \newline
  \vskip 1mm\noindent
  \imagetop{\begin{picture}(0,0)%
\includegraphics{figs/machine.pstex}%
\end{picture}%
\setlength{\unitlength}{789sp}%
\begingroup\makeatletter\ifx\SetFigFont\undefined%
\gdef\SetFigFont#1#2#3#4#5{%
  \reset@font\fontsize{#1}{#2pt}%
  \fontfamily{#3}\fontseries{#4}\fontshape{#5}%
  \selectfont}%
\fi\endgroup%
\begin{picture}(1244,1394)(-21,-608)
\end{picture}%
}\
\imagetop{\includegraphics[scale=0.16]{ComputerTalk/Machine-#1.eps}}
  \hfill\hspace{-1in}\parbox[t]{1in}{
    \rightline{Fuller output:}
    \rightline{\citeweb{#1.nb}}
  }
  \newline
}}
\def\ACQuote#1{{\label{C:#1}%
  \newline
  \null%
  \imagetop{\includegraphics[width=6mm]{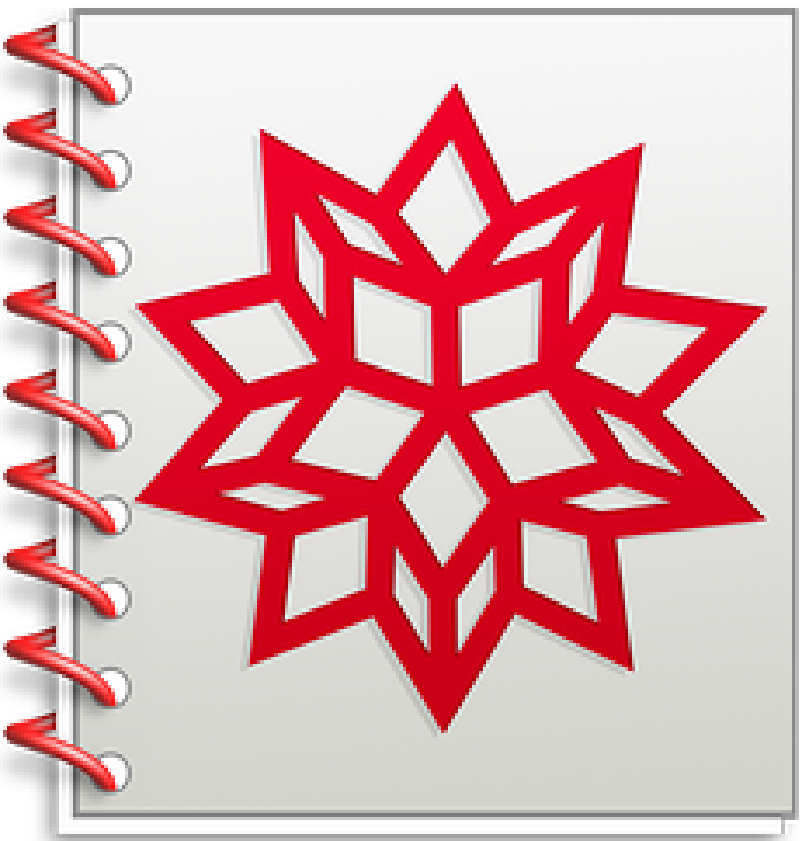}}%
  \raisebox{-2mm}{\href{\web/AwCalculus.m}{\tiny\tt AC}}%
  \ \imagetop{\includegraphics[scale=0.16]{ComputerTalk/AC#1.eps}}%
  \newline
}}
\def\FLQuote#1{{\label{C:#1}%
  \newline
  \null%
  \imagetop{\includegraphics[width=6mm]{figs/NotebookIcon.eps}}%
  \raisebox{-2mm}{\href{\web/FreeLie.m}{\tiny\tt FL}}%
  \ \imagetop{\includegraphics[scale=0.16]{ComputerTalk/FL#1.eps}}%
  \newline
}}

\def\Topology{%
  \parpic[l]{\begin{picture}(0,0)%
\includegraphics{figs/Topology.pstex}%
\end{picture}%
\setlength{\unitlength}{1316sp}%
\begingroup\makeatletter\ifx\SetFigFont\undefined%
\gdef\SetFigFont#1#2#3#4#5{%
  \reset@font\fontsize{#1}{#2pt}%
  \fontfamily{#3}\fontseries{#4}\fontshape{#5}%
  \selectfont}%
\fi\endgroup%
\begin{picture}(1052,1180)(0,-190)
\end{picture}%
  \hspace{-6pt}}%
  \noindent%
}
\def\Lie{%
  \parpic[l]{\begin{picture}(0,0)%
\includegraphics{figs/Lie.pstex}%
\end{picture}%
\setlength{\unitlength}{1421sp}%
\begingroup\makeatletter\ifx\SetFigFont\undefined%
\gdef\SetFigFont#1#2#3#4#5{%
  \reset@font\fontsize{#1}{#2pt}%
  \fontfamily{#3}\fontseries{#4}\fontshape{#5}%
  \selectfont}%
\fi\endgroup%
\begin{picture}(960,990)(46,0)
\end{picture}%
  \hspace{-6pt}}%
  \noindent%
}
\def\AT{%
  \parpic[l]{\begin{picture}(0,0)%
\includegraphics{figs/AT.pstex}%
\end{picture}%
\setlength{\unitlength}{1421sp}%
\begingroup\makeatletter\ifx\SetFigFont\undefined%
\gdef\SetFigFont#1#2#3#4#5{%
  \reset@font\fontsize{#1}{#2pt}%
  \fontfamily{#3}\fontseries{#4}\fontshape{#5}%
  \selectfont}%
\fi\endgroup%
\begin{picture}(1041,1009)(118,-19)
\end{picture}%
  \hspace{-8pt}}%
  \noindent%
}
\def\NotLie{%
  \parpic[l]{\begin{picture}(0,0)%
\includegraphics{figs/NotLie.pstex}%
\end{picture}%
\setlength{\unitlength}{1421sp}%
\begingroup\makeatletter\ifx\SetFigFont\undefined%
\gdef\SetFigFont#1#2#3#4#5{%
  \reset@font\fontsize{#1}{#2pt}%
  \fontfamily{#3}\fontseries{#4}\fontshape{#5}%
  \selectfont}%
\fi\endgroup%
\begin{picture}(960,1052)(46,-62)
\end{picture}%
  \hspace{-6pt}}%
  \noindent%
}
\def\TopLie{%
  \parpic[l]{\begin{picture}(0,0)%
\includegraphics{figs/Topology.pstex}%
\end{picture}%
\setlength{\unitlength}{1316sp}%
\begingroup\makeatletter\ifx\SetFigFont\undefined%
\gdef\SetFigFont#1#2#3#4#5{%
  \reset@font\fontsize{#1}{#2pt}%
  \fontfamily{#3}\fontseries{#4}\fontshape{#5}%
  \selectfont}%
\fi\endgroup%
\begin{picture}(1052,1180)(0,-190)
\end{picture}%
\,\begin{picture}(0,0)%
\includegraphics{figs/Lie.pstex}%
\end{picture}%
\setlength{\unitlength}{1421sp}%
\begingroup\makeatletter\ifx\SetFigFont\undefined%
\gdef\SetFigFont#1#2#3#4#5{%
  \reset@font\fontsize{#1}{#2pt}%
  \fontfamily{#3}\fontseries{#4}\fontshape{#5}%
  \selectfont}%
\fi\endgroup%
\begin{picture}(960,990)(46,0)
\end{picture}%
  \hspace{-6pt}}%
  \noindent%
}
\def\TopLieSmall{%
  \parpic[l]{\scalebox{0.6}{%
    \begin{picture}(0,0)%
\includegraphics{figs/Topology.pstex}%
\end{picture}%
\setlength{\unitlength}{1316sp}%
\begingroup\makeatletter\ifx\SetFigFont\undefined%
\gdef\SetFigFont#1#2#3#4#5{%
  \reset@font\fontsize{#1}{#2pt}%
  \fontfamily{#3}\fontseries{#4}\fontshape{#5}%
  \selectfont}%
\fi\endgroup%
\begin{picture}(1052,1180)(0,-190)
\end{picture}%
\,\begin{picture}(0,0)%
\includegraphics{figs/Lie.pstex}%
\end{picture}%
\setlength{\unitlength}{1421sp}%
\begingroup\makeatletter\ifx\SetFigFont\undefined%
\gdef\SetFigFont#1#2#3#4#5{%
  \reset@font\fontsize{#1}{#2pt}%
  \fontfamily{#3}\fontseries{#4}\fontshape{#5}%
  \selectfont}%
\fi\endgroup%
\begin{picture}(960,990)(46,0)
\end{picture}%
    \hspace{-6pt}}}%
  \noindent%
}
\def\TopLieAT{%
  \parpic[l]{\begin{picture}(0,0)%
\includegraphics{figs/Topology.pstex}%
\end{picture}%
\setlength{\unitlength}{1316sp}%
\begingroup\makeatletter\ifx\SetFigFont\undefined%
\gdef\SetFigFont#1#2#3#4#5{%
  \reset@font\fontsize{#1}{#2pt}%
  \fontfamily{#3}\fontseries{#4}\fontshape{#5}%
  \selectfont}%
\fi\endgroup%
\begin{picture}(1052,1180)(0,-190)
\end{picture}%
\,\begin{picture}(0,0)%
\includegraphics{figs/Lie.pstex}%
\end{picture}%
\setlength{\unitlength}{1421sp}%
\begingroup\makeatletter\ifx\SetFigFont\undefined%
\gdef\SetFigFont#1#2#3#4#5{%
  \reset@font\fontsize{#1}{#2pt}%
  \fontfamily{#3}\fontseries{#4}\fontshape{#5}%
  \selectfont}%
\fi\endgroup%
\begin{picture}(960,990)(46,0)
\end{picture}%
  \,\,\begin{picture}(0,0)%
\includegraphics{figs/AT.pstex}%
\end{picture}%
\setlength{\unitlength}{1421sp}%
\begingroup\makeatletter\ifx\SetFigFont\undefined%
\gdef\SetFigFont#1#2#3#4#5{%
  \reset@font\fontsize{#1}{#2pt}%
  \fontfamily{#3}\fontseries{#4}\fontshape{#5}%
  \selectfont}%
\fi\endgroup%
\begin{picture}(1041,1009)(118,-19)
\end{picture}%

  \hspace{-6pt}}%
  \noindent%
}

\def\ds{\displaystyle}
\def\inlineeq{\refstepcounter{equation}{\rm(\theequation)}}

\def\ob#1{\overbracket[0.5pt][1pt]{#1}}

\makeatletter
\def\smallunderbrace#1{\mathop{\vtop{\m@th\ialign{##\crcr
   $\hfil\displaystyle{#1}\hfil$\crcr
   \noalign{\kern3\p@\nointerlineskip}%
   \tiny\upbracefill\crcr\noalign{\kern3\p@}}}}\limits}
\makeatother

\def\glosm#1#2{{\label{g:#1}\yellowm{#2}}}
\def\glosi#1#2#3{{\item[{#2}] #3~\hfill\pageref{g:#1}}}

\newcounter{tunnel}

\begin{document}
\newdimen\captionwidth\captionwidth=\hsize
\setcounter{secnumdepth}{4}

\title{{Finite Type Invariants of w-Knotted Objects IV: Some Computations}}

\author{Dror~Bar-Natan}
\address{
  Department of Mathematics\\
  University of Toronto\\
  Toronto Ontario M5S 2E4\\
  Canada
}
\email{drorbn@math.toronto.edu}
\urladdr{http://www.math.toronto.edu/~drorbn}

\date{First edition Nov.\ 15, 2015, this edition Nov.~17,~2015. Electronic version
and related files at~\cite{WKO4}, \url{\web}. The \arXiv{????.????}
edition may be older}

\subjclass[2010]{57M25}
\keywords{
  w-knots,
  w-tangles,
  Kashiwara-Vergne,
  associators,
  double tree,
  Mathematica,
  free Lie algebras%
}

\thanks{This work was partially supported by NSERC grant RGPIN 262178.}

\begin{abstract}
  In the previous three papers in this series, \cite{WKO1}--\cite{WKO3},
Z.~Dancso and I studied a certain theory of ``homomorphic expansions'' of
``w-knotted objects'', a certain class of knotted objects in 4-dimensional
space. When all layers of interpretation are stripped off, what remains
is a study of a certain number of equations written in a family of spaces
${\mathcal A}^w$, closely related to degree-completed free Lie algebras
and to degree-completed spaces of cyclic words.

The purpose of this paper is to introduce mathematical and computational
tools that enable explicit computations (up to a certain degree)
in these ${\mathcal A}^w$ spaces and to use these tools to solve the
said equations and verify some properties of their solutions, and as
a consequence, to carry out the computation (up to a certain degree)
of certain knot-theoretic invariants discussed in \cite{WKO1}--\cite{WKO3}
and in my related paper \cite{KBH}.

\end{abstract}

\maketitle

\setcounter{tocdepth}{3}
\tableofcontents

\draftcut\section{Introduction} Within the previous three
papers in this series~\cite{WKO1}--\cite{WKO3}\footnoteT{Also within
my~\cite{KBH}, and within papers by Alekseev, Enriquez,
and Torossian~\cite{AT,
AlekseevEnriquezTorossian:ExplicitSolutions}, and within
Kashiwara's and Vergne's~\cite{KashiwaraVergne:Conjecture},
and also within many older papers about Drinfel'd associators
(e.g. Drinfel'd's~\cite{Drinfeld:QuasiHopf, Drinfeld:GalQQ} and my
\cite{Bar-Natan:NAT}.} a number of intricate equations written in various
graded spaces related to free Lie algebras and to spaces of cyclic words
were examined in detail, for good reasons that were explained there and
elsewhere. The purpose of this paper is to introduce mathematical tools
(on the upper parts of pages) and computational tools (on the lower
parts of pages, below the bold dividing lines\footnotemarkC) that allow for
the explicit solution of these equations, at least up to a certain degree.

\footnotetextC{
If you are not interested in the actual computations, it is safe to
ignore the parts of pages below the bold dividing lines and restrict to
``strict'' mathematics, which is always above these lines. {\bf Alert.} If
you are interested in the computations, note that the computational
footnotes are sometimes long and crawl across page boundaries. This
footnote is the first example.

The programs described in this paper were written in
Mathematica~\cite{Wolfram:Mathematica} and are available at~\cite{WKO4}.
Before starting with any computations, download the packages
\href{\web/FreeLie.m}{\tt FreeLie.m} and
\href{\web/AwCalculus.m}{\tt AwCalculus.m} and type within Mathematica:
(the interactive Mathematica session demonstrated in this paper is
available as \citeweb{WKO4Session.nb})

\dialoginclude{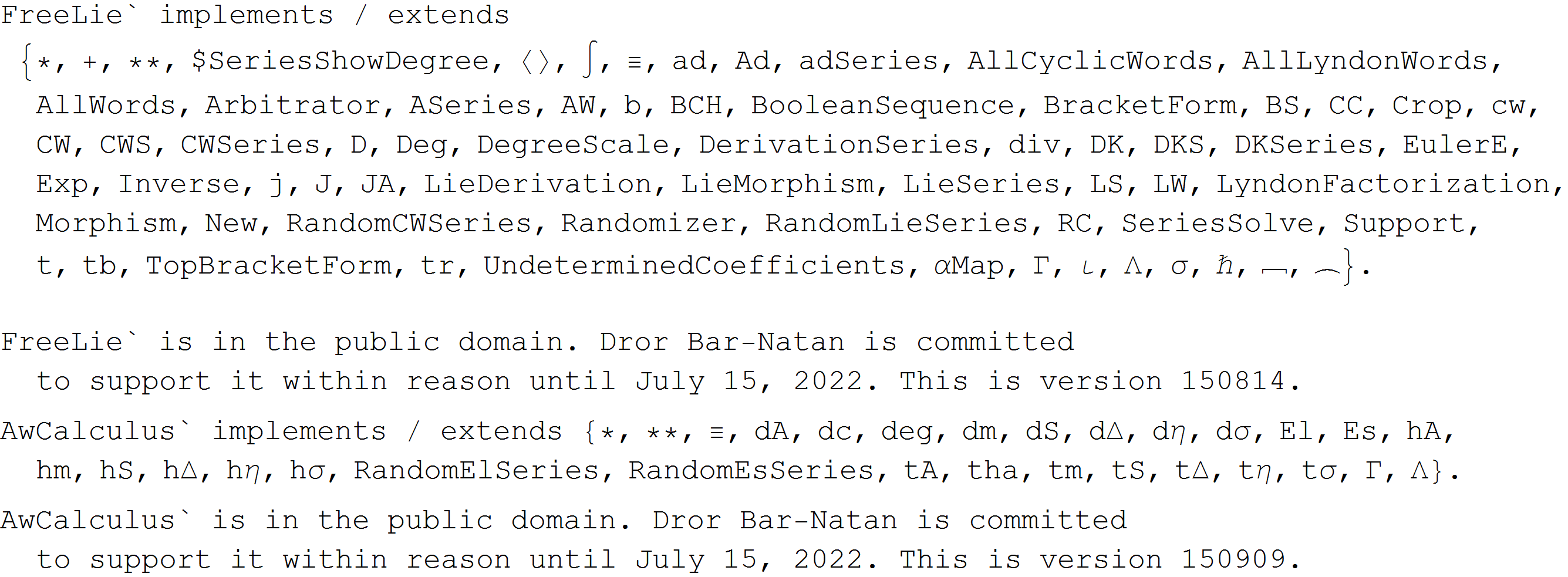} \rule{0in}{6pt}
% For unknown reasons the \rule above prevents a page break here.

The last input (``human'') line above declares that by default we wish the
computer to print series within graded spaces (such as free Lie algebras)
to degree 4. Note that we \lightred{highlight in pink} input lines that
affect later computations.
}

The equations we have in mind arise in other papers and appear throughout
this paper. Yet to help our impatient readers orient themselves,
Figure~\ref{fig:flash} contains a ``flash summary'' of the most important
equations and their topological and algebraic significance.

\begin{figure}
\[
  \def\YB{$R^{12}R^{13}R^{23}=R^{23}R^{13}R^{12}$}
  \def\RFour{$R^{23}R^{13}V=R^{12,3}$}
  \def\RFourAT{{\cite{AT}:
    $F(x+y)=\log e^xe^y$
  }}
  \def\UniCap{$VV^\ast=1;\quad VC^{12}=C^1C^2$}
  \def\UniCapAT{{\cite{AT}:
    $j(F)\in\im(\tilde{\delta})$
  }}
  \def\Twist{$\Theta=V^{-1}RV^{21}$}
  \def\Pentagon{$\Phi\Phi^{1,23,4}\Phi^{234}=\Phi^{12,3,4}\Phi^{1,2,34}$}
  \def\Buckle{$(\Phi^{-1})^{13,2,4}\Phi^{132}R^{23}\Phi^{-1}\Phi^{12,3,4}$}
  \input{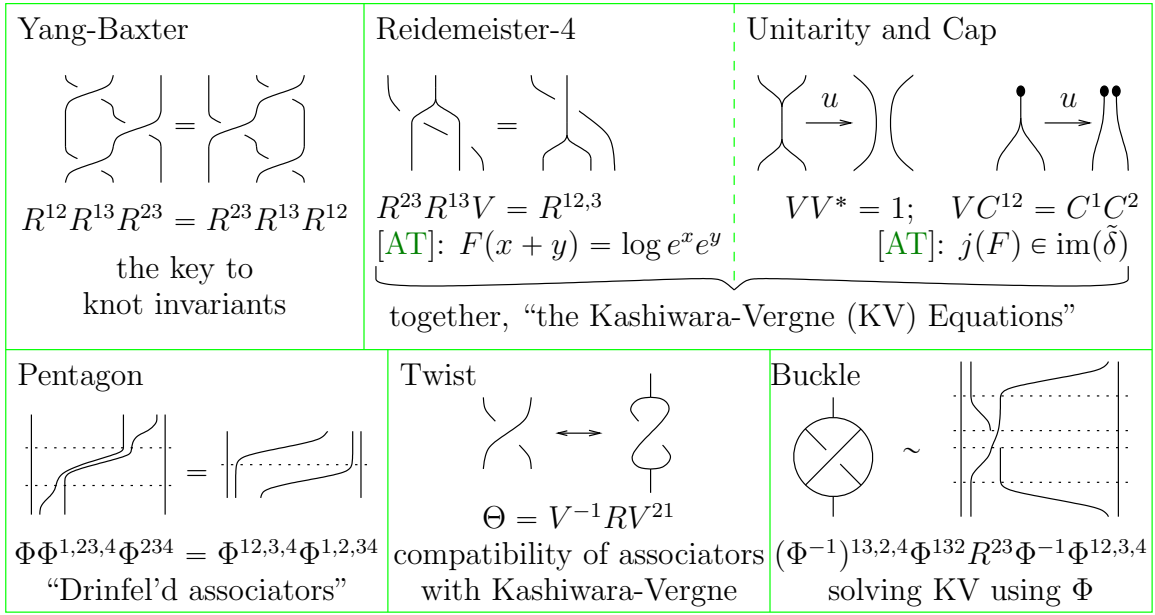}
\]
\caption{The most important equations.} \label{fig:flash}
\end{figure}

Why bother? What do limited explicit computations add, given that these
intricate equations are known to be soluble, and given that the conceptual
framework within which these equations make sense is reasonably well
understood~\cite{WKO1}--\cite{WKO3}? My answers are three:

\begin{enumerate}[leftmargin=*,labelindent=0pt]

\item Personally, my belief in what I can't compute decays quite rapidly
as a function of the complexity involved. Even if the overall picture is
clear, the details will surely go wrong, and sooner or later, something
bigger than a detail will go wrong. Even a limited computation may
serve as a wonderful sanity check. In situations such as ours, where
many signs and conventions need to be decided and may well go wrong,
even a low-degree computation increases my personal confidence level by a
great degree. Given computations that work to degree 6 (say), it is hard
to imagine that a detail was missed or that conventions were established
in an inconsistent manner. In fact, if the computer programs are clear
enough and are shown to work, these programs become the authoritative
declarations of the details and conventions.

\item The computational tools introduced here may well be useful in other
contexts where free Lie algebras and/or cyclic words arise.

\item The papers~\cite{WKO1,WKO2} (and likewise~\cite{KBH}) are about
equations, but even more so, about the construction of certain knot and
tangle invariants. With the tools presented here, the invariants of
arbitrary knotted objects of the types studied in~\cite{WKO1,WKO2,KBH}
may be computed.

\end{enumerate}

The equations of~\cite{WKO1}--\cite{WKO3} always involve group-like,
or ``exponential'' elements, and are written in some spaces of ``arrow
diagrams'' that go under the umbrella name $\calA^w$. Hence a crucial
first step is to find convenient presentations for the group-like
elements $\calA^w_{\exp}$ in $\calA^w$-spaces. It turns out that there
are (at least) two such presentations, each with its own advantages and
disadvantages. Hence in Section~\ref{sec:Aw} we recall $\calA^w$ briefly
(\ref{subsec:Aw}), then discuss some free-Lie-algebra preliminaries
(\ref{subsec:FL}), then describe the Alekseev-Torossian-\cite{AT}-inspired
``lower-interlaced'' presentation $E_l$ of $\calA^w_{\exp}$
(\ref{subsec:AT}), then describe the \cite{KBH}-inspired ``factored''
presentation $E_f$ of $\calA^w_{\exp}$ and its stronger precursor
``split'' presentation $E_s$ (\ref{subsec:Ef}), and then describe how
to convert between the two primary presentations (\ref{subsec:Conversion}).

We then present our computations in Section~\ref{sec:Computations}: Some
knot and tangle invariants are computed in
Section~\ref{subsec:TangleInvariants} and solutions of the Kashiwara-Vergne
(KV) equations in Section~\ref{subsec:SolKV}. In Section~\ref{subsec:tau} we
discuss the ``Twist Equation'' and compute dimensions of spaces of
solutions of the linearized KV equations, with and without the Twist
Equation. In Section~\ref{subsec:Associators} we compute a Drinfel'd
associator, in Section~\ref{subsec:wAssociators} we compute associators in
$\calA^w$ starting from a solution of the KV equations, and in
Section~\ref{subsec:KVandAssoc} we show how to compute a solution of KV
from a Drinfel'd associator. The last computational result is in
Section~\ref{subsec:Trivolution}, where we give computational support to
the existence of an action of the symmetric group $S_4$ on the set of
solutions of the Kashiwara-Vergne Equations.

\begin{figure}
\[ \xymatrix@C=0.8in@R=0.8in{
  \left\{\TW_l(S)\right\}
    \ar@`{p+(16,16),p+(-16,16)}_{\ast,{\gray dm},\ldots}
    \ar@/^/[r]^\Gamma
    \ar[d]^{E_l}
  & \left\{\TW_s(S)\right\}
    \ar@`{p+(16,16),p+(-16,16)}_{\ast,\#,dm,\ldots}
    \ar@/^/[l]^{\Lambda=\Gamma^{-1}}
    \ar@{^{(}->}[r]
    \ar[d]^{E_s}
    \ar[dl]^(0.4){E_f}
  & \left\{\TW_s(H;T)\right\}
    \ar@`{p+(16,16),p+(-16,16)}_{\ast,\#,dm,hm,tm,tha,\ldots}
    \ar[d]^{E_s}
  \\
  \left\{\calA^w_{\exp}(S)\right\}
    \ar@`{p+(-16,-16),p+(16,-16)}_{\ast,dm,\ldots}
    \ar@/^/[r]^{\delta}
  & \left\{\calA^w_{\exp}(S;S)\right\}
    \ar@`{p+(-16,-16),p+(16,-16)}_{\ast,\#,dm,\ldots}
    \ar@/^/[l]^{\delta^{-1}}
    \ar@{^{(}->}[r]
  & \left\{\calA^w_{\exp}(H;T)\right\}
    \ar@`{p+(-16,-16),p+(16,-16)}_{\ast,\#,dm,hm,tm,tha,\ldots}
} \]
\caption{The main spaces and maps appearing in this paper.}
\label{fig:diagram}
\end{figure}
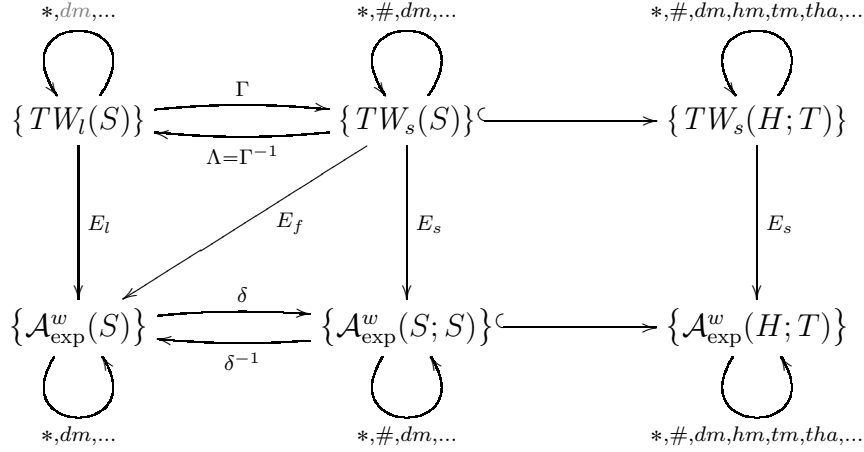

We conclude this introduction with a description of the commutative
diagram in Figure~\ref{fig:diagram} which displays the main
spaces and maps appearing in this paper, as described in detail in
Section~\ref{sec:Aw}. The bottom row of this diagram consists of spaces of
``group-like'' elements inside spaces $\calA^w$ of ``arrow diagrams'';
these are the spaces that have direct knot-theoretic significance. The
top row are spaces of ``trees and wheels'', or more precisely, various
elements of free Lie algebras and various cyclic words. They are the
spaces of ``primitives'' corresponding to the group-like elements at the
bottom, via various ``exponentiation'' maps $E_l$, $E_f$, and $E_s$. In
this paper we study\footnoteC{Or ``implement'', in computer-speak.} the
spaces on the bottom row by means of their presentations by elements in
the top row.

The collection $\left\{\calA^w_{\exp}(S)\right\}$ of spaces
we primarily wish to study (and in which most of the equations of
Figure~\ref{fig:flash} are written) appears on the bottom left. There
are many binary and unary operations acting on the spaces within
$\left\{\calA^w_{\exp}(S)\right\}$ as indicated by the circular
self-arrow appearing there, which is labelled with the most important
of these operations, the binary $\ast$ and the unary $dm$. On the top
left of the diagram are the spaces $\left\{\TW_l(S)\right\}$ of trees
and wheels which represent $\left\{\calA^w_{\exp}(S)\right\}$
via the $E_l$ presentation. The same collection of operations acts here
too, though notice that the operation $dm$ is grayed-out, because we
have no direct implementation for it in $\TW_l$ language.

On the bottom right is a bigger collection of spaces,
$\left\{\calA^w_{\exp}(H;T)\right\}$, which contains as a subset the
collection $\left\{\calA^w_{\exp}(S;S)\right\}$ (bottom middle), which
is isomorphic in a non-trivial manner (via $\delta$ and $\delta^{-1}$) to
$\left\{\calA^w_{\exp}(S)\right\}$. A richer collection of
operations act on $\left\{\calA^w_{\exp}(H;T)\right\}$, and the most
important of those are $\ast$, $\#$, $dm$, $hm$, $tm$, and $tha$. 

On the top right is the collection $\left\{\TW_s(H;T)\right\}$ of spaces
of trees and wheels which represent $\left\{\calA^w_{\exp}(H;T)\right\}$
via the $E_s$ presentation. When restricted to $H=T=S$,
this is the collection $\left\{\TW_s(S)\right\}$ representing
$\left\{\calA^w_{\exp}(S;S)\right\}$, and representing our primary
interest $\left\{\calA^w_{\exp}(S)\right\}$ via $E_f$, the
composition of $E_s$ with $\delta^{-1}$.

Note that $\TW_l$ and $\TW_s$ are set-theoretically the same spaces
of trees and wheels. Yet the operations $\ast$, $dm$, etc.\ act on
them in a different manner, and hence they deserve to have different
names\footnoteT{Much as in group theory, a direct product $N\times H$
is set-theoretically the same as a semi-direct product $N\rtimes H$, yet
it is wrong to refer to them by the same name.}. Note also that $\TW_l$
and $\TW_s$ are in fact isomorphic via structure-preserving isomorphisms
(denoted $\Gamma$ and $\Lambda=\Gamma^{-1}$). These isomorphisms are
compositions of the relatively simple-minded $\delta$ and $\delta^{-1}$
with the more complex ``exponentiations'' $E_l$ and $E_s$ and their
inverses. Thus the isomorphisms $\Gamma$ and $\Lambda$ are non-linear
and quite complicated.

  \parpic[l]{\begin{picture}(0,0)%
\includegraphics{figs/Topology.pstex}%
\end{picture}%
%
%  pstex_opts: -m 0.3333 
%
\setlength{\unitlength}{1316sp}%
\begingroup\makeatletter\ifx\SetFigFont\undefined%
\gdef\SetFigFont#1#2#3#4#5{%
  \reset@font\fontsize{#1}{#2pt}%
  \fontfamily{#3}\fontseries{#4}\fontshape{#5}%
  \selectfont}%
\fi\endgroup%
\begin{picture}(1052,1180)(0,-190)
\end{picture}%
\,\begin{picture}(0,0)%
\includegraphics{figs/Lie.pstex}%
\end{picture}%
%
%  pstex_opts: -m 0.36 
%
\setlength{\unitlength}{1421sp}%
\begingroup\makeatletter\ifx\SetFigFont\undefined%
\gdef\SetFigFont#1#2#3#4#5{%
  \reset@font\fontsize{#1}{#2pt}%
  \fontfamily{#3}\fontseries{#4}\fontshape{#5}%
  \selectfont}%
\fi\endgroup%
\begin{picture}(960,990)(46,0)
\end{picture}%
  \,\,\begin{picture}(0,0)%
\includegraphics{figs/AT.pstex}%
\end{picture}%
%
%  pstex_opts: -m 0.36 
%
\setlength{\unitlength}{1421sp}%
\begingroup\makeatletter\ifx\SetFigFont\undefined%
\gdef\SetFigFont#1#2#3#4#5{%
  \reset@font\fontsize{#1}{#2pt}%
  \fontfamily{#3}\fontseries{#4}\fontshape{#5}%
  \selectfont}%
\fi\endgroup%
\begin{picture}(1041,1009)(118,-19)
\end{picture}%

  \hspace{-6pt}}%
  \noindent%
 We will occasionally comment on the relationship between
the constructs appearing in this papers and three related topics:
``topology'', or more precisely certain aspects of the theory of
2-knots, ``Lie theory'', or more precisely certain classes of formulas
that make sense in arbitrary finite-dimensional Lie algebras, and
``Alekseev-Torossian'', or more precisely, issues related to the
paper~\cite{AT}. These comments will in general be incomplete and should
be regarded as ``hints for the already initiated'' --- people familiar
with the papers \cite{WKO1, WKO2, WKO3, KBH, AT} will hopefully find
that these comments help to put the current paper in context. These
comments will always be labelled by one (or more) of the three logos at
the head of this paragraph, which correspond, in order, to ``topology'',
``Lie theory'', and ``Alekseev-Torossian''.

  \parpic[l]{\begin{picture}(0,0)%
\includegraphics{figs/Topology.pstex}%
\end{picture}%
%
%  pstex_opts: -m 0.3333 
%
\setlength{\unitlength}{1316sp}%
\begingroup\makeatletter\ifx\SetFigFont\undefined%
\gdef\SetFigFont#1#2#3#4#5{%
  \reset@font\fontsize{#1}{#2pt}%
  \fontfamily{#3}\fontseries{#4}\fontshape{#5}%
  \selectfont}%
\fi\endgroup%
\begin{picture}(1052,1180)(0,-190)
\end{picture}%
  \hspace{-6pt}}%
  \noindent%
 Within the study of simply-knotted (ribbon) 2-knots, or more
precisely w-knotted-objects as they
appear in \cite{WKO1, WKO2, KBH}, the rows of
Figure~\ref{fig:diagram} correspond to the extra row
\[ \xymatrix@C=0.8in{
  \left\{\calK^w(S)\right\}
    \ar@`{p+(-16,-16),p+(16,-16)}_{\ast,dm,\ldots}
    \ar@/^/[r]^{\delta}
  & \left\{\calK^w(S;S)\right\}
    \ar@`{p+(-16,-16),p+(16,-16)}_{\ast,\#,dm,\ldots}
    \ar@/^/[l]^{\delta^{-1}}
    \ar@{^{(}->}[r]
  & \left\{\calK^w(H;T)\right\},
    \ar@`{p+(-16,-16),p+(16,-16)}_{\ast,\#,dm,hm,tm,tha,\ldots}
} \]
via the ``associated graded'' procedure
described in~\cite{WKO2}. Here $\calK^w(S)$ is the set of
$S$-labelled w-tangles~\cite{WKO2}, $\calK^w(H;T)$ is the set of w-knotted
$H$-labelled hoops and $T$-labelled balloons~\cite{KBH}, $\calK^w(S;S)$ is
the same but with $H=T=S$, and $\delta$ is the same as in~\cite{KBH}. This
correspondence is further recalled throughout the rest of this paper.

  \parpic[l]{\begin{picture}(0,0)%
\includegraphics{figs/Lie.pstex}%
\end{picture}%
%
%  pstex_opts: -m 0.36 
%
\setlength{\unitlength}{1421sp}%
\begingroup\makeatletter\ifx\SetFigFont\undefined%
\gdef\SetFigFont#1#2#3#4#5{%
  \reset@font\fontsize{#1}{#2pt}%
  \fontfamily{#3}\fontseries{#4}\fontshape{#5}%
  \selectfont}%
\fi\endgroup%
\begin{picture}(960,990)(46,0)
\end{picture}%
  \hspace{-6pt}}%
  \noindent%
 The corresponding Lie-theoretic spaces (compare
\cite[Section~\ref{1-subsec:LieAlgebras}]{WKO1}) are
\[ \xymatrix@C=0.8in{
  \left\{\calU(I\frakg)^{\otimes S}\right\}
    \ar@`{p+(-16,-16),p+(16,-16)}_{\ast,dm,\ldots}
    \ar@/^/[r]^-{\delta}
  & \left\{\calU(\frakg)^{\otimes S}\otimes\calS(\frakg^\ast)^{\otimes S}\right\}
    \ar@`{p+(-16,-16),p+(16,-16)}_{\ast,\#,dm,\ldots}
    \ar@/^/[l]^-{\delta^{-1}}
    \ar@{^{(}->}[r]
  & \left\{\calU(\frakg)^{\otimes H}\otimes\calS(\frakg^\ast)^{\otimes T}\right\}.
    \ar@`{p+(-16,-16),p+(16,-16)}_{\ast,\#,dm,hm,tm,tha,\ldots}
} \]
This correspondence is further recalled throughout the rest of this paper.

  \parpic[l]{\begin{picture}(0,0)%
\includegraphics{figs/AT.pstex}%
\end{picture}%
%
%  pstex_opts: -m 0.36 
%
\setlength{\unitlength}{1421sp}%
\begingroup\makeatletter\ifx\SetFigFont\undefined%
\gdef\SetFigFont#1#2#3#4#5{%
  \reset@font\fontsize{#1}{#2pt}%
  \fontfamily{#3}\fontseries{#4}\fontshape{#5}%
  \selectfont}%
\fi\endgroup%
\begin{picture}(1041,1009)(118,-19)
\end{picture}%
  \hspace{-8pt}}%
  \noindent%
 In~\cite{AT} there is no good counterparts for last two columns of our
diagram. The counterpart of the first (and primary) column is a mixture
$\hat{\calU}((\fraka_n\oplus\attder_n)\ltimes\attr_n)$ containing the most
important spaces occurring in~\cite{AT}. More in the next section.

\subsection{Acknowledgement} This paper was written almost entirely with
Z.~Dancso in the room (physically or virtually via Skype), working on
various parts of our joint series \cite{WKO1}--\cite{WKO3}. Hence her
indirect contribution to it, in a huge number of routine consultations,
should be acknowledged in capitals: THANKS, ZSUZSI. I would like
to further thank A.~Alekseev and S.~Morgan for their comments and
suggestions.

\draftcut
\needspace{3\baselineskip}
\section{Group-like elements in $\calA^w$} \label{sec:Aw}

\subsection{A brief review of $\calA^w$} \label{subsec:Aw} Let
$\glosm{S}{S} = \{\glosm{a}{a}_1,a_2,\dots\}$\footnoteT{Yellow highlighting
corresponds to the \yellowt{glossary}, Section~\ref{sec:glossary}.}
be a finite set of ``strand labels''. The space
$\glosm{calAw}{\calA^w}(S)$ is the completed
graded vector space\footnoteT{For simplicity we always work over $\bbQ$.}
of diagrams made of (vertical) ``strands'' labelled by the elements of $S$,
and ``arrows'' as summarized by the following picture:
\[ 
  \def\comments{{\parbox{2.28in}{\scriptsize
    \begin{myitemize}
      \item Diagrams are connected.
      \item Vertices are 2-in 1-out.
      \item Vertices are oriented.
      \item Degree is half the number of trivalent vertices.
      \item The ``skeleton'' is a union of vertical strands labelled by the
        elements of $S$.
    \end{myitemize}
  }}}
  \label{g:wRels}
  \input{figs/AwSummary.pstex_t}
\]

When $S=\{1,2,\ldots,n\}$ we abbreviate
$\calA^w(\glosm{uparrow}{\uparrow_n})\coloneqq\calA^w(S)$.

  \parpic[l]{\begin{picture}(0,0)%
\includegraphics{figs/Topology.pstex}%
\end{picture}%
%
%  pstex_opts: -m 0.3333 
%
\setlength{\unitlength}{1316sp}%
\begingroup\makeatletter\ifx\SetFigFont\undefined%
\gdef\SetFigFont#1#2#3#4#5{%
  \reset@font\fontsize{#1}{#2pt}%
  \fontfamily{#3}\fontseries{#4}\fontshape{#5}%
  \selectfont}%
\fi\endgroup%
\begin{picture}(1052,1180)(0,-190)
\end{picture}%
\,\begin{picture}(0,0)%
\includegraphics{figs/Lie.pstex}%
\end{picture}%
%
%  pstex_opts: -m 0.36 
%
\setlength{\unitlength}{1421sp}%
\begingroup\makeatletter\ifx\SetFigFont\undefined%
\gdef\SetFigFont#1#2#3#4#5{%
  \reset@font\fontsize{#1}{#2pt}%
  \fontfamily{#3}\fontseries{#4}\fontshape{#5}%
  \selectfont}%
\fi\endgroup%
\begin{picture}(960,990)(46,0)
\end{picture}%
  \,\,\begin{picture}(0,0)%
\includegraphics{figs/AT.pstex}%
\end{picture}%
%
%  pstex_opts: -m 0.36 
%
\setlength{\unitlength}{1421sp}%
\begingroup\makeatletter\ifx\SetFigFont\undefined%
\gdef\SetFigFont#1#2#3#4#5{%
  \reset@font\fontsize{#1}{#2pt}%
  \fontfamily{#3}\fontseries{#4}\fontshape{#5}%
  \selectfont}%
\fi\endgroup%
\begin{picture}(1041,1009)(118,-19)
\end{picture}%

  \hspace{-6pt}}%
  \noindent%
 In topology, elements of $\calA^w(S)$ are closely related
to (finite type invariants of) simply knotted 2-dimensional tubes in
$\bbR^4$ (\cite{WKO1}--\cite{WKO3}, \cite{KBH}). In Lie theory, they
represent ``universal'' $\glosm{frakg}{\frakg}$-invariant tensors in
$\glosm{calU}{\calU}(I\frakg)^{\otimes S}$, where $\glosm{Ifrakg}{I\frakg}
\coloneqq \frakg\ltimes\frakg^\ast$\footnoteT{In earlier papers we have
used the order $I\frakg=\frakg^\ast\rtimes\frakg$.} and $\frakg$ is some
finite dimensional Lie algebra (\cite{WKO1}--\cite{WKO3}).  Readers of
Alekseev and Torossian~\cite{AT} may care about $\calA^w$ because
using notation from~\cite{AT}, $\calA^w(\uparrow_n)$ is the completed
universal enveloping algebra of $(\fraka_n\oplus\attder_n)\ltimes\attr_n$
(see~\cite{WKO2}), and hence much of the \cite{AT} story can be told
within $\calA^w$. Several significant Lie theoretic problems (e.g., the
Kashiwara-Vergne problem,~\cite{KashiwaraVergne:Conjecture, AT, WKO2})
can be interpreted as problems about $\calA^w(\uparrow_n)$.

\begin{comment} \label{com:SortedForm} Using the $\aSTU_2$ relation
one may sort the skeleton vertices in every $\glosm{D}{D}\in\calA^w(S)$
so that along every skeleton component all arrow heads appear ahead
of all arrow tails, and by a diagrammatic analogue of the PBW theorem
(compare~\cite[Theorem~8]{Bar-Natan:OnVassiliev}), this sorted form is
unique modulo $\aSTU_1$, $TC$, $\aAS$ and $\aIHX$ relations.
\end{comment}

\begin{definition} \label{def:Operations}
A number of operations are defined on elements of the $\calA^w(S)$
spaces:

\vskip \topsep
\begin{enumerate}[leftmargin=*,labelindent=0pt,topsep=0pt]

\parpic[r]{\input{figs/Union.pstex_t}}
\item If $S_1$ and $S_2$ are disjoint, then given
$D_1\in\calA^w({S_1})$ and $D_2\in\calA^w({S_2})$,
their union $D_1D_2=D_1\glosm{sqcup}{\sqcup} D_2\in\calA^w(S)$, where
$S=S_1\sqcup S_2$,
is obtained by placing them side by side as illustrated on the right.

\vskip 1mm %
  \parpic[l]{\begin{picture}(0,0)%
\includegraphics{figs/Topology.pstex}%
\end{picture}%
%
%  pstex_opts: -m 0.3333 
%
\setlength{\unitlength}{1316sp}%
\begingroup\makeatletter\ifx\SetFigFont\undefined%
\gdef\SetFigFont#1#2#3#4#5{%
  \reset@font\fontsize{#1}{#2pt}%
  \fontfamily{#3}\fontseries{#4}\fontshape{#5}%
  \selectfont}%
\fi\endgroup%
\begin{picture}(1052,1180)(0,-190)
\end{picture}%
\,\begin{picture}(0,0)%
\includegraphics{figs/Lie.pstex}%
\end{picture}%
%
%  pstex_opts: -m 0.36 
%
\setlength{\unitlength}{1421sp}%
\begingroup\makeatletter\ifx\SetFigFont\undefined%
\gdef\SetFigFont#1#2#3#4#5{%
  \reset@font\fontsize{#1}{#2pt}%
  \fontfamily{#3}\fontseries{#4}\fontshape{#5}%
  \selectfont}%
\fi\endgroup%
\begin{picture}(960,990)(46,0)
\end{picture}%
  \hspace{-6pt}}%
  \noindent%
 In topology, $\sqcup$ corresponds to the disjoint
union of 2-tangles\footnoteT{To be clear, the ``$2$'' in ``2-tangles''
refers to the dimension of the things being knotted, and not to
the number of components.}.  In Lie theory, it corresponds to the
map $\calU(I\frakg)^{\otimes S_1}\otimes\calU(I\frakg)^{\otimes
S_2}\to\calU(I\frakg)^{\otimes(S_1\sqcup S_2)}$.

\vskip 2mm
\item Given $D_1\in\calA^w(S)$ and
$D_2\in\calA^w(S)$, their product
$D_1\glosm{ast}{\ast} D_2\in\calA^w(S)$ is obtained by ``stacking $D_2$
on top of $D_1$'':
\begin{equation} \label{eq:TubeProduct}
  (D_1,D_2)=
  \begin{array}{c}\input{figs/Stacking.pstex_t}\end{array}
  =D_1\ast D_2.
\end{equation}

  \parpic[l]{\begin{picture}(0,0)%
\includegraphics{figs/Topology.pstex}%
\end{picture}%
%
%  pstex_opts: -m 0.3333 
%
\setlength{\unitlength}{1316sp}%
\begingroup\makeatletter\ifx\SetFigFont\undefined%
\gdef\SetFigFont#1#2#3#4#5{%
  \reset@font\fontsize{#1}{#2pt}%
  \fontfamily{#3}\fontseries{#4}\fontshape{#5}%
  \selectfont}%
\fi\endgroup%
\begin{picture}(1052,1180)(0,-190)
\end{picture}%
\,\begin{picture}(0,0)%
\includegraphics{figs/Lie.pstex}%
\end{picture}%
%
%  pstex_opts: -m 0.36 
%
\setlength{\unitlength}{1421sp}%
\begingroup\makeatletter\ifx\SetFigFont\undefined%
\gdef\SetFigFont#1#2#3#4#5{%
  \reset@font\fontsize{#1}{#2pt}%
  \fontfamily{#3}\fontseries{#4}\fontshape{#5}%
  \selectfont}%
\fi\endgroup%
\begin{picture}(960,990)(46,0)
\end{picture}%
  \,\,\begin{picture}(0,0)%
\includegraphics{figs/AT.pstex}%
\end{picture}%
%
%  pstex_opts: -m 0.36 
%
\setlength{\unitlength}{1421sp}%
\begingroup\makeatletter\ifx\SetFigFont\undefined%
\gdef\SetFigFont#1#2#3#4#5{%
  \reset@font\fontsize{#1}{#2pt}%
  \fontfamily{#3}\fontseries{#4}\fontshape{#5}%
  \selectfont}%
\fi\endgroup%
\begin{picture}(1041,1009)(118,-19)
\end{picture}%

  \hspace{-6pt}}%
  \noindent%
 In topology, the stacking product corresponds to the
concatenation operation on knotted tubes, akin to the standard
stacking product of tangles. In Lie theory, it comes from the
algebra structure of $\calU(I\frakg)^{\otimes S}$. In~\cite{AT},
it is the product of the completed universal enveloping algebra
$\hat{\calU}((\fraka_n\oplus\attder_n)\ltimes\attr_n)$.

\setcounter{tunnel}{\theenumi}
\end{enumerate}

\noindent Note that below and throughout this paper we use
$\glosm{act}{\act}$ for postfix operator application and for ``composition
done right''. Meaning that $x\act f$ is equivalent to $f(x)$ and $f\act g$
is $g\circ f$ is ``do $f$ then do $g$''.

\begin{enumerate}[leftmargin=*,labelindent=0pt]\setcounter{enumi}{\thetunnel}

\vskip 2mm
\parpic[r]{
  \def\eps{{$\act(d\eta^1,d\eta^2,d\eta^3)$}}
  \input{figs/depsilon.pstex_t}
}
\item Given $D\in\calA^w(S)$ and $a\in S$, $D\act
\glosm{deta}{d\eta}^a$ is the result of deleting strand $a$ from $D$
and mapping it to $0$ if any arrow connects to $a$, as illustrated on
the right.

\vskip 1mm
  \parpic[l]{\begin{picture}(0,0)%
\includegraphics{figs/Topology.pstex}%
\end{picture}%
%
%  pstex_opts: -m 0.3333 
%
\setlength{\unitlength}{1316sp}%
\begingroup\makeatletter\ifx\SetFigFont\undefined%
\gdef\SetFigFont#1#2#3#4#5{%
  \reset@font\fontsize{#1}{#2pt}%
  \fontfamily{#3}\fontseries{#4}\fontshape{#5}%
  \selectfont}%
\fi\endgroup%
\begin{picture}(1052,1180)(0,-190)
\end{picture}%
\,\begin{picture}(0,0)%
\includegraphics{figs/Lie.pstex}%
\end{picture}%
%
%  pstex_opts: -m 0.36 
%
\setlength{\unitlength}{1421sp}%
\begingroup\makeatletter\ifx\SetFigFont\undefined%
\gdef\SetFigFont#1#2#3#4#5{%
  \reset@font\fontsize{#1}{#2pt}%
  \fontfamily{#3}\fontseries{#4}\fontshape{#5}%
  \selectfont}%
\fi\endgroup%
\begin{picture}(960,990)(46,0)
\end{picture}%
  \hspace{-6pt}}%
  \noindent%
 In topology, $d\eta^a$ is the removal of one component from a
2-tangle. In Lie theory it corresponds to the co-unit
$\glosm{eta}{\eta}\colon\calU(I\frakg)\to\bbQ$.

\Needspace{2cm}
\setcounter{tunnel}{\theenumi}
\end{enumerate}
\begin{enumerate}[leftmargin=*,labelindent=0pt]\setcounter{enumi}{\thetunnel}

\vskip 2mm
\parpic[r]{
  \def\dAAA{{$\act(\dA^1,\dA^2,\dA^3)$}}
  \input{figs/dA.pstex_t}
}
\item Given $D\in\calA^w(S)$ and $a\in S$, $D\act
\glosm{dA}{\dA}^a$ is the result of ``flipping over stand $a$ and
multiplying by a $(-)$ sign for each arrow whose head connects to $a$'', as
illustrated above. We denote by $\dA$ the operation of likewise flipping
(with signs) {\em all} strands: $\dA=\dA^S\coloneqq\prod_{a\in S}\dA^a$.

\vskip 1mm
  \parpic[l]{\begin{picture}(0,0)%
\includegraphics{figs/Topology.pstex}%
\end{picture}%
%
%  pstex_opts: -m 0.3333 
%
\setlength{\unitlength}{1316sp}%
\begingroup\makeatletter\ifx\SetFigFont\undefined%
\gdef\SetFigFont#1#2#3#4#5{%
  \reset@font\fontsize{#1}{#2pt}%
  \fontfamily{#3}\fontseries{#4}\fontshape{#5}%
  \selectfont}%
\fi\endgroup%
\begin{picture}(1052,1180)(0,-190)
\end{picture}%
\,\begin{picture}(0,0)%
\includegraphics{figs/Lie.pstex}%
\end{picture}%
%
%  pstex_opts: -m 0.36 
%
\setlength{\unitlength}{1421sp}%
\begingroup\makeatletter\ifx\SetFigFont\undefined%
\gdef\SetFigFont#1#2#3#4#5{%
  \reset@font\fontsize{#1}{#2pt}%
  \fontfamily{#3}\fontseries{#4}\fontshape{#5}%
  \selectfont}%
\fi\endgroup%
\begin{picture}(960,990)(46,0)
\end{picture}%
  \hspace{-6pt}}%
  \noindent%
 In topology, $\dA^a$ is the reversal of the 1D orientation
of a knotted tube \cite{WKO2}. In Lie theory, it is the antipode of
$\calU(I\frakg)$ combined with the sign reversal $\varphi\to-\varphi$
acting on the $\frakg^\ast$ factor of $I\frakg$. When elements of
$\calU(I\frakg)^{\otimes S}$ are interpreted as differential operators
acting on functions on $\frakg^S$, $\dA$ corresponds to the $L^2$ adjoint.

\Needspace{3cm}
\setcounter{tunnel}{\theenumi}
\end{enumerate}
\begin{enumerate}[leftmargin=*,labelindent=0pt]\setcounter{enumi}{\thetunnel}

\vskip 2mm
\parpic[r]{
  \def\dSSS{{$\act(\dS^1,\dS^2,\dS^3)$}}
  \input{figs/dS.pstex_t}
}
\item Similarly, $D\act \glosm{dS}{\dS}^a$ is the result of ``flipping over
stand $a$ and multiplying by a $(-)$ sign for each arrow head or tail
that connects to $a$'', as illustrated above\footnoteT{The letter $S$
is used here for both ``a set of strands'' and ``an operation similar
to an antipode''. Hopefully no confusion will arise.}.

\vskip 1mm
  \parpic[l]{\begin{picture}(0,0)%
\includegraphics{figs/Topology.pstex}%
\end{picture}%
%
%  pstex_opts: -m 0.3333 
%
\setlength{\unitlength}{1316sp}%
\begingroup\makeatletter\ifx\SetFigFont\undefined%
\gdef\SetFigFont#1#2#3#4#5{%
  \reset@font\fontsize{#1}{#2pt}%
  \fontfamily{#3}\fontseries{#4}\fontshape{#5}%
  \selectfont}%
\fi\endgroup%
\begin{picture}(1052,1180)(0,-190)
\end{picture}%
\,\begin{picture}(0,0)%
\includegraphics{figs/Lie.pstex}%
\end{picture}%
%
%  pstex_opts: -m 0.36 
%
\setlength{\unitlength}{1421sp}%
\begingroup\makeatletter\ifx\SetFigFont\undefined%
\gdef\SetFigFont#1#2#3#4#5{%
  \reset@font\fontsize{#1}{#2pt}%
  \fontfamily{#3}\fontseries{#4}\fontshape{#5}%
  \selectfont}%
\fi\endgroup%
\begin{picture}(960,990)(46,0)
\end{picture}%
  \hspace{-6pt}}%
  \noindent%
 In topology, $\dS^a$ is the reversal of both the 1D and the 2D
orientation of a knotted tube \cite{WKO2}. In Lie theory, it is the
antipode of $\calU(I\frakg)$.

\vskip 2mm
\parpic[r]{
  \def\dm{{$\act dm^{23}_2$}}
  \input{figs/dm.pstex_t}
}
\item Given $D\in\calA^w(S)$, given $a,b\in S$, and
given $c\not\in S\remove\{a,b\}$, $D\act \glosm{dm}{dm}^{ab}_c$ is the
result of ``stitching strands $a$ and $b$ and calling the resulting
strand $c$'', as illustrated on the right.

\vskip 1mm
  \parpic[l]{\begin{picture}(0,0)%
\includegraphics{figs/Topology.pstex}%
\end{picture}%
%
%  pstex_opts: -m 0.3333 
%
\setlength{\unitlength}{1316sp}%
\begingroup\makeatletter\ifx\SetFigFont\undefined%
\gdef\SetFigFont#1#2#3#4#5{%
  \reset@font\fontsize{#1}{#2pt}%
  \fontfamily{#3}\fontseries{#4}\fontshape{#5}%
  \selectfont}%
\fi\endgroup%
\begin{picture}(1052,1180)(0,-190)
\end{picture}%
\,\begin{picture}(0,0)%
\includegraphics{figs/Lie.pstex}%
\end{picture}%
%
%  pstex_opts: -m 0.36 
%
\setlength{\unitlength}{1421sp}%
\begingroup\makeatletter\ifx\SetFigFont\undefined%
\gdef\SetFigFont#1#2#3#4#5{%
  \reset@font\fontsize{#1}{#2pt}%
  \fontfamily{#3}\fontseries{#4}\fontshape{#5}%
  \selectfont}%
\fi\endgroup%
\begin{picture}(960,990)(46,0)
\end{picture}%
  \hspace{-6pt}}%
  \noindent%
 In topology, $dm^{ab}_c$ is the ``internal stitching''
of two tubes within a single 2-link, akin to the ``stitching''
operation that combines two strands of an ordinary tangle into a
single ``longer'' one. In Lie theory, it is an ``internal product''
$\calU(I\frakg)^{\otimes n}\to\calU(I\frakg)^{\otimes(n-1)}$ which
``merges'' two factors within $\calU(I\frakg)^{\otimes n}$.

\vskip 2mm
\parpic[r]{
  \def\dD{{$\act d\Delta^2_{2'2''}$}}
  \input{figs/dD.pstex_t}
}
\item Given $D\in\calA^w(S)$, given $a\in S$, and
given $b,c\not\in S\remove a$, $D\act\glosm{dDelta}{d\Delta}^a_{bc}$ is
the result of ``doubling'' strand $a$, calling the resulting ``daughter
strands'' $b$ and $c$, and summing over all ways of lifting the arrows
that were connected to $a$ to either $b$ or $c$ (so if there are $k$
arrows connected to $a$, $D\act d\Delta^a_{bc}$ is a sum of $2^k$
diagrams).

\vskip 1mm
  \parpic[l]{\begin{picture}(0,0)%
\includegraphics{figs/Topology.pstex}%
\end{picture}%
%
%  pstex_opts: -m 0.3333 
%
\setlength{\unitlength}{1316sp}%
\begingroup\makeatletter\ifx\SetFigFont\undefined%
\gdef\SetFigFont#1#2#3#4#5{%
  \reset@font\fontsize{#1}{#2pt}%
  \fontfamily{#3}\fontseries{#4}\fontshape{#5}%
  \selectfont}%
\fi\endgroup%
\begin{picture}(1052,1180)(0,-190)
\end{picture}%
\,\begin{picture}(0,0)%
\includegraphics{figs/Lie.pstex}%
\end{picture}%
%
%  pstex_opts: -m 0.36 
%
\setlength{\unitlength}{1421sp}%
\begingroup\makeatletter\ifx\SetFigFont\undefined%
\gdef\SetFigFont#1#2#3#4#5{%
  \reset@font\fontsize{#1}{#2pt}%
  \fontfamily{#3}\fontseries{#4}\fontshape{#5}%
  \selectfont}%
\fi\endgroup%
\begin{picture}(960,990)(46,0)
\end{picture}%
  \,\,\begin{picture}(0,0)%
\includegraphics{figs/AT.pstex}%
\end{picture}%
%
%  pstex_opts: -m 0.36 
%
\setlength{\unitlength}{1421sp}%
\begingroup\makeatletter\ifx\SetFigFont\undefined%
\gdef\SetFigFont#1#2#3#4#5{%
  \reset@font\fontsize{#1}{#2pt}%
  \fontfamily{#3}\fontseries{#4}\fontshape{#5}%
  \selectfont}%
\fi\endgroup%
\begin{picture}(1041,1009)(118,-19)
\end{picture}%

  \hspace{-6pt}}%
  \noindent%
 In topology, $d\Delta$ is the operation of ``doubling''
one component in a 2-link. In Lie theory, it is the co-product
$\glosm{Delta}{\Delta}\colon\calU(I\frakg)\to\calU(I\frakg)^{\otimes 2}$
acting on the $a$ factor in $\calU(I\frakg)^{\otimes S}$, extended by
the identity acting on all other factors. In~\cite{AT}, it is the coface
maps of \cite[Example~3.14]{AT}.

\setcounter{tunnel}{\theenumi}
\end{enumerate}
\begin{enumerate}[leftmargin=*,labelindent=0pt]\setcounter{enumi}{\thetunnel}

\item Finally, the operation $\glosm{dsigma}{d\sigma}^a_b\colon
\calA(S)\to \calA({S\remove\{a\}\sqcup\{b\}})$ does
nothing but renaming the strand $a$ to $b$ (assuming $a\in S$ and
$b\not\in S\remove\{a\}$).  \endpar{def:Operations}

\end{enumerate}

\end{definition}

We note that the product operation $(D_1,D_2)\mapsto D_1\ast D_2$ can be
implemented using the union operation $\sqcup$, the stitching
operation $dm$, and some renaming --- namely, if $\bar{S}=\{\bar{a}\colon
a\in S\}$ is some set of ``temporary'' labels disjoint from $S$ but in
a bijection with $S$, then
\begin{equation} \label{eq:multiplem} D_1\ast D_2 = 
  \left(
    D_1\sqcup\left(D_2\act\prod_ad\sigma^a_{\bar{a}}\right)
  \right)\act\prod_adm^{a\bar{a}}_a.
\end{equation}
Therefore below we will sometimes omit the implementation of
$(D_1,D_2)\mapsto D_1D_2$ provided all other operations are implemented.

We note that $\calA^w(S)$ is a co-algebra, with the co-product
$\glosm{Box}{\Box}(D)$, for a diagram $D$ representing an element of
$\calA^w(S)$, being the sum of all ways of dividing $D$ between a ``left
co-factor'' and a ``right co-factor'' so that connected components of
$D\setminus(\uparrow\!\!\times S)$ ($D$ with its skeleton removed) are
kept intact (compare with~\cite[Definition~3.7]{Bar-Natan:OnVassiliev}).

\begin{definition} \label{def:GroupLike} An element $Z$ of
$\calA^w(S)$ is ``group-like'' if $\Box(Z)=Z\otimes Z$. We
denote the set of group-like elements in $\calA^w(S)$ by
$\glosm{calAwexp}{\calA^w_{\exp}}(S)$.
\end{definition}

We leave it for the reader to verify that all the operations defined above
restrict to operations $\calA^w_{\exp}\to\calA^w_{\exp}$.

  \parpic[l]{\begin{picture}(0,0)%
\includegraphics{figs/Topology.pstex}%
\end{picture}%
%
%  pstex_opts: -m 0.3333 
%
\setlength{\unitlength}{1316sp}%
\begingroup\makeatletter\ifx\SetFigFont\undefined%
\gdef\SetFigFont#1#2#3#4#5{%
  \reset@font\fontsize{#1}{#2pt}%
  \fontfamily{#3}\fontseries{#4}\fontshape{#5}%
  \selectfont}%
\fi\endgroup%
\begin{picture}(1052,1180)(0,-190)
\end{picture}%
  \hspace{-6pt}}%
  \noindent%
 In topology, $\Box$ is the operation of ``cloning'' an entire
2-link. It is not to be confused with $d\Delta$; one dimension down and
with just one component, the pictures are:
\[ \input{figs/BoxVsDelta.pstex_t} \]

  \parpic[l]{\begin{picture}(0,0)%
\includegraphics{figs/AT.pstex}%
\end{picture}%
%
%  pstex_opts: -m 0.36 
%
\setlength{\unitlength}{1421sp}%
\begingroup\makeatletter\ifx\SetFigFont\undefined%
\gdef\SetFigFont#1#2#3#4#5{%
  \reset@font\fontsize{#1}{#2pt}%
  \fontfamily{#3}\fontseries{#4}\fontshape{#5}%
  \selectfont}%
\fi\endgroup%
\begin{picture}(1041,1009)(118,-19)
\end{picture}%
  \hspace{-8pt}}%
  \noindent%
 In~\cite{AT}, $\Box$ is the co-product of
$\hat{\calU}((\fraka\oplus\attder)\ltimes\attr)$ and moding out by
wheels, $\calA^w_{\exp}$ is $\TAut$.

  \parpic[l]{\begin{picture}(0,0)%
\includegraphics{figs/Lie.pstex}%
\end{picture}%
%
%  pstex_opts: -m 0.36 
%
\setlength{\unitlength}{1421sp}%
\begingroup\makeatletter\ifx\SetFigFont\undefined%
\gdef\SetFigFont#1#2#3#4#5{%
  \reset@font\fontsize{#1}{#2pt}%
  \fontfamily{#3}\fontseries{#4}\fontshape{#5}%
  \selectfont}%
\fi\endgroup%
\begin{picture}(960,990)(46,0)
\end{picture}%
  \hspace{-6pt}}%
  \noindent%
 In Lie theory, $\Box$ is {\em not} the co-product
$\Delta\colon\calU(I\frakg)\to\calU(I\frakg)^{\otimes 2}$. Rather, given
two finite dimensional Lie algebras $\frakg_1$ and $\frakg_1$, $\Box$
corresponds to the map
\[ \Box\colon
  \calU(I(\frakg_1\oplus\frakg_2))^{\otimes S}
  \to \calU(I\frakg_1)^{\otimes S}\otimes\calU(I\frakg_2)^{\otimes S}.
\]

\picskip{0}
\begin{discussion} \label{disc:Primitives} We seek to have efficient
descriptions of the elements of $\calA^w_{\exp}(S)$ and efficient
means of computing the above operations on such elements.

Let $\glosm{Aprimw}{\calA_{\text{prim}}^w}(S)$\footnoteT{$\calA_{\text{prim}}^w$
is elsewhere denoted $\calP^w$.} denote the set of primitives
of $\calA^w(S)$: these are the elements
$\zeta\in\calA^w(S)$ satisfying $\Box(\zeta)=\zeta\otimes
1+1\otimes\zeta$. Let $\glosm{FL}{\FL}(S)$ denote the degree-completed free Lie
algebra with generators $S$, and let $\glosm{CW}{\CW}(S)$ denote the degree-completed
vector space spanned by non-empty cyclic words on the alphabet $S$.
In~\cite[Proposition~\ref{2-prop:Pnses}]{WKO2} we have shown that there
is a short exact sequence of vector spaces
\begin{equation} \label{eq:Primitives}
  0\to\CW(S)\to\calA_{\text{prim}}^w(S)\to\FL(S)^S\to 0,
\end{equation}
where $\FL(S)^S$ denotes the set of all functions $S\to\FL(S)$. Hence
$\calA_{\text{prim}}^w(S)\simeq\FL(S)^S\oplus\CW(S)$ (not canonically!). Often
in bi-algebras there is a bijection given by $\zeta\mapsto
e^\zeta$ between primitive elements $\zeta$ and group-like elements
$e^\zeta$. Hence we may expect to be able to present elements of
${\calA^w_{\exp}(S)}$ as formal exponentials of combinations
of ``trees'' (elements of $\FL(S)^S$) and ``wheels'' (elements of
$\CW(S)$)\footnoteT{
  We use the set-theoretic notation ``$\times$'' rather than the
  linear-algebraic ``$\oplus$'' in Equation~\eqref{eq:expectation} to
  emphasize that the two sides of that equation are only expected to be
  set-theoretically isomorphic. The left-hand-side, in fact, is not even
  a linear space in a natural way.
}:
\begin{equation} \label{eq:expectation}
  \calA^w_{\exp}(S) \sim \yellowm{\TW}(S)
    \coloneqq \FL(S)^S\times\CW(S)
  =\left\{
    (\yellowm{\lambda};\,\yellowm{\omega})\colon\begin{array}{c}
      \lambda=\{a\to\lambda_a\}_{a\in S},\,\lambda_a\in\FL(S) \\
      \omega\in\CW(S)
    \end{array}
  \right\}.
\end{equation}
\glosm{expectation}{\null}We implement Equation~\eqref{eq:expectation}
in a more-or-less straightforward way in Section~\ref{subsec:AT} and in a
less straightforward but somewhat stronger way in Section~\ref{subsec:Ef}.
\endpar{disc:Primitives}

\end{discussion}

\begin{discussion} \label{disc:WhyTwo}
Why are there two presentations for elements of $\calA^w_{\exp}$?

Because as we shall see, $\calA^w$ is a bi-algebra in two ways, using two
different products, yet with the same co-product $\Box$. In $\calA^w$, the
notions ``primitive'' and ``group-like'', whose definition involves only
$\Box$, are canonical. Yet the bijection between primitive and group-like
elements, $\zeta\leftrightarrow e^\zeta$, depends also on the product
used within the power-series interpretation of $e^\zeta$. Thus there are
two different ways to describe the group-like elements $\calA^w_{\exp}$
of $\calA^w$ in terms of its primitives $\TW$.

The first product on $\calA^w$ is the stacking product of
Equation~\eqref{eq:TubeProduct}. The second will be introduced later,
in Equations~\eqref{eq:AHTStacking} and~\eqref{eq:FProduct}.

  \parpic[l]{\begin{picture}(0,0)%
\includegraphics{figs/Topology.pstex}%
\end{picture}%
%
%  pstex_opts: -m 0.3333 
%
\setlength{\unitlength}{1316sp}%
\begingroup\makeatletter\ifx\SetFigFont\undefined%
\gdef\SetFigFont#1#2#3#4#5{%
  \reset@font\fontsize{#1}{#2pt}%
  \fontfamily{#3}\fontseries{#4}\fontshape{#5}%
  \selectfont}%
\fi\endgroup%
\begin{picture}(1052,1180)(0,-190)
\end{picture}%
  \hspace{-6pt}}%
  \noindent%
 Very roughly speaking, $\calA^w$ is a combinatorial
model of ``$\pi_1\ltimes\pi_2$'' (with homotopies replaced by isotopies;
see~\cite{KBH}). The other product on $\calA^w$ is the one coming from the
direct product ``$\pi_1\times\pi_2$''.

  \parpic[l]{\begin{picture}(0,0)%
\includegraphics{figs/Lie.pstex}%
\end{picture}%
%
%  pstex_opts: -m 0.36 
%
\setlength{\unitlength}{1421sp}%
\begingroup\makeatletter\ifx\SetFigFont\undefined%
\gdef\SetFigFont#1#2#3#4#5{%
  \reset@font\fontsize{#1}{#2pt}%
  \fontfamily{#3}\fontseries{#4}\fontshape{#5}%
  \selectfont}%
\fi\endgroup%
\begin{picture}(960,990)(46,0)
\end{picture}%
  \hspace{-6pt}}%
  \noindent%
 Very roughly speaking, $\calA^w$ is a combinatorial model of
(tensor powers of a completion of) $\calU(I\frakg)$.  By PBW,
$\calU(I\frakg)\simeq\calU(\frakg)\otimes\glosm{calS}{\calS}(\frakg^\ast)$
as co-algebras but not as algebras. The other product on
$\calA^w$ is the one corresponding to the natural product on
$\calU(\frakg)\otimes\calS(\frakg^\ast)$. The reality is a bit more
delicate, though. $\calA^w$ is only a model of (a small part of)
the $\frakg$-invariant part of $\calU(I\frakg)$, and the co-product
$\Box$ of $\calA^w$ does not correspond to the co-product $\Delta$
of $\calU(I\frakg)$. \endpar{disc:WhyTwo}

\end{discussion}

\draftcut\subsection{Some preliminaries about free Lie algebras and cyclic
  words} \label{subsec:FL}

It should be clear from Discussion~\ref{disc:Primitives} that free Lie
algebras and cyclic words play a prominent role in this paper. For the
convenience of our readers we collect in this section some preliminaries
about about these topics. Almost everything in this section comes either
from Alekseev-Torossian's~\cite{AT}, or from~\cite{WKO2,KBH}, and the
detailed proofs of the assertions made here can be found in these papers.

  \parpic[l]{\begin{picture}(0,0)%
\includegraphics{figs/NotLie.pstex}%
\end{picture}%
%
%  pstex_opts: -m 0.36 
%
\setlength{\unitlength}{1421sp}%
\begingroup\makeatletter\ifx\SetFigFont\undefined%
\gdef\SetFigFont#1#2#3#4#5{%
  \reset@font\fontsize{#1}{#2pt}%
  \fontfamily{#3}\fontseries{#4}\fontshape{#5}%
  \selectfont}%
\fi\endgroup%
\begin{picture}(960,1052)(46,-62)
\end{picture}%
  \hspace{-6pt}}%
  \noindent%
 Note that Lie algebras appear in two distinct
roles in this paper.  {\em Free} Lie algebras $\FL$ appear
along with cyclic words $\CW$ as the primitives of $\calA^w$
(Equation~\eqref{eq:Primitives}). {\em Finite dimensional} Lie algebras
$\frakg$ appear only as motivational comments, always marked with a
\raisebox{-4pt}{\scalebox{0.75}{\begin{picture}(0,0)%
\includegraphics{figs/Lie.pstex}%
\end{picture}%
%
%  pstex_opts: -m 0.36 
%
\setlength{\unitlength}{1421sp}%
\begingroup\makeatletter\ifx\SetFigFont\undefined%
\gdef\SetFigFont#1#2#3#4#5{%
  \reset@font\fontsize{#1}{#2pt}%
  \fontfamily{#3}\fontseries{#4}\fontshape{#5}%
  \selectfont}%
\fi\endgroup%
\begin{picture}(960,990)(46,0)
\end{picture}%
}} symbol. As
already indicated, elements in $\calA^w$, and hence elements of $\FL$
and of $\CW$ can represent ``universal'' formulas that make sense in any
finite dimensional Lie algebra $\frakg$. Hence part of our discussion
of $\FL$ and $\CW$ is a discussion of things that make sense universally
for all finite dimensional Lie algebras.

Recall that $\FL(S)$ denotes the graded completion of the
free Lie algebra over a set of generators $S$, all considered
to have degree $1$. In the case when $S=\{x_1,\ldots,x_n\}$,
Alekseev and Torossian~\cite{AT}
denote this space $\glosm{lie}{\lie}_n$.\footnoteC{
In computer talk, generators of $\FL(S)$ are always
single-character ``Lyndon words'' (e.g.~\cite{Reutenauer:FreeLie}); 
in our case we set $x$ and $y$ to be the single-character words ``$x$'' and
``$y$'', and then $\alpha$, $\beta$, and $\gamma$ to be the Lie series
$x+[x,y]$, $y-[x,[x,y]]$, and $x+y-2[x,y]$ (elements of $\FL$ are infinite
series, in general, but these examples are finite):

\shortdialoginclude{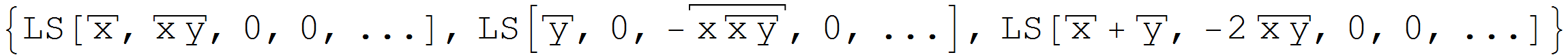}

Note that as we requested earlier, our example series are printed
to degree 4. Note also that they are printed using ``top bracket''
$\glosm{ob}{\ob{xy}}\coloneqq [x,y]$ notation, which is easier to read
when many brackets are nested.

We then compute $[\alpha,\beta]$ and verify the Jacobi identity for
$\alpha$, $\beta$, and $\gamma$:

\shortdialoginclude{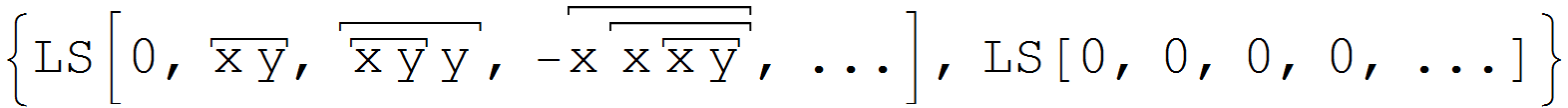}
}

A noteworthy element of $\FL(x,y)$ is the Baker-Campbell-Hausdorff
series,\footnoteC{In computer talk:

\shortdialogincludewithlink{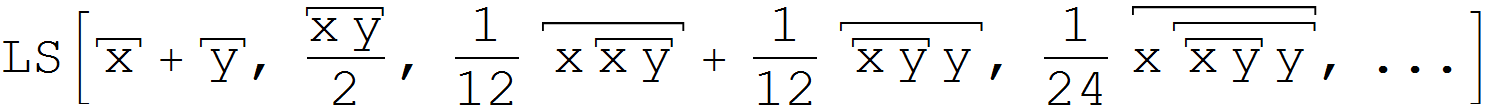}

Just to show that we can, here are the lexicographically
middle three of the 2,181 terms of the BCH series in degree 16, along
with the time in seconds it took my humble laptop to compute it:

\shortdialoginclude{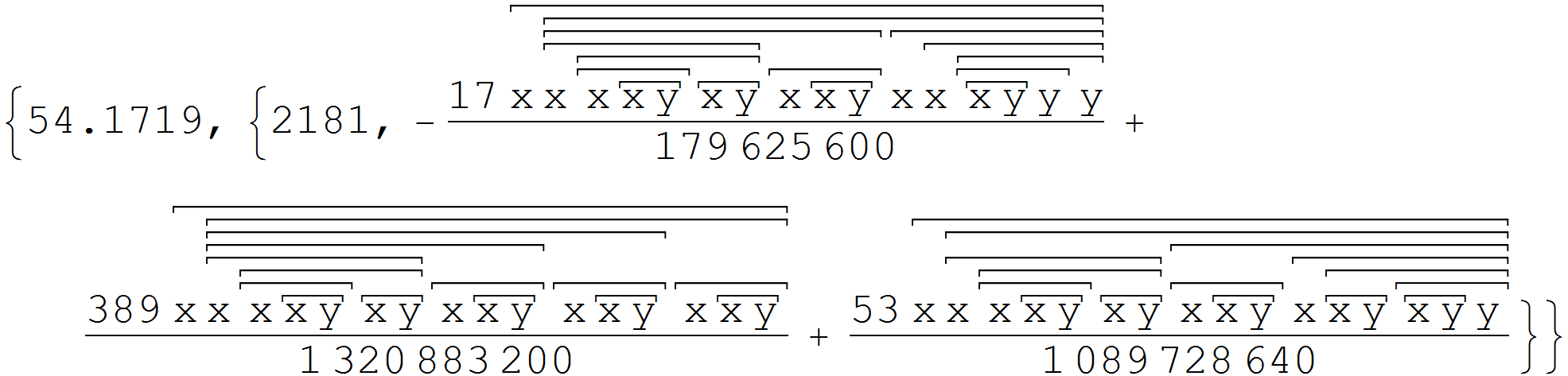}

(In a few hours my laptop computed the BCH series to degree 22; in as much
as I know, the farthest it was ever computed. See~\cite{Bar-Natan:BCH,
CasasMurua:EfficientBCH}.)
}
\[ \glosm{BCH}{\BCH}(x,y)\coloneqq\log(e^xe^y)
  = x+y+\frac{[x,y]}{2} +\frac{[x,[x,y]] + [[x,y],y]}{12}
    +\ldots.
\]

Recall also that $\CW(S)$ ($\glosm{attr}{\attr}_n$,
in~\cite{AT}) denotes the graded completion
of the vector space spanned by non-empty cyclic words in the alphabet
$S$. Our convention is to crown cyclic words with an ``arch''; thus
$\glosm{wideparen}{\wideparen{uvw}} = \wideparen{vwu}$\footnoteC{Cyclic
words in computer talk:

\shortdialoginclude{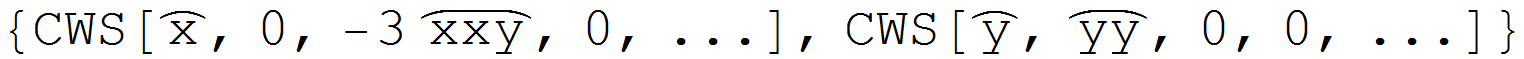}
}. Note that there is a map $\CW(\FL(S))\to\CW(S)$ by interpreting brackets
within elements of $\FL(S)$ as commutators and then mapping ``long'' words
to cyclic words. E.g., $\wideparen{u[v,w]} = \wideparen{uvw} -
\wideparen{uwv}$.

We denote by $\glosm{hdeg}{h^{\deg}}$ the operations $\FL\to\FL$ and
$\CW\to\CW$ which multiply any degree $k$ element by $h^k$. In particular,
$\glosm{mdeg}{(-1)^{\deg}}$ acts on $\FL/\CW$ as the identity in even
degrees and as minus the identity in odd degrees.\footnoteC{In computer
talk:

\shortdialoginclude{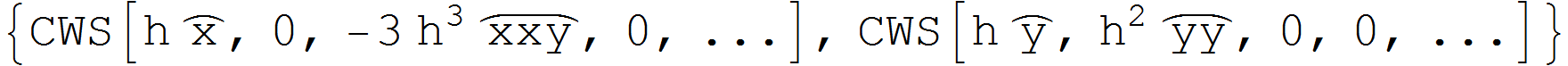}
}

Let $\glosm{der}{\der}_S$ denote the Lie algebra of all derivations of $\FL(S)$
($\glosm{atder}{\atder}_n$ in~\cite{AT}). There
is a linear map $\glosm{partial}{\partial}\colon\FL(S)^S\to\der_S$
which assigns to every $\lambda=(\lambda_a)_{a\in
S}\in\FL(S)^S$ the unique derivation $\partial_\lambda$ for
which $\partial_\lambda(a)=[a,\lambda_a]$ for every $a\in
S$.\footnotemarkT\,\footnotemarkC\ The image of $\partial$ is
a subalgebra of $\der_S$ denoted $\glosm{tder}{\tder}_S$
($\glosm{attder}{\attder}_n$ in~\cite{AT}); the elements of $\tder_S$
are called ``tangential derivations''. The kernel of $\partial$ can
be identified as the Abelian Lie algebra $\glosm{A}{A}_S$ generated by $S$
($\glosm{fraka}{\fraka}_n$ in~\cite{AT}), which is linearly embedded
in $\FL(S)^S$ as the set of all sequences $\lambda\colon S\to\FL(S)$ for
which $\lambda_a$ is a scalar multiple of $a$ for every $a\in S$. Thus
we have a short exact sequence of vector spaces
\begin{equation} \label{eq:FLisAtder}
  0 \rightarrow A_S\rightarrow \FL(S)^S\xrightarrow{\partial}
  \tder_S\rightarrow 0.
\end{equation}
The map $\FL(S)^S\ni\lambda=(\lambda_a)\mapsto\sum_a\langle\lambda_a,
a\rangle a\in A_S$, where $\langle\lambda_a, a\rangle$ is the
coefficient of $a$ in $\lambda_a$ is a splitting of the above sequence,
and hence $\FL(S)^S\simeq A_S\oplus\tder_S$ in a canonical manner.

\footnotetextC{An example:

\shortdialoginclude{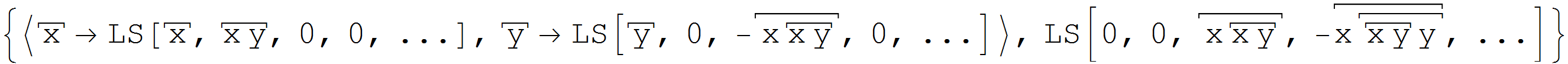}
}

There is a unique Lie bracket $\glosm{tb}{[\cdot,\cdot]_{tb}}$ (the
``tangential bracket'') on $\FL(S)^S$ which makes \eqref{eq:FLisAtder}
a split exact sequence of Lie algebras, and hence
$(\FL(S)^S,[,]_{tb})\simeq A_S\oplus\tder_S$ as Lie algebras. With
$[\cdot,\cdot]$ denoting the ordinary direct-sum bracket on $\FL(S)^S$
and with the action of $\partial_\lambda$ extended to
$\partial_\lambda\colon\FL(S)^S\to\FL(S)^S$ in the obvious manner,
we have\footnotemarkC
\[ [\lambda_1,\lambda_2]_{tb}
  =[\lambda_1,\lambda_2]
    +\partial_{\lambda_1}\lambda_2
    -\partial_{\lambda_2}\lambda_1.
\]
\footnotetextT{Using the notation of~\cite{KBH}, $\partial_\lambda =
-\sum_{a\in S}\ad_a^{\lambda_a} = -\sum_{a\in S}\ad_a\{\lambda_a\}$. I
apologize for the minus sign which stems from a bad choice made
in~\cite{KBH}.}
\footnotetextC{For example:

\shortdialoginclude{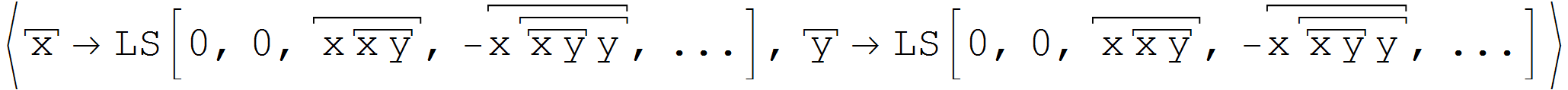}
}

The $\lambda\mapsto\partial_\lambda$ action of $(\FL(S)^S,[,]_{tb})$
on $\FL(S)$ extends to an action on the universal enveloping
algebra of $\FL(S)$, the free associative algebra $\FA(S)$ on $S$
generators, and then descends to the vector-space quotient of
$\FA(S)$ by commutators, namely to cyclic words.  Leaving aside
the empty word, we find that $(\FL(S)^S,[,]_{tb})$ acts on
$\CW(S)$, and hence also on $\TW(S)$.\footnoteC{ We check that
up to degree 8, $\partial_{[\lambda_1,\lambda_2]_{tb}}(\omega_1) =
[\partial_{\lambda_1},\partial_{\lambda_2}](\omega_1)$ (for our choice
of $\lambda_1$, $\lambda_2$, and $\omega_1$, both sides vanish below
degree 8):

\shortdialoginclude{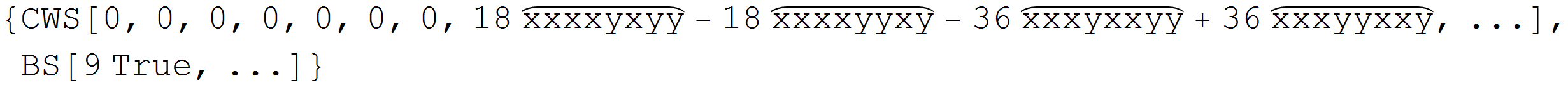}

Note that the comparison operator $\equiv$ returns a ``Boolean Sequence''
({\tt BS}) rather than a single {\tt True}/{\tt False} value, as the
computer has no way of knowing whether two series are equal without
computing them up to a given degree. In our case, we've asked for the
comparison of {\tt lhs} with {\tt rhs} up to degree 8, and the output,
including degree 0, is a sequence of 9 affirmations, summarized as ``{\tt 9
True}''.
}

There are two ways to assign an automorphism of the free Lie algebra
$\FL(S)$ to an element $\lambda\in\FL(S)^S$:
\begin{enumerate}
\item One may exponentiate the derivation $\partial_\lambda$ to get
$e^{\partial_\lambda}\colon\FL(S)\to\FL(S)$.
\item One may define an automorphism
$\glosm{C}{C^\lambda}\colon\FL(S)\to\FL(S)$ by setting
its values on the generators by $C^\lambda(a)\coloneqq
e^{\lambda_a}ae^{-\lambda_a}=e^{\ad\lambda_a}a$. We denote the inverse
of $C^\lambda$ by $\glosm{RC}{RC^{-\lambda}}$ and note that it is {\em not}
$C^{-\lambda}$.
\end{enumerate}

  \parpic[l]{\begin{picture}(0,0)%
\includegraphics{figs/AT.pstex}%
\end{picture}%
%
%  pstex_opts: -m 0.36 
%
\setlength{\unitlength}{1421sp}%
\begingroup\makeatletter\ifx\SetFigFont\undefined%
\gdef\SetFigFont#1#2#3#4#5{%
  \reset@font\fontsize{#1}{#2pt}%
  \fontfamily{#3}\fontseries{#4}\fontshape{#5}%
  \selectfont}%
\fi\endgroup%
\begin{picture}(1041,1009)(118,-19)
\end{picture}%
  \hspace{-8pt}}%
  \noindent%
 In~\cite{AT},~(1) corresponds to the presentation of elements of the
automorphism group $\glosm{TAut}{\TAut_n}$ as exponentials of elements of
its Lie algebra $\tder_n$, while~(2) corresponds to its presentation in
terms of ``basis conjugating automorphisms'' $x_i\mapsto g_i^{-1}x_ig_i$
where $g_i=e^{-\lambda_i}$. Compare with \cite[Section~5.1]{AT}.

The following pair of propositions, which we could not find elsewhere, relates
these two automorphisms:

\begin{proposition} \label{prop:Gamma}
Given $\lambda\in\FL(S)^S$, let $t$ be a scalar-valued
formal variable and let $\glosm{Gammat}{\Gamma_t(\lambda)}\in\FL(S)^S$
be the (unique) solution of the ordinary differential equation
\begin{equation} \label{eq:GammaODE}
  \Gamma_0(\lambda)=0
  \qquad\text{and}\qquad
  \frac{d\Gamma_t(\lambda)}{dt} = \lambda \act e^{-t\partial_\lambda}
    \act \frac{\ad\Gamma_t(\lambda)}{e^{\ad\Gamma_t(\lambda)}-1}.
\end{equation}
\begin{flalign} \label{eq:Gamma}
  & \text{Then} & e^{-t\partial_\lambda}=C^{\Gamma_t(\lambda)}.\footnotemarkC &&
\end{flalign}
\end{proposition}
\footnotetextC{
We verify that the computer-calculated $\Gamma_t(\lambda)$ satisfies the
ODE in~\eqref{eq:GammaODE} and then that the operator equality~\eqref{eq:Gamma}
holds, at least when evaluated on ``our'' $\gamma$:

\shortdialoginclude{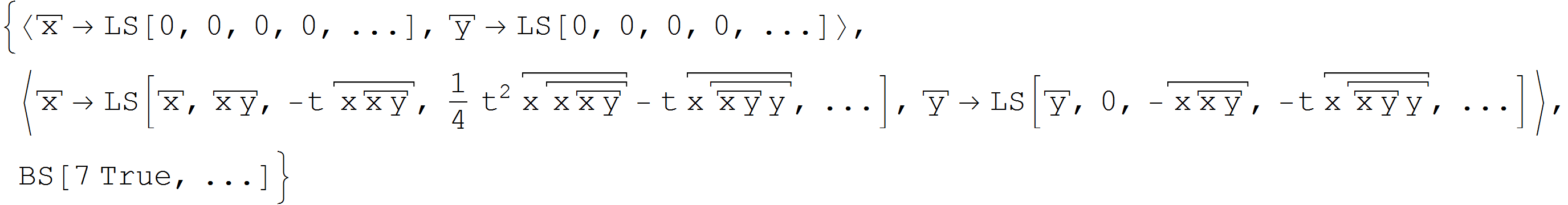}

\shortdialoginclude{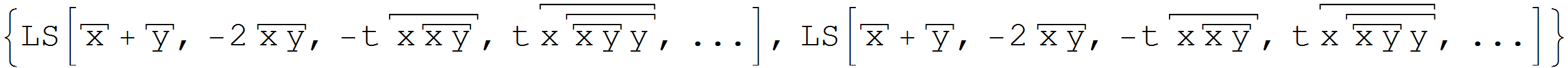}
}

\begin{proof} The two sides $L_t$ and $R_t$ of Equation~\eqref{eq:Gamma}
are power-series perturbations of the identity automorphism of
$\FL(S)$. More fully, $L_t$ can be written $L_t=\sum_{d\geq 0}t^dL(d)$
where $L(d)\colon\FL(S)\to\FL(S)$ raises degrees by at least $d$
(and so the sum converges), and where $L(0)$ is the identity.  $R_t$
can be written in a similar way. We claim that it is enough to prove that
\begin{equation} \label{eq:AB}
  A_t\coloneqq(\frac{dL_t}{dt})\act L_t^{-1}
  = (\frac{dR_t}{dt})\act R_t^{-1} \eqqcolon B_t.
\end{equation}
Indeed, if otherwise $L_t\neq R_t$, consider the minimal $d$ for which
$L(d)\neq R(d)$.  Then $d>0$ and the least-degree term in $A_t-B_t$
is the degree $d-1$ term, which equals $dt^{d-1}L(d)\act L_t^{-1} -
dt^{d-1}R(d)\act R_t^{-1} = dt^{d-1}(L(d)-R(d))\act L_t^{-1} \neq 0$
(the last equality is because $L_t^{-1}=R_t^{-1}$ to degree $d$),
contradicting Equation~\eqref{eq:AB}. Note that in fact we have shown
that if $A_t=B_t$ to degree $d$ in $t$, then Equation~\eqref{eq:Gamma}
holds to degree $d+1$.

To compute $B_t$ we need the differential of
$C^\mu$ (at $\mu=\Gamma_t(\lambda)$) and the chain rule. The differential
of $C^\mu$ is quite difficult; fortunately, we have computed it in the
case where $\mu=(u\to\gamma)$ is supported on just one $u\in S$,
in~\cite[Lemma~\ref{KBH-lem:dC}]{KBH}. Both the result and its proof
generalize simply, and so we have
\[ \delta C^\mu = -\partial\left\{
    \delta\mu \act \frac{e^{\ad\mu}-1}{\ad\mu} \act RC^{-\mu}
  \right\}\act C^\mu,
\]
where we have written $\partial\{\text{mess}\}$ instead of
$\partial_{\text{mess}}$ because $\text{mess}$ is too big to fit as
a subscript. Hence by the chain rule and then by Equation~\eqref{eq:GammaODE},
\[ B_t
  = -\partial\left.\left\{
    \frac{d\Gamma_t(\lambda)}{dt} \act \frac{e^{\ad\mu}-1}{\ad\mu} \act RC^{-\mu}
  \right\}\right|_{\mu=\Gamma_t(\lambda)}
  = -\partial\left\{
    \lambda \act e^{-t\partial_\lambda} \act RC^{-\Gamma_t(\lambda)}
  \right\}
  = -\partial_{\lambda \act e^{-t\partial_\lambda} \act RC^{-\Gamma_t(\lambda)}}.
\]
On the other hand, computing $A_t$ is a simple differentiation, and
we get that $A_t=-\partial_\lambda$. Comparing with the line above,
we find that if Equation~\eqref{eq:Gamma} holds to degree $d$, then
Equation~\eqref{eq:AB} also holds to degree $d$. But then as we
noted,~\eqref{eq:Gamma} holds to degree $d+1$. As
Equation~\eqref{eq:Gamma} clearly holds at $t=0$, we find that it holds
to all orders. \qed
\end{proof}

\begin{comment} It is easier (though insufficient) to assume that there is
a solution $\Gamma_t(\lambda)$ to Equation~\eqref{eq:Gamma} and deduce that
it must satisfy the differential equation~\eqref{eq:GammaODE}: simply
differentiate~\eqref{eq:Gamma} with respect to $t$ and simplify as much as
you can allowing yourself to use~\eqref{eq:Gamma} as needed within the
simplification process. The result is~\eqref{eq:GammaODE}, and the steps
follow the computational steps of the above proof rather closely. The
actual proof is a bit harder because if we cannot assume~\eqref{eq:Gamma}
while deriving it, so we have to resort to an inductive process. 
\end{comment}

\begin{proposition} \label{prop:Lambda} As in the previous proposition,
let $\glosm{Lambdat}{\Lambda_t(\lambda)}$ be the (unique) solution of
\begin{equation} \label{eq:LambdaODE}
  \Lambda_0(\lambda)=0
  \qquad\text{and}\qquad
  \frac{d\Lambda_t(\lambda)}{dt} =
    \lambda \act e^{\partial_{\Lambda_t(\lambda)}}
    \act \frac{\ad_{tb}\Lambda_t(\lambda)}{e^{\ad_{tb}\Lambda_t(\lambda)}-1}.
\end{equation}
\begin{flalign} \label{eq:Lambda}
  & \text{Then} & C^{t\lambda}=e^{-\partial_{\Lambda_t(\lambda)}}. &&
\end{flalign}
\end{proposition}

The proof of this proposition is very similar and not even a tiny bit
nicer than the proof of the previous one. So we skip it and
instead include a computer verification.\footnoteC{
We verify that the computer-calculated $\Lambda_t(\lambda)$
satisfies the ODE in~\eqref{eq:LambdaODE} and then that
the operator equality~\eqref{eq:Lambda} holds, at least when
evaluated on ``our'' $\gamma$:

\shortdialoginclude{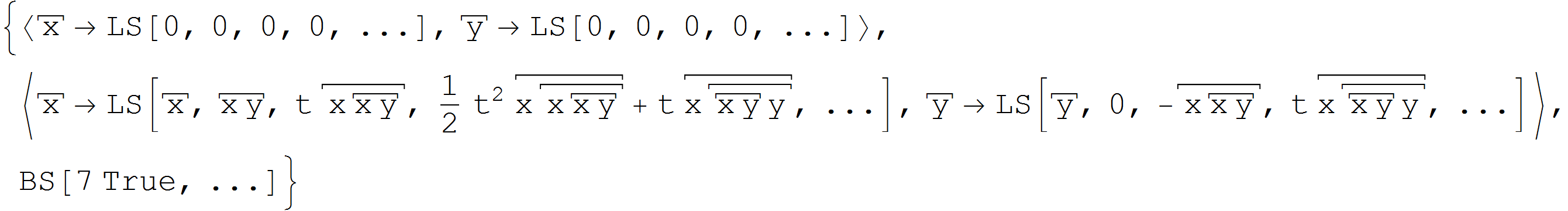}

\shortdialoginclude{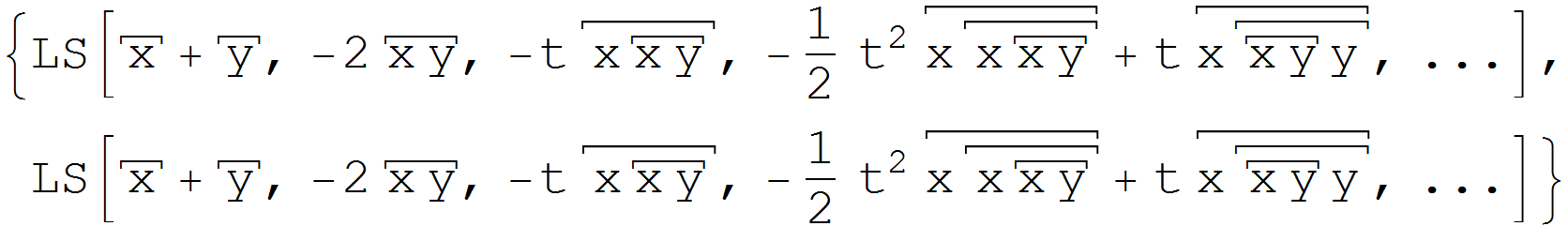}
}

As special cases, we denote $\Gamma_1(\lambda)$ by
$\glosm{Gamma}{\Gamma(\lambda)}$ and $\Lambda_1(\lambda)$ by
$\glosm{Lambda}{\Lambda(\lambda)}$.

One special case of $C^\lambda$ deserves to be named:

\begin{definition} \label{def:CRC}
(Compare~\cite[Section~\ref{KBH-subsec:FLSuccess}]{KBH}) Given $u\in
S$ and $\gamma\in\FL(S)$ let $\glosm{Cu}{C_u^{\gamma}}$ denote the
automorphism of $\FL(S)$ defined by mapping the generator $u$ to its
``conjugate'' $e^{\gamma}ue^{-\gamma}=e^{-\ad\gamma}(u)$ (this is simply
$C^\lambda$, where $\lambda$ is the length $1$ sequence $(u\to\gamma)$).
Let $\glosm{RCu}{RC_u^{-\gamma}}$ be the inverse of $C_u^{\gamma}$
(which is {\em not} $C_u^{-\gamma}$).\footnoteC{Just testing:

\shortdialoginclude{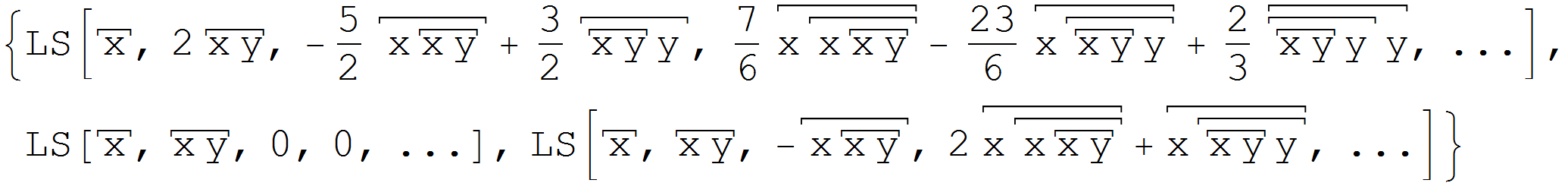}
}
\end{definition}

\Needspace{2in}
Last we define/recall a number of functionals $\FL(S)\to\CW(S)$:

\vskip 1mm
{
\makeatletter\def\thm@space@setup{%
  \thm@preskip=0cm plus 0cm minus 0cm %\thm@postskip=\thm@preskip
}\makeatother
\parpic[r]{\input{figs/tru.pstex_t}}
\begin{definition} \label{def:tru}
For $u\in S$ we let $\glosm{tru}{\tr_u}\colon\FL(S)\to\CW(S)$
be the sum of all ways of connecting the head of $\gamma$ to any
of its $u$-labelled tails and regarding the result as an element
of $\CW(\FL(S))\to\CW(S)$.  The example on the right corresponds
to the specific computation $\tr_u[[v,u],u] = \wideparen{[v,u]} +
\wideparen{v(-u)} = -\wideparen{uv}$\footnoteC{In computer talk, and
using a temporary value for $\gamma$, so as not to interfere with its
existing value:

\shortdialoginclude{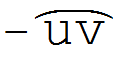}
}
\end{definition}
}

\vskip 1mm
{
\makeatletter\def\thm@space@setup{%
  \thm@preskip=0cm plus 0cm minus 0cm %\thm@postskip=\thm@preskip
}\makeatother
\parpic[r]{\input{figs/divu.pstex_t}}
\begin{definition} \label{def:J}
(Compare~\cite[Section~\ref{KBH-subsec:divJ}]{KBH})
For $u\in S$ we let $\glosm{atdivu}{\atdiv_u}\colon\FL(S)\to\CW(S)$
be the functional defined schematically by the picture on the right,
which corresponds to the specific computation $\atdiv_u[[v,u],u] =
\wideparen{u[v,u]} + \wideparen{uv(-u)} = -\wideparen{uuv}$\footnoteC{In
computer talk:

\shortdialoginclude{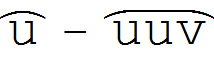}
}
(more details in~\cite{KBH}). Given also $\gamma\in\FL(S)$, set
\[ \glosm{Ju}{J_u}(\gamma) := \int_0^1ds\,\atdiv_u\!\left(
    \gamma \sslash RC_u^{s\gamma}
  \right) \sslash C_u^{-s\gamma}.\footnotemarkC
\]
\end{definition}
}

\footnotetextC{We quote the implementation of $J$ in
\href{\web/FreeLie.m}{\tt FreeLie.m} (\href{\web/FreeLie.m}{\tt FL}) and,
reverting to the ``old'' $\gamma$, compute $J_1(\gamma)$:
\FLQuote{JDef}

\shortdialoginclude{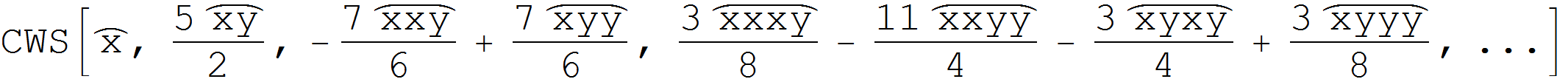}
}

\picskip{0}
\begin{definition} \label{def:j} Let
$\glosm{atdiv}{\atdiv}\colon\FL(S)\to\CW(S)$ be the Alekseev-Torossian
``divergence'' functional, as in~\cite[Section~5.1]{AT}, but extended
by $0$ on $A_S$.  In our language, $\atdiv\lambda=\sum_{u\in
S}\atdiv_u\lambda$.  Let $\glosm{j}{j}\colon\FL(S)\to\CW(S)$ is
the Alekseev-Torossian ``logarithm of the Jacobian'':
$j(\lambda) = \frac{e^{\partial_\lambda}-1}{\partial_\lambda}
(\atdiv\lambda)$.\footnoteC{A quote of the computer-definition, and then
$\atdiv\lambda$ and $j(\lambda)$, computed to degree 5:
\FLQuote{divDef}

\shortdialoginclude{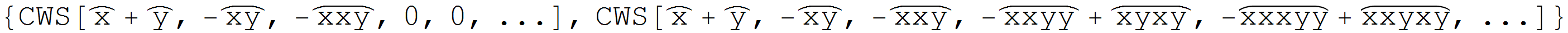}
}
\end{definition}

Alekseev and Torossian prove in~\cite{AT} that $j$ is the unique functional
$j\colon\FL(S)\to\CW(S)$ satisfying the ``cocycle condition''
$j\left(\glosm{BCHtb}{\BCH_{tb}}(\lambda_1,\lambda_2)\right) =
j(\lambda_1)+e^{\partial_{\lambda_1}}j(\lambda_2)$, where $\BCH_{tb}$
stands for the $\BCH$ formula using the tangential bracket
$[\cdot,\cdot]_{tb}$ on $\FL(S)^S$:
\[ \BCH_{tb}(\lambda_1,\lambda_2)
  = \lambda_1+\lambda_1+\frac12[\lambda_1,\lambda_2]_{tb}
    +\ldots,
\]
and the ``initial condition''
$\frac{\partial}{\partial\epsilon}j(\epsilon\lambda) =
\atdiv\lambda$.\footnoteC{We verify the cocycle condition and the initial
condition. For the latter, we first declare $\epsilon$ to be ``an
infinitesimal'' by declaring that $\epsilon^2=0$, and then we verify that
$j(\epsilon\lambda) = \epsilon\atdiv\lambda$:

\shortdialoginclude{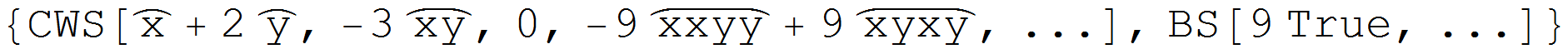}

\shortdialoginclude{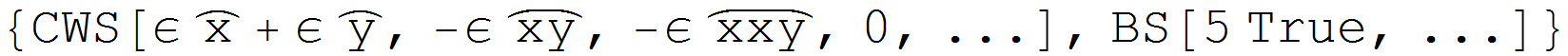}
}

\draftcut
\subsection{The lower-interlaced presentation $E_l$ of
$\calA^w_{\exp}$} \label{subsec:AT}

For a finite set $S$ let $\glosm{TWl}{\TW_l}(S)$ be set-theoretically
the same as $\TW(S) = \FL(S)^S\times\CW(S)$ --- we only add the ``$l$''
subscript to emphasize that $\TW_l$ carries an algebraic structure,
and that it is different from the algebraic structure on $\TW_s$,
which we will study later. Elements of $\TW_l(S)$ are ordered pairs
$\glosm{parenl}{(\lambda;\,\omega)_l}$, where $\lambda\in\FL(S)^S$,
$\omega\in\CW(S)$, and the subscript $l$ is there only to remind us of
the context.

\begin{samepage}
Set $\glosm{ElDef}{\null}$
\[ \glosm{El}{E_l}(\lambda;\,\omega)_l\coloneqq \exp(\yellowm{l}\lambda)
      \ast\exp(\yellowm{\iota}\omega)
    \in\calA^w_{\exp}(S),
  \qquad\left(\parbox{1.8in}{\centering
    ``$E_l$'' for ``\underline{E}xponentiation after using $\underline{l}$''
  }\right)
\]
\end{samepage}
\begin{samepage}
\parpic[r]{\parbox{1.75in}{
  \centering{\input{figs/El.pstex_t}}
  \vspace{-3mm}
  \begin{figcap} \label{fig:El} $E_l(\lambda;\,\omega)_l$. \end{figcap}
  \vspace{3mm}
}}
\noindent where $l\colon\FL(S)^S=A_S\oplus\tder_S\to\calA^w(S)$
is the ``lower'' Lie embedding\footnotemarkT\
of trees into ${\calA^w(S)}$
(see~\cite[Section~\ref{2-subsec:ATSpaces}]{WKO2}), where $\iota$
is the obvious inclusion of wheels ($=\CW(S)=\attr_S$) into
${\calA^w(S)}$, and where exponentiation is taken using the
stacking product~\eqref{eq:TubeProduct} of $\calA^w(S)$. A
pictorial representation of $E_l(\lambda;\,\omega)_l$ appears on the right:
Reading from the bottom up, we see ``exponentially many'' copies of
$\lambda$ (meaning, a sum over $n$ of $n$ copies with coefficient $1/n!$).
Each $\lambda$ is a linear combination of trees with one head and many
tails, which are attached to the strands in $T$ with the head below the
tails. Each copy of $\lambda$ appears on the right as a gray 
``wizard's cap'' whose tip corresponds to the head of $\lambda$,
and is therefore tipped downward. Above $\exp(l\lambda)$ is our symbolic
representation of $\exp(\iota\omega)$.

Figure~\ref{fig:El} also explains the name ``interlaced'' for this
presentation, for in it heads and tails are interlaced along the strands
of $S$ (contrast with $E_s$ in Figure~\ref{fig:Es} and with $E_f$
in Figure~\ref{fig:Ef}).
\end{samepage}

It follows from the results
of~\cite[Section~\ref{2-subsec:ATSpaces}]{WKO2} that the map
$E_l\colon\TW_l(S)\to\calA^w_{\exp}(S)$ is a set-theoretic
bijection. Hence the operations of Definition~\ref{def:Operations}
induce corresponding operations on $\TW_l(S)$. We list these within the
(long!) definition-proposition below.

\footnotetextT{We could have equally well used the ``upper'' Lie embedding
$u$, setting $\glosm{Eu}{E_u}\glosm{parenu}{(\lambda;\,\omega)_u}
\coloneqq \exp(\iota\omega)\exp(\glosm{u}{u}\lambda)$, with only minor
modifications to the formulas that follow.}

\begin{defprop} \label{dp:ElOps}
The bijection $E_l$ intertwines the 
operations defined below with the operations in
Definition~\ref{def:Operations}:\footnoteC{We cannot verify
Definition-Proposition~\ref{dp:ElOps} {\it per se} on the computer, as we have
no direct computer implementation of $\calA^w$. Indeed, the whole point
of this paper is to provide an implementation of $\calA^w$ by means of
$E_l$ (and later, $E_s$ and $E_f$). Instead, we verify below that many
properties of operations on $\calA^w$ (the associativity of the stacking
product, etc.)  indeed hold for their $E_l$ implementations. We start
by setting the values of some ``sample'' elements on which we will run
our tests (note that on the computer we represent $(\lambda;\,\omega)_l$
as {\tt El[$\lambda$,$\omega$]}):

\dialoginclude{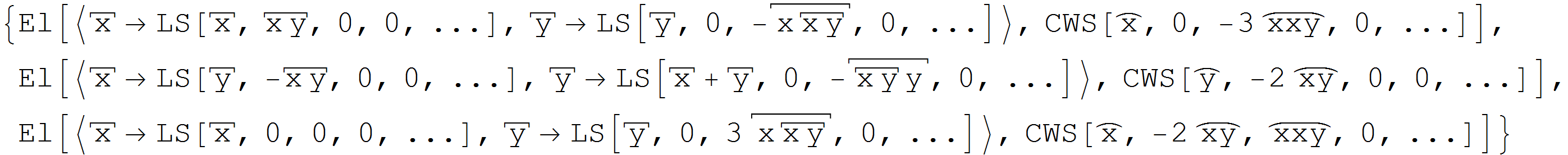}
}

\begin{enumerate}[leftmargin=*,labelindent=0pt]

\item If $S_1\cap S_2=\emptyset$ and $(\lambda_i;\,\omega_i)_l\in\TW_l(S_i)$,
$\glosm{lsqcup}{\null}$
\begin{equation} \label{eq:ElCup}
  (\lambda_1;\,\omega_1)_l(\lambda_2;\,\omega_2)_l
  = (\lambda_1;\,\omega_1)_l\yellowm{\sqcup}(\lambda_2;\,\omega_2)_l
  \coloneqq (\lambda_1\sqcup\lambda_2;\,\omega_1+\omega_2)_l,
\end{equation}
where $\sqcup\colon\FL(S_1)^{S_1}\times\FL(S_2)^{S_2}\to\FL(S_1\sqcup
S_2)^{S_1\sqcup S_2}$ is the union operation of functions (or, in computer
speak, the concatenation of associative arrays) followed by
the inclusions $\FL(S_i)\to\FL(S_1\sqcup S_2)$, and $\omega_1+\omega_2$
is defined using the inclusions $\CW(S_i)\to\CW(S_1\sqcup S_2)$.

\item If $(\lambda_i;\,\omega_i)_l\in\TW_l(S)$,$\glosm{last}{\null}$
\begin{equation} \label{eq:ElProduct}
  (\lambda_1;\,\omega_1)_l\yellowm{\ast}(\lambda_2;\,\omega_2)_l
  \coloneqq (
    \BCH_{tb}(\lambda_1,\lambda_2);\,
    e^{-\partial_{\lambda_2}}(\omega_1)+\omega_2
  )_l.\footnotemarkC
\end{equation}

\footnotetextC{We quote the $E_l$ implementation of the stacking product
from \href{\web/AwCalculus.m}{\tt AwCalculus.m}
(\href{\web/AwCalculus.m}{\tt AC}) and verify that it is associative,
at least to degree $8$:
\ACQuote{ElStackingDef}

\shortdialoginclude{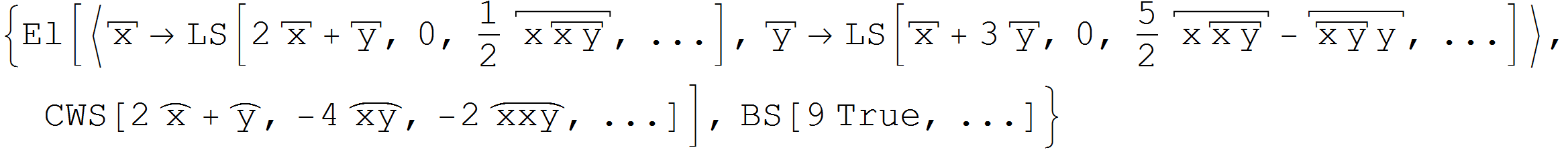}
}

\item If $(\lambda;\,\omega)_l\in\TW_l(S)$ and $a\in S$,$\glosm{ldeta}{\null}$
\begin{equation} \label{eq:ElEta}
  (\lambda;\,\omega)_l\act \yellowm{d\eta^a}
  \coloneqq ((\lambda\yellowm{\remove}a)\act(a\to 0);\,\omega\act(a\to 0))_l,
\end{equation}
where $\lambda\remove a$ denotes the function $\lambda$ with the element
$a$ removed from its domain (in computer talk, ``remove the key $a$''),
and $(a\to 0)$ denotes the substitution $a=0$, which is defined on both
$\FL$ and $\CW$ and maps $\FL(S)\to\FL(S\remove a)$ and
$\CW(S)\to\CW(S\remove a)$.\footnoteC{Example:

\shortdialoginclude{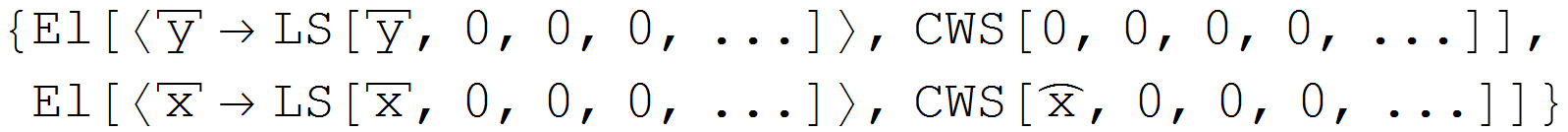}
}

\item For a single $a\in S$, I don't know a simple description
of the operation $\dA^a$ in $E_l$ language\footnoteT{\label{foot:notsimple}
  A not-so-simple description would be to use the language of the factored
  presentation of Section~\ref{subsec:Ef}, converting back and forth
  using the results of Section~\ref{subsec:Conversion}.
}.
Yet the composition
$\glosm{ldA}{\dA}\coloneqq\yellowm{\dA^S}\coloneqq\prod_{a\in S}\dA^a$
is manageable: ($j$ is defined in Definition~\ref{def:j})
\begin{equation} \label{eq:ElA}
  (\lambda;\,\omega)_l\act \dA^S \coloneqq
  (-\lambda;\,e^{\partial_\lambda}(\omega)-j(\lambda))_l.\footnotemarkC
\end{equation}

\footnotetextC{We quote the computer-definition of $\dA$, compute an
example, verify that $\dA$ is an involution, and then that it is an
anti-homomorphism relative to the stacking product:
\ACQuote{EldA}

\shortdialoginclude{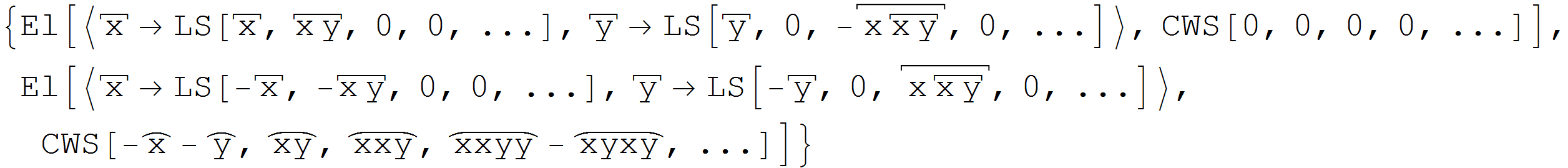}

\shortdialoginclude{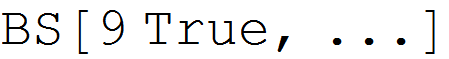}

\shortdialoginclude{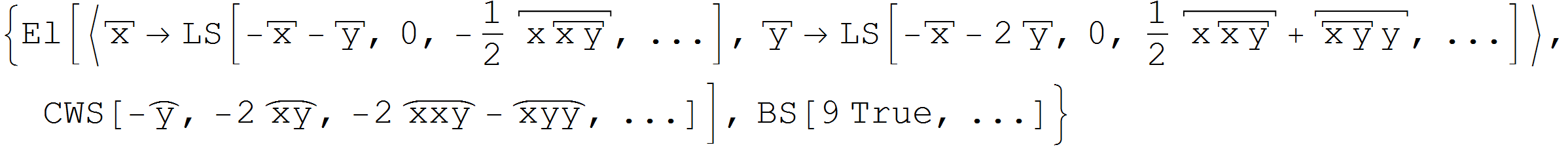}
}

\addtocounter{footnoteT}{-1}%\addtocounter{HfootnoteT}{-1}
\item For a single $a\in S$, I don't know a simple description
of the operation $\dS^a$ in $E_l$
language\footnotemarkT. Yet the composition
$\glosm{ldS}{\dS}\coloneqq\yellowm{\dS^S}\coloneqq\prod_{a\in S}\dS^a$
is manageable:
\begin{equation} \label{eq:ElS}
  (\lambda;\,\omega)_l\act \dS^S
  \coloneqq (-\lambda\act(-1)^{\deg};\,
    (e^{\partial_\lambda}(\omega)-j(\lambda))\act(-1)^{\deg})_l.\footnotemarkC
\end{equation}

\footnotetextC{An example:

\shortdialoginclude{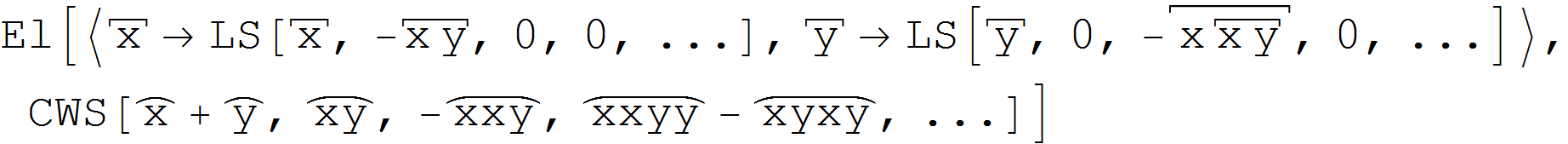}
}

\addtocounter{footnoteT}{-1}%\addtocounter{HfootnoteT}{-1}
\item I don't know a simple description
of the operation $dm^{ab}_c$ in $E_l$ language\footnotemarkT. Yet note that
Equation~\eqref{eq:multiplem} implies that ``applying $dm$ to all strands''
is manageable, being the stacking product described
in~\eqref{eq:ElProduct}.

\item We have\glosm{ldDelta}{$\null$}
\begin{equation} \label{eq:ElDelta}
  (\lambda;\,\omega)_l\act \yellowm{d\Delta}^a_{bc} \coloneqq (
    (\lambda\remove a)\sqcup(b\to\lambda_a,\,c\to\lambda_a)
      \act (a\to b+c);\,
    \omega \act (a\to b+c)
  )_l,
\end{equation}
where $(a\to b+c)$ denotes the obvious replacement of the generator $a$
with the sum $b+c$. It represents morphisms $\FL(S)\to\FL((S\remove
a)\sqcup \{b,c\})$, $\FL(S)^H\to\FL((S\remove a)\sqcup \{b,c\})^H$ (for any
set $H$), and  $\CW(S)\to\CW((S\remove a)\sqcup \{b,c\})$.\footnoteC{The
computer-definition, an example, and then a verification that $d\Delta$
is homomorphism relative to
the stacking product:
\ACQuote{EldDelta}

\shortdialoginclude{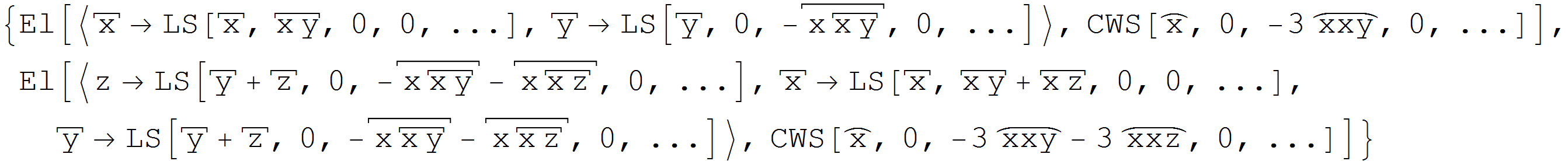}

\shortdialoginclude{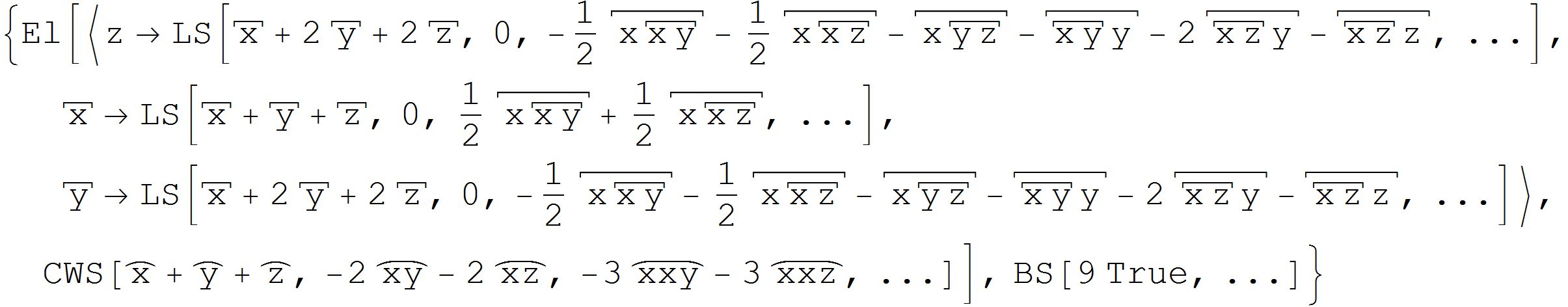}
}

\item We have\glosm{ldsigma}{$\null$}
\begin{equation} \label{eq:ElSigma}
  (\lambda;\,\omega)_l\act \yellowm{d\sigma}^a_b \coloneqq (
    ((\lambda\remove a)\sqcup(b\to\lambda_a))\act (a\to b);\,
    \omega\act (a\to b)
  )_l,
\end{equation}
where $(a\to b)$ denotes the obvious ``generator renaming''
morphisms $\FL(S)\to\FL((S\remove a)\sqcup b)$, $\FL(S)^H\to\FL((S\remove
a)\sqcup b)^H$ (for any set $H$), and $\CW(S)\to\CW((S\remove a)\sqcup b)$.

\end{enumerate}
\end{defprop}

\begin{proof} Equations \eqref{eq:ElCup}, \eqref{eq:ElEta},
\eqref{eq:ElDelta}, and \eqref{eq:ElSigma} are trivial and
were stated only to introduce notation. The tree-level
part of Equation~\eqref{eq:ElProduct} follows from the
fact that $l$ is a morphism of Lie algebras (see within the
proof of~\cite[Proposition~\ref{2-prop:Pnses}]{WKO2}). The
wheels part of Equation~\eqref{eq:ElProduct} follows
from~\cite[Remark~\ref{2-rem:tderontr}]{WKO2}. Equation~\eqref{eq:ElA}
follows from the observation that $\dA^S$ is the adjoint map
$\ast$ of~\cite[Definition~\ref{2-def:Adjoint}]{WKO2} and then
from~\cite[Proposition~\ref{2-prop:Jandj}]{WKO2}. Equation~\eqref{eq:ElS}
is the easily-established fact that on $\calA^w$,
$\dS^S=(-1)^{\deg}\dA^S$. \qed
\end{proof}

  \parpic[l]{\begin{picture}(0,0)%
\includegraphics{figs/Topology.pstex}%
\end{picture}%
%
%  pstex_opts: -m 0.3333 
%
\setlength{\unitlength}{1316sp}%
\begingroup\makeatletter\ifx\SetFigFont\undefined%
\gdef\SetFigFont#1#2#3#4#5{%
  \reset@font\fontsize{#1}{#2pt}%
  \fontfamily{#3}\fontseries{#4}\fontshape{#5}%
  \selectfont}%
\fi\endgroup%
\begin{picture}(1052,1180)(0,-190)
\end{picture}%
  \hspace{-6pt}}%
  \noindent%
 Note that the absence of simple descriptions of $dA^a$, $\dS^a$,
and $dm^{ab}_c$ in the $E_l$ language is fatal for its applicability to
knot theory, as these operations are needed within the computation of 
knot and tangle invariants. See Section~\ref{subsec:TangleInvariants}.

\refstepcounter{theorem}
  \parpic[l]{\begin{picture}(0,0)%
\includegraphics{figs/AT.pstex}%
\end{picture}%
%
%  pstex_opts: -m 0.36 
%
\setlength{\unitlength}{1421sp}%
\begingroup\makeatletter\ifx\SetFigFont\undefined%
\gdef\SetFigFont#1#2#3#4#5{%
  \reset@font\fontsize{#1}{#2pt}%
  \fontfamily{#3}\fontseries{#4}\fontshape{#5}%
  \selectfont}%
\fi\endgroup%
\begin{picture}(1041,1009)(118,-19)
\end{picture}%
  \hspace{-8pt}}%
  \noindent%
 {\em Comment}~\thetheorem. \label{com:ElAT}
Let $\glosm{piT}{\pi_T}\colon\TW(S)\to\FL(S)^S$ denote the projection
onto the first factor (``trees'') of $\TW(S)=\FL(S)^S\times\CW(S)$,
and recall that up to a minor central factor, $(\FL(S)^S,\,tb)$
is $\tder_S$. Recall also that $\tder_S$ is the Lie algebra of
$\TAut_S$, and that elements of $\tder_S$ represent elements of
$\TAut_S$ by exponentiation. With this in mind, the tree part of
Equation~\eqref{eq:ElProduct} becomes the product of $\TAut_S$.
In other words, the diagram
\[ \xymatrix{
  \TW_l(S)\times\TW_l(S) \ar[r]^-\ast
    \ar[d]_{(\pi_T\act\exp)\times(\pi_T\act\exp)} &
  \TW_l(S) \ar[d]^{\pi_T\act\exp} \\
  \TAut_S\times\TAut_S \ar[r]^-{\text{mult.}} &
  \TAut_S
} \]
is commutative. Hence the $E_l$ presentation is valuable for~\cite{AT} as
many of the~\cite{AT} equations involve the group structure of $\TAut_S$.

\draftcut
\parpic[r]{\begin{picture}(0,0)%
\includegraphics{figs/AwDiagram.pstex}%
\end{picture}%
%
%  pstex_opts: -m 0.6 
%
\setlength{\unitlength}{2368sp}%
\begingroup\makeatletter\ifx\SetFigFont\undefined%
\gdef\SetFigFont#1#2#3#4#5{%
  \reset@font\fontsize{#1}{#2pt}%
  \fontfamily{#3}\fontseries{#4}\fontshape{#5}%
  \selectfont}%
\fi\endgroup%
\begin{picture}(2526,3034)(-62,-1594)
\end{picture}%
}
\subsection{The factored presentation $E_f$ of $\calA^w_{\exp}$ and its
stronger precursor $E_s$} \label{subsec:Ef}
Following~\cite{KBH}, in the ``factored'' presentation $E_f$
of $\calA^w_{\exp}$ arrow heads are treated separately from arrow
tails in diagrams such as the one on the right.  This presentation of
$\calA^w_{\exp}$ is more complicated than the previous one, yet it is also
more powerful, and in some sense, it is made of simpler ingredients. We
first enlarge the collection of spaces $\{\calA^w(S)\}$ to a somewhat
bigger collection $\{\calA^w(H;T)\}$ on which a larger class of operations
act. The new operations are more ``atomic'' than the old ones, in the
sense that each of the operations of Definition~\ref{def:Operations} is
a composition of 2-3 of the new operations. The advantage is that the
new operations all have reasonably simple descriptions as operations
on the group-like subsets $\{\calA^w_{\exp}(H;T)\}$ (the ``split''
presentation $E_s$ below), and hence even the few operations whose description
in the $E_l$ presentation was omitted in Definition-Proposition~\ref{dp:ElOps}
can be fully described and computed in the $E_f$ presentation.

A sketch of our route is as follows: In Section~\ref{sssec:Family},
right below, we describe the spaces $\{\calA^w(H;T)\}$. In
Section~\ref{sssec:AHTOperations} we describe the zoo of operations acting
on $\{\calA^w(H;T)\}$. Section~\ref{sssec:AHTExp} is the tofu of the matter
--- we describe the operations of the previous section in
terms of spaces $\{\TW_s(H;T)\}$ of trees and wheels, whose elements are in a
bijection $E_s$ with the group like elements of
$\{\calA^w(H;T)\}$. Finally in Section~\ref{sssec:Inclusion} we explain
how the system of spaces $\{\calA^w(S)\}$ includes into the system
$\{\calA^w(H;T)\}$ and how the operations of the former are expressed
in terms of the latter, concluding the description of $E_f$.

\subsubsection{The family $\{\calA^w(H;T)\}$} \label{sssec:Family}

Let $\glosm{H}{H} = \{\glosm{h}{h_1},h_2,\ldots\}$ be some finite set of
``head labels'' and let $\glosm{T}{T} = \{\glosm{t}{t_1},t_2,\ldots\}$
be some finite set of ``tail labels'' (these sets need not be of the
same cardinality). Let $\glosm{calAwHT}{\calA^w(H;T)}$ be
$\calA^w({H\sqcup T})$\footnotemarkT\ moded out by the following
further relations:

\footnotetextT{\label{foot:BruteDisjoint} We will often use sets of labels
$H$ and $T$ that are {\em not} disjoint. The notation
``$H\glosm{Brutesqcup}{\sqcup}T$'' stands for the union of $H$ and
$T$, made disjoint by brute force; for example, by setting $H\sqcup
T\coloneqq(\{h\}\times H)\cup(\{t\}\times T)$, where $h$ and $t$ are two
distinct labels chosen in advance to indicate ``heads'' and ``tails''. In
practise we will keep referring to the images of the elements of $H$
within $H\sqcup T$ as $h_i$ rather than $(h,h_i)$, and likewise for
the $t_i$'s.  We will mostly avoid the confusion that may arise when
$H\cap T\neq\emptyset$ by labelling operations as ``head operations'' which
will always refer to labels in $H\hookrightarrow H\sqcup T$ or as ``tail
operations'', when referring to labels in $T\hookrightarrow H\sqcup T$.}

\parbox[t]{4in}{\begin{myitemize}
\item If an arrow tail lands anywhere on a head strand ($\ast1$ on the
right), the whole diagram is zero.
\item The $\glosm{CP}{CP}$ relation: If an arrow head is the lowest vertex
on a tail strand ($\ast2$ on the right), the whole diagram is zero. (As
on the right, we indicate the bottom ends of tail strands with bullets
``$\bullet$'').
\end{myitemize}}
\hfill\imagetop{\input{figs/HeadTailRels.pstex_t}}

\Needspace{2in}
\vskip 1mm
{
\makeatletter\def\thm@space@setup{%
  \thm@preskip=0cm plus 0cm minus 0cm %\thm@postskip=\thm@preskip
}\makeatother
\parpic[r]{\begin{picture}(0,0)%
\includegraphics{figs/TypicalAHT.pstex}%
\end{picture}%
%
%  pstex_opts: -m 0.9 
%
\setlength{\unitlength}{3552sp}%
\begingroup\makeatletter\ifx\SetFigFont\undefined%
\gdef\SetFigFont#1#2#3#4#5{%
  \reset@font\fontsize{#1}{#2pt}%
  \fontfamily{#3}\fontseries{#4}\fontshape{#5}%
  \selectfont}%
\fi\endgroup%
\begin{picture}(1074,1651)(64,-800)
\put(976,-736){\makebox(0,0)[b]{\smash{{\SetFigFont{11}{13.2}{\rmdefault}{\mddefault}{\updefault}{\color[rgb]{0,0,0}$T$}%
}}}}
\put(226,-736){\makebox(0,0)[b]{\smash{{\SetFigFont{11}{13.2}{\rmdefault}{\mddefault}{\updefault}{\color[rgb]{0,0,0}$H$}%
}}}}
\end{picture}%
}
\begin{comment} \label{com:PureForm}
Using these two relations one may show that $\calA^w(H;T)$
is isomorphic to the set of arrow diagrams in which only arrow heads
land on the head strands (obvious, by the first relation) and in which
only arrow tails meet the tail strands (use $\aSTU_2$ to slide any arrow
head on a tail strand until it's near the bottom, then use the second
relation; see also Comment~\ref{com:SortedForm}), still modulo $\aAS$,
$\aIHX$, $\aSTU_1$ and $TC$. Thus a typical element of $\calA^w(H;T)$
is shown on the right.
\end{comment}
}

  \parpic[l]{\begin{picture}(0,0)%
\includegraphics{figs/Topology.pstex}%
\end{picture}%
%
%  pstex_opts: -m 0.3333 
%
\setlength{\unitlength}{1316sp}%
\begingroup\makeatletter\ifx\SetFigFont\undefined%
\gdef\SetFigFont#1#2#3#4#5{%
  \reset@font\fontsize{#1}{#2pt}%
  \fontfamily{#3}\fontseries{#4}\fontshape{#5}%
  \selectfont}%
\fi\endgroup%
\begin{picture}(1052,1180)(0,-190)
\end{picture}%
  \hspace{-6pt}}%
  \noindent%
 In topology (see~\cite{KBH}), head strands correspond to
``hoops'', or based knotted circles, and tail strands correspond to
balloons, or based knotted spheres. The two relations and the isomorphism
above are also meaningful~\cite{KBH}.

  \parpic[l]{\begin{picture}(0,0)%
\includegraphics{figs/Lie.pstex}%
\end{picture}%
%
%  pstex_opts: -m 0.36 
%
\setlength{\unitlength}{1421sp}%
\begingroup\makeatletter\ifx\SetFigFont\undefined%
\gdef\SetFigFont#1#2#3#4#5{%
  \reset@font\fontsize{#1}{#2pt}%
  \fontfamily{#3}\fontseries{#4}\fontshape{#5}%
  \selectfont}%
\fi\endgroup%
\begin{picture}(960,990)(46,0)
\end{picture}%
  \hspace{-6pt}}%
  \noindent%
 In Lie theory head strands represent $\calU(\frakg)$ and tail strands
represent the (right) Verma module
$\calU(I\frakg)/\frakg\calU(I\frakg) \simeq \calU(\frakg^\ast) \simeq
\calS(\frakg^\ast)$. The evaluation $\frakg^\ast\to 0$ induces a surjection
of $\calU(I\frakg)$ onto the first of these spaces whose kernel is ``any
word containing a letter in $\frakg^\ast$'', explaining the first relation
above. The second relation is the definition of the Verma module. 

\Needspace{2cm}
\subsubsection{Operations on $\{\calA^w(H;T)\}$.}
\label{sssec:AHTOperations}

\begin{definition} \label{def:AHTOperations}
Just as in Definition~\ref{def:Operations}, there are several
operations that are defined on $\calA^w(H;T)$. In brief, these are:

\begin{enumerate}[leftmargin=*,labelindent=0pt]

\item A union operation $\glosm{htsqcup}{\sqcup}\colon \calA^w(H_1;T_1)
\otimes \calA^w(H_2;T_2) \to \calA^w(H_1\sqcup H_2;T_1\sqcup T_2)$,
defined when $H_1\cap H_2=T_1\cap T_2=\emptyset$, with obvious topological
(compare with ``$\ast$'' of~\cite[Figure~\ref{KBH-fig:ConnectedSum}]{KBH})
and Lie theoretic meanings. (The symbol $\sqcup$ is sometimes omitted:
$D_1D_2\coloneqq D_1\sqcup D_2$).

\item A ``stacking'' product $\glosm{jail}{\#}$ can be defined on
$\calA^w(H;T)$ by stitching all pairs of equally-labelled head strands
and then merging all pairs of equally-labelled tail strands in a
pair of diagrams $D_1,D_2\in\calA^w(H;T)$. The ``merging'' of tail
strands is described in more detail as the operation $tm$ below. In
fact, it may be better to define $\#$ using a formula similar to
Equation~\eqref{eq:multiplem} and the operations $hm$, $tm$, $h\sigma$,
and $t\sigma$ defined below:
\begin{equation} \label{eq:AHTStacking}
  D_1\#D_2 = \left(
    D_1\sqcup\left(D_2
      \act\prod_{x\in H}h\sigma^x_{\bar{x}}
      \act\prod_{u\in T}t\sigma^u_{\bar{u}}
    \right)
  \right)
    \act\prod_{x\in H}hm^{x\bar{x}}_x
    \act\prod_{u\in T}tm^{u\bar{u}}_u.
\end{equation}

  \parpic[l]{\begin{picture}(0,0)%
\includegraphics{figs/Topology.pstex}%
\end{picture}%
%
%  pstex_opts: -m 0.3333 
%
\setlength{\unitlength}{1316sp}%
\begingroup\makeatletter\ifx\SetFigFont\undefined%
\gdef\SetFigFont#1#2#3#4#5{%
  \reset@font\fontsize{#1}{#2pt}%
  \fontfamily{#3}\fontseries{#4}\fontshape{#5}%
  \selectfont}%
\fi\endgroup%
\begin{picture}(1052,1180)(0,-190)
\end{picture}%
\,\begin{picture}(0,0)%
\includegraphics{figs/Lie.pstex}%
\end{picture}%
%
%  pstex_opts: -m 0.36 
%
\setlength{\unitlength}{1421sp}%
\begingroup\makeatletter\ifx\SetFigFont\undefined%
\gdef\SetFigFont#1#2#3#4#5{%
  \reset@font\fontsize{#1}{#2pt}%
  \fontfamily{#3}\fontseries{#4}\fontshape{#5}%
  \selectfont}%
\fi\endgroup%
\begin{picture}(960,990)(46,0)
\end{picture}%
  \hspace{-6pt}}%
  \noindent%
 In topology, $\#$ is the stitching of hoops followed by the
merging of balloons; this is not the same as the stitching of knotted
tubes. In Lie theory, $\#$ corresponds to the componentwise product of
$\calU(\frakg)^{\otimes H}\otimes\calS(\frakg^\ast)^{\otimes T}$. Even when
$H$ and $T$ are both singletons, this is not the same as the product of
$\calU(I\frakg)$, even though linearly
$\calU(I\frakg)\simeq\calU(\frakg)\otimes\calS(\frakg^\ast)$.

\setcounter{tunnel}{\theenumi}
\end{enumerate}
\begin{enumerate}[leftmargin=*,labelindent=0pt]\setcounter{enumi}{\thetunnel}

\item If $\glosm{xinH}{x}\in H$ and $\glosm{uinT}{u}\in T$, the operations
$\glosm{heta}{h\eta}^x$ and $\glosm{teta}{t\eta}^u$ drop the head-strand
$x$ or the tail-strand $u$ similarly to the operation $d\eta^a$ of
Definition~\ref{def:Operations}.

\item $\glosm{hA}{hA}^x$ reverses the head-strand $x$ while multiplying
by a $(-1)$ factor for every arrow head on $x$. $\glosm{tA}{tA}^u$
is the identity.

\item $\glosm{hS}{hS}^x=hA^x$ while $\glosm{tS}{tS}^u$ multiplies by a
factor of $(-1)$ for every arrow tail on $u$ (by $TC$, there's no need
to reverse $u$).

\item The operation $\glosm{hm}{hm}^{xy}_z$ is defined similarly
to $dm^{ab}_c$ of Definition~\ref{def:Operations}. Likewise for
$\glosm{tm}{tm}^{uv}_w$, except in this case, the tail-strands $u$
and $v$ must first be cleared of all arrow-heads using the process of
Comment~\ref{com:PureForm}. Once $u$ and $v$ carry only arrow-tails,
all these tail can be put on a new tail-strand $w$ in some arbitrary
order (which doesn't matter, by $TC$). Note that $tm^{uv}_w=tm^{vu}_w$,
so $tm$ is ``meta-commutative''.

  \parpic[l]{\begin{picture}(0,0)%
\includegraphics{figs/Topology.pstex}%
\end{picture}%
%
%  pstex_opts: -m 0.3333 
%
\setlength{\unitlength}{1316sp}%
\begingroup\makeatletter\ifx\SetFigFont\undefined%
\gdef\SetFigFont#1#2#3#4#5{%
  \reset@font\fontsize{#1}{#2pt}%
  \fontfamily{#3}\fontseries{#4}\fontshape{#5}%
  \selectfont}%
\fi\endgroup%
\begin{picture}(1052,1180)(0,-190)
\end{picture}%
  \hspace{-6pt}}%
  \noindent%
 In topology, $tm^{uv}_w$ is the ``merging of balloons''
operation of \cite[Section~\ref{KBH-subsec:MMAOperations}]{KBH}, which
in itself is analogues to the (commutative) multiplication of $\pi_2$.

  \parpic[l]{\begin{picture}(0,0)%
\includegraphics{figs/Lie.pstex}%
\end{picture}%
%
%  pstex_opts: -m 0.36 
%
\setlength{\unitlength}{1421sp}%
\begingroup\makeatletter\ifx\SetFigFont\undefined%
\gdef\SetFigFont#1#2#3#4#5{%
  \reset@font\fontsize{#1}{#2pt}%
  \fontfamily{#3}\fontseries{#4}\fontshape{#5}%
  \selectfont}%
\fi\endgroup%
\begin{picture}(960,990)(46,0)
\end{picture}%
  \hspace{-6pt}}%
  \noindent%
 In Lie theory, $tm^{uv}_w$ is the product of $\calS(\frakg^\ast)$.
Note that tail strands more closely represent the Verma module
$\calU(I\frakg)/\frakg\calU(I\frakg)$ whose isomorphism with
$\calS(\frakg^\ast)$ involves ``sliding all $\frakg$-letters in a
$\calU(I\frakg)$-word to the left and then cancelling them''. This is
analogous to the process of cancelling arrow-heads which is a
pre-requisite to the definition of $tm^{uv}_w$.

\setcounter{tunnel}{\theenumi}
\end{enumerate}
\begin{enumerate}[leftmargin=*,labelindent=0pt]\setcounter{enumi}{\thetunnel}

\item $\glosm{hDelta}{h\Delta}^x_{yz}$ and
$\glosm{tDelta}{t\Delta}^u_{vw}$ are defined similarly to
$d\Delta^a_{bc}$.

\item $\glosm{hsigma}{h\sigma}^x_y$ and $\glosm{tsigma}{t\sigma}^u_v$
are defined similarly to $d\sigma^a_b$.

\item \textbf{New!} Given a tail $u\in T$, a ``new'' tail
label $v\not\in T\remove u$ and a head $x\in H$ the operation
$\glosm{thm}{thm}^{ux}_v\colon\calA^w(H;T)\to\calA^w(H\remove x;(T\remove
u)\sqcup\{v\})$ is the obvious ``tail-strand head-strand stitching''
--- similarly to $dm^{ab}_c$, stitch the strand $u$ to the strand $x$
putting $u$ before $x$, and call the resulting ``new'' strand $v$. Note
that for this to be well defined, $v$ must be a tail strand.\footnoteT{
Note also that the analogous operation $htm^{xu}_v$ ``put $x$ before $u$
to get a tail $v$'' is $0$ and hence we can safely ignore it, and that
$thm^{xu}_y$ and $htm^{xu}_y$, defined in the same way as $thm^{ux}_v$
and $htm^{xu}_v$ except to produce a head strand $y$, are not well
defined because they do not respect the $CP$ relation.}

In practise, $thm^{ux}_v$ is never used on its own, but the combination
$h\Delta^x_{xx'}\act thm^{ux'}_u$ (where $x'$ is a temporary label) is
very useful. Hence we set
$\glosm{tha}{tha}^{ux}\colon\calA^w(H;T)\to\calA^w(H;T)$ (``tail by head
action on $u$ by $x$'') to be that combination. In words, this is ``double
the strand $x$ and put one of the copies on top of $u$''.\footnoteT{Note
that $thm^{ux}_v=tha^{ux}\act h\eta^x\act t\sigma^u_v$ so we lose no
generality by considering $tha^{ux}$ instead of $thm^{ux}_v$.}

  \parpic[l]{\begin{picture}(0,0)%
\includegraphics{figs/Topology.pstex}%
\end{picture}%
%
%  pstex_opts: -m 0.3333 
%
\setlength{\unitlength}{1316sp}%
\begingroup\makeatletter\ifx\SetFigFont\undefined%
\gdef\SetFigFont#1#2#3#4#5{%
  \reset@font\fontsize{#1}{#2pt}%
  \fontfamily{#3}\fontseries{#4}\fontshape{#5}%
  \selectfont}%
\fi\endgroup%
\begin{picture}(1052,1180)(0,-190)
\end{picture}%
\,\begin{picture}(0,0)%
\includegraphics{figs/Lie.pstex}%
\end{picture}%
%
%  pstex_opts: -m 0.36 
%
\setlength{\unitlength}{1421sp}%
\begingroup\makeatletter\ifx\SetFigFont\undefined%
\gdef\SetFigFont#1#2#3#4#5{%
  \reset@font\fontsize{#1}{#2pt}%
  \fontfamily{#3}\fontseries{#4}\fontshape{#5}%
  \selectfont}%
\fi\endgroup%
\begin{picture}(960,990)(46,0)
\end{picture}%
  \hspace{-6pt}}%
  \noindent%
 In topology, $tha$ is the action of hoops on balloons as
in~\cite[Section~\ref{KBH-subsec:MMAOperations}]{KBH}, which is similar to
the action of $\pi_1$ on $\pi_2$. In Lie theory, it is the right action of
$\calU(\frakg)$ on the Verma module $\calU(I\frakg)/\frakg\calU(I\frakg)$,
or better, the action of $\calU(\frakg)$ on $\calS(\frakg^\ast)$ induced
from the co-adjoint action of $\frakg$ on $\frakg^\ast$.
\endpar{def:AHTOperations}

\end{enumerate}
\end{definition}

  \parpic[l]{\begin{picture}(0,0)%
\includegraphics{figs/Topology.pstex}%
\end{picture}%
%
%  pstex_opts: -m 0.3333 
%
\setlength{\unitlength}{1316sp}%
\begingroup\makeatletter\ifx\SetFigFont\undefined%
\gdef\SetFigFont#1#2#3#4#5{%
  \reset@font\fontsize{#1}{#2pt}%
  \fontfamily{#3}\fontseries{#4}\fontshape{#5}%
  \selectfont}%
\fi\endgroup%
\begin{picture}(1052,1180)(0,-190)
\end{picture}%
\,\begin{picture}(0,0)%
\includegraphics{figs/Lie.pstex}%
\end{picture}%
%
%  pstex_opts: -m 0.36 
%
\setlength{\unitlength}{1421sp}%
\begingroup\makeatletter\ifx\SetFigFont\undefined%
\gdef\SetFigFont#1#2#3#4#5{%
  \reset@font\fontsize{#1}{#2pt}%
  \fontfamily{#3}\fontseries{#4}\fontshape{#5}%
  \selectfont}%
\fi\endgroup%
\begin{picture}(960,990)(46,0)
\end{picture}%
  \hspace{-6pt}}%
  \noindent%
 {\em Exercise \refstepcounter{theorem}\thetheorem.} In the cases
when we did not state the topological or Lie theoretical meaning of an
operation in Definition~\ref{def:AHTOperations}, find what it is.

\subsubsection{Group-like elements in $\{\calA^w(H;T)\}$.}
\label{sssec:AHTExp}

For any fixed finite sets $H$ and $T$ there is a co-product
$\glosm{htBox}{\Box}\colon\calA^w(H;T)\otimes\calA^w(H;T)$ defined just
as in the case of $\calA^w(S)$ (Definition~\ref{def:GroupLike}),
and along with the product $\#$ (and obvious units and co-units),
$\calA^w(H;T)$ is a graded connected co-commutative bi-algebra. Hence
it makes sense to speak of the group-like elements $\glosm{AwHTexp}{\calA^w_{\exp}(H;T)}$
within $\calA^w(H;T)$, and they are all $\#$-exponentials of primitives in
$\calA^w(H;T)$. The primitives $\glosm{AprimwHT}{\calA_{\text{prim}}^w(H;T)}$ in
$\calA^w(H;T)$ are connected diagrams and hence they are trees and
wheels. As in Comment~\ref{com:PureForm}, the trees must have their
roots on head strands and their leafs on tail strands, and the
wheels must have all their ``legs'' on tail strands. As tails commute, we
may think of the trees as abstract trees with leafs labelled by labels in
$T$ and roots in $H$, and the wheels are abstract cyclic words with letters
in $T$. Hence canonically $\calA_{\text{prim}}^w(H;T)\simeq\FL(T)^H\oplus\CW(T)$ and
hence there is a bijection (called ``the split presentation $E_s$'')\glosm{Es}{$\null$}
\begin{equation} \label{eq:EsHT}
  \yellowm{E_s}\colon\yellowm{\TW_s(H;T)}\coloneqq\FL(T)^H\oplus\CW(T)
    \overset{\sim}{\longrightarrow} \calA^w_{\exp}(H;T)
\end{equation}
defined on an ordered pair $\glosm{parens}{(\lambda;\,\omega)_s}$ in
$\TW_s(H;T)$ by
\begin{equation} \label{eq:esHT}
  (\lambda;\,\omega)_s\mapsto
    \exp_\#\left(e_s(\lambda;\omega)\right),
\end{equation}

\Needspace{2in}
\parpic[r]{\parbox{1.75in}{
  \centering{\input{figs/Es.pstex_t}}
  \begin{figcap} \label{fig:Es} $E_s(\lambda;\,\omega)_s$. \end{figcap}
}}
\noindent where $\glosm{es}{e_s}(\lambda;\omega)_s$ is the sum over $x\in
H$ of planting $\lambda_x$ with its root on strand $x$ and its leafs on
the strands in $T$ so that the labels match but at an arbitrary order
on any $T$ strand, plus the result of planting $\omega$ on just the $T$
strands so that the labels match but at an arbitrary order on any $T$
strand. A pictorial representation of $E_s(\lambda;\,\omega)_s$, using
the same visual language as in Figure~\ref{fig:El}, appears on the right.

It is easy to verify that the
operations in Definition~\ref{def:AHTOperations} intertwine $\Box$ and
hence map group-like elements to group-like elements and hence they induce
operations on $\TW_s(H;T)$. These are summarized within the following
definition-proposition.

\begin{defprop} \label{dp:EsOps} The bijection $E_s$ intertwines the
operations defined below with the operations in
Definition~\ref{def:AHTOperations}:\footnoteT{Here
we no longer state conditions such as $H_1\cap H_2=\emptyset$,
$u\in T$, $x\in H$. They are the same as in
Definition~\ref{def:AHTOperations}, and more importantly, they are ``what
makes sense''.}\,\footnoteC{We quote from
\href{\web/AwCalculus.m}{\tt AwCalculus.m} only the most interesting
implementations --- of $\sqcup$~\eqref{eq:EsCup}, of $hm$~\eqref{eq:Eshm},
of $tm$~\eqref{eq:Estm}, and of $tha$~\eqref{eq:Estha}. Then we set
the values of two ``sample'' elements in the $E_s$ presentation
(on the computer we represent $(\lambda;\,\omega)_s$ as {\tt
Es[$\lambda$,$\omega$]}):
\ACQuote{EsSampleDefs}

\shortdialoginclude{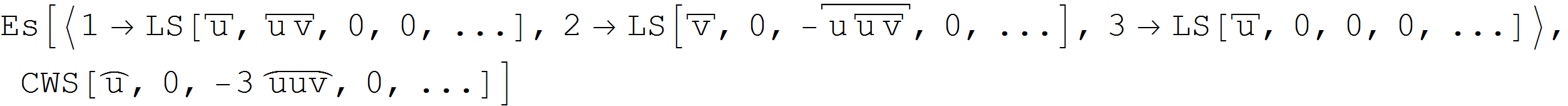}

\shortdialoginclude{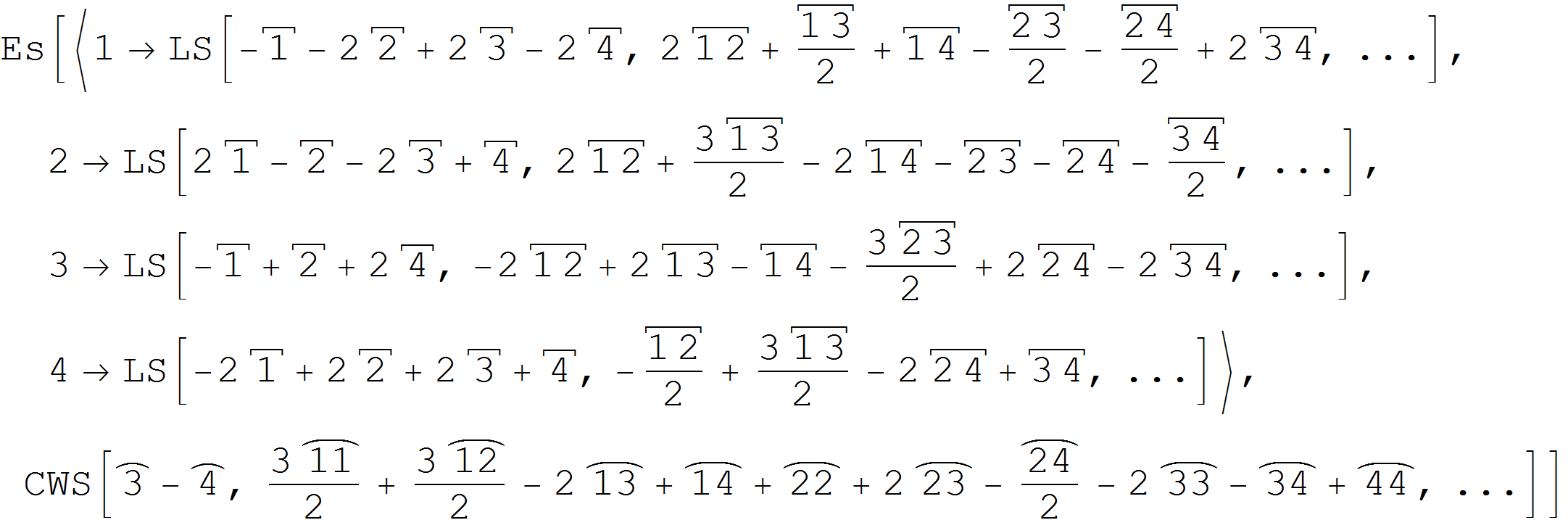}

(Note that the second of sample elements was set to be a random series,
with a seed of $0$. It is printed only to degree $2$, but it extends
indefinitely as a random series.)
}

\begin{enumerate}[leftmargin=*,labelindent=0pt]

\item
$\ds
  (\lambda_1;\,\omega_1)_s(\lambda_2;\,\omega_2)_s
  = (\lambda_1;\,\omega_1)_s\glosm{ssqcup}{\sqcup}(\lambda_2;\,\omega_2)_s
  \coloneqq (\lambda_1\sqcup\lambda_2;\,\omega_1+\omega_2)_s
$\hfill\inlineeq\label{eq:EsCup}

\item
$\ds
  (\lambda_1;\,\omega_1)_s\glosm{sjail}{\#}(\lambda_2;\,\omega_2)_s
  \coloneqq \left(
    (x\to\BCH(\lambda_{1x},\lambda_{2x}))_{x\in H}
    ;\,
    \omega_1+\omega_2
  \right)_s
$\hfill\inlineeq\label{eq:EsProduct}

\item
$\ds
  (\lambda;\,\omega)_s \act \glosm{sheta}{h\eta}^x
  \coloneqq (\lambda\remove x;\,\omega)_s
$\hfill\inlineeq\label{eq:EshEta}
\newline\hspace{16pt}
$\ds (\lambda;\,\omega)_s \act \glosm{steta}{t\eta}^u
  \coloneqq (\lambda\act(u\to 0) ;\, \omega\act(u\to 0))_s
$\hfill\inlineeq\label{eq:EstEta}

\item
$\ds
  (\lambda;\,\omega)_s \act \glosm{shA}{hA}^x
  \coloneqq ((\lambda\remove x)\sqcup(x\to-\lambda_x) ;\, \omega)_s
$\hfill\inlineeq\label{eq:EshA}
\newline\hspace{16pt}
$\ds \glosm{stA}{tA}^u \coloneqq I $\hfill\inlineeq\label{eq:EstA}

\item
$\ds \glosm{shS}{hS}^x\coloneqq hA^x,$\hfill\inlineeq\label{eq:EshS}
\newline\hspace{16pt}
$\ds
  (\lambda;\,\omega)_s \act \glosm{stS}{tS}^u
  \coloneqq (\lambda\act(u\to -u) ;\, \omega\act(u\to -u))_s
$\hfill\inlineeq\label{eq:EstS}

\item
$\ds
  (\lambda;\,\omega)_s \act \glosm{shm}{hm}^{xy}_z
  \coloneqq ((\lambda\remove\{x,y\})\sqcup(z\to\BCH(\lambda_x, \lambda_y));\,\omega)_s
$\hfill\inlineeq\label{eq:Eshm}
\newline\hspace{16pt}$\ds 
  (\lambda;\,\omega)_s \act \glosm{stm}{tm}^{uv}_w
  \coloneqq (\lambda\act(u,v\to w) ;\, \omega\act(u,v\to w))_s
$\hfill\inlineeq\label{eq:Estm}

\item
$\ds 
  (\lambda;\,\omega)_s \act \glosm{shDelta}{h\Delta}^x_{yz}
  \coloneqq ((\lambda\remove x)\sqcup(y\to\lambda_x, z\to\lambda_x);\,\omega)_s
$\hfill\inlineeq\label{eq:EshDelta}
\newline\hspace{16pt}$\ds
  (\lambda;\,\omega)_s \act \glosm{stDelta}{t\Delta}^u_{vw}
  \coloneqq (\lambda\act(u\to v+w) ;\, \omega\act(u\to v+w))_s
$\hfill\inlineeq\label{eq:EstDelta}

\item
$\ds
  (\lambda;\,\omega)_s \act \glosm{shsigma}{h\sigma}^x_y
  \coloneqq ((\lambda\remove x)\sqcup(y\to\lambda_x);\,\omega)_s
$\hfill\inlineeq\label{eq:EshSigma}
\newline\hspace{16pt}$\ds
  (\lambda;\,\omega)_s \act \glosm{stsigma}{t\sigma}^u_v
  \coloneqq (\lambda\act(u\to v) ;\, \omega\act(u\to v))_s
$\hfill\inlineeq\label{eq:EstSigma}

\item $\ds
  (\lambda;\,\omega)_s \act \glosm{stha}{tha}^{ux}
  \coloneqq \left(
    \lambda \act RC_u^{\lambda_x} ;\,
    (\omega+J_u(\lambda_x))\act RC_u^{\lambda_x}
  \right)_s.
$\hfill\inlineeq\label{eq:Estha}
\end{enumerate}
\end{defprop}

\noindent{\em Proof.} The first 8
assertions (14 operations) are very easy. The main challenge to the reader
should be to gather her concentration for the 14-times repetitive task
of unwrapping definitions. If you are ready to cut corners, only go over
\eqref{eq:EsCup}, \eqref{eq:Eshm}, \eqref{eq:Estm}, \eqref{eq:EshDelta},
and \eqref{eq:EstDelta}. Let us turn to the proof of the last assertion,
Equation~\eqref{eq:Estha}. That proof is in fact in~\cite{KBH}, or
at least can be assembled from pieces already in~\cite{KBH}. Yet the
assembly would be a bit delicate, and hence a proof is reproduced below
which refers back to~\cite{KBH} only at one technical point.

By inspecting the definition of $tha^{ux}$, it is clear that there
is {\em some} assignment $\gamma\mapsto R_u^\gamma$ that assigns an
operator $R_u^\gamma\colon\FL(T)\to\FL(T)$ to every $\gamma\in\FL(T)$
and that there is {\em some} functional $K_u\colon\FL(T)\to\CW(T)$,
for which a version of Equation~\eqref{eq:Estha} holds:

\begin{equation} \label{eq:Esthap}
  E_s(\lambda;\,\omega)_s \act tha^{ux}
  = E_s\left(
    \lambda \act R_u^{\lambda_x} ;\,
    (\omega+K_u(\lambda_x))\act R_u^{\lambda_x}
  \right)_s
\end{equation}

Indeed, $tha^{ux}$ acts on $E_s(\lambda;\,\omega)_s$ by placing a copy
of $\exp(\lambda_x)$ at the top of the tail strand $u$, and then
re-writing the result without having any heads on strand $u$ so as to
invert $E_s$ back again. The re-writing is done by sliding the heads of
$\exp(\lambda_x)$ down to the bottom of strand $u$, where they cancel
by $CP$. Every time a head slides past a tail we get a contribution from
$\aSTU_2$. Sometimes a head of a $\lambda_x$ will slide against a tail of
another $\lambda_x$, whose head will have to slide down too, leading to
a rather complicated iterative process. Nevertheless, these contributions
are the same for every tail on strand $u$, namely for every occurrence of
the variable $u$ in $\FL(T)^H$ and/or in $\CW(T)$. This explains the terms
$\lambda \act R_u^{\lambda_x}$ and $\omega \act R_u^{\lambda_x}$ in
Equation~\eqref{eq:Esthap}. We note that the degree $0$ part of the
operator $R_u^{\lambda_x}$ is the identity, and hence it is invertible.

But yet another type of term arises in the process --- sometimes a head of
some tree will slide against a tail of its own, and then the contribution
arising from $\aSTU_2$ will be a wheel. Hence there is an additional
contribution to the output, some $L_u(\lambda_x)$ which clearly can depend
only on $u$ and $\lambda_x$. Using the invertibility of $R_u^{\lambda_x}$
to write $L_u(\lambda_x)=K_u(\lambda_x)\act R_u^{\lambda_x}$ we completely
reproduce Equation~\eqref{eq:Esthap}.

We now need to show that $R_u^\gamma$ and $K_u(\gamma)$ are $RC_u^\gamma$
and $J_u(\gamma)$ of Definitions~\ref{def:CRC} and~\ref{def:J}. Tracing
again through the discussion in the previous two paragraphs, we see that at
any fixed degree, $R_u^\gamma$ and $K_u(\gamma)$ depend polynomially on the
coefficients of $\gamma$, and hence it is legitimate to study their
variation with
respect to $\gamma$. It is also easy to verify that $R_u^0=RC_u^0=I$ and
that $K_u(0)=J_u(0)=0$, and hence it is enough to show that, with an
indeterminate scalar $\tau$,
\begin{equation} \label{eq:DerEqns}
  \frac{d}{d\tau}R_u^{\tau\gamma}=\frac{d}{d\tau}RC_u^{\tau\gamma}
  \qquad\text{and}\qquad
  \frac{d}{d\tau}K_u(\tau\gamma)=\frac{d}{d\tau}J_u(\tau\gamma).
\end{equation}

Let us compute the left-hand-sides of the above equations. If $\tau$ is an
infinitesimal (so $\tau^2=0$), or more precisely, computing the above
left-hand-sides at $\tau=0$, we can re-trace the process described in the
two paragraphs following Equation~\eqref{eq:Esthap} keeping in mind that
with $\lambda_x=\tau\gamma$ the $\aSTU_2$ relation can only by applied
once (or else terms proportional to $\tau^2$ will arise). The result is
\begin{equation} \label{eq:DersAtZero}
  \left.\frac{d}{d\tau}R_u^{\tau\gamma}\right|_{\tau=0}
    = \ad_u^\gamma
  \qquad\text{and}\qquad
  \left.\frac{d}{d\tau}K_u(\tau\gamma)\right|_{\tau=0}
    = \atdiv_u(\gamma),
\end{equation}
where $\glosm{adugamma}{\ad_u^\gamma}\colon\FL(T)\to\FL(T)$
is the derivation which maps the generator $u$ of $\FL(T)$ to
$[\gamma,u]$ and annihilates all other generators of $\FL(T)$
(compare~\cite[Definition~\ref{KBH-def:adu}]{KBH}) and where
$\atdiv_u(\gamma)$ is the same as in Definition~\ref{def:J}.

Moving on to general $\tau$, we
note that the operations $hm$ and $tha$ satisfy\footnoteC{None 
should believe without a verification:

\shortdialoginclude{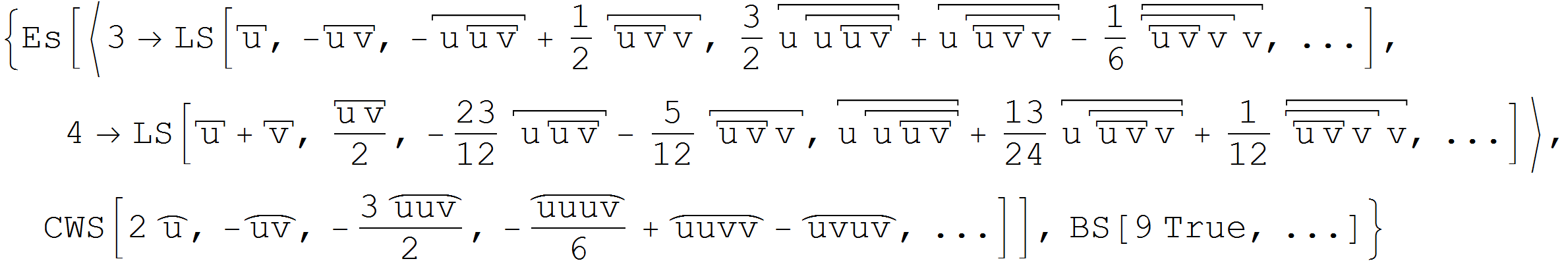}
}
\begin{equation} \label{eq:haction}
  hm^{xy}_z\act tha^{uz}=tha^{ux}\act tha^{uy}\act hm^{xy}_z
\end{equation}
(stitching strands $x$ and $y$ and then stitching a copy
of the result to $u$ is the same as stitching a copy of $x$
to $u$, then a copy of $y$, and then stitching $x$ to $y$;
compare~\cite[Equation~\eqref{KBH-eq:haction}]{KBH}). Applying the
operators on the two sides of Equation~\eqref{eq:haction} to
$E_s(\lambda;\,\omega)$ (assuming $H$ and $T$ are such that it makes
sense), then expanding using~\eqref{eq:Eshm} and~\eqref{eq:Esthap}, and
then ignoring the wheels in the resulting
equality, we find that $R_u$ satisfies
\begin{equation} \label{eq:Rh}
  R_u^{\BCH(\lambda_x,\lambda_y)}
    = R_u^{\lambda_x}\act R_u^{\lambda_y\act R_u^{\lambda_x}}
\end{equation}
(compare~\cite[Equation~\eqref{KBH-eq:RCh}]{KBH}). Similarly, looking only
at the wheel part of~\eqref{eq:haction} we get
\[
  K_u(\BCH(\lambda_x,\lambda_y))\act R_u^{\BCH(\lambda_x,\lambda_y)}
    = K_u(\lambda_x)\act
      R_u^{\lambda_x}\act R_u^{\lambda_y\act R_u^{\lambda_x}}
      + K_u(\lambda_y\act R_u^{\lambda_x})\act
        R_u^{\lambda_y\act R_u^{\lambda_x}},
\]
which, composing on the right with $R_u^{\BCH(\lambda_x,\lambda_y)}$
and using~\eqref{eq:Rh}, is equivalent to
\begin{equation} \label{eq:Kh}
  K_u(\BCH(\lambda_x,\lambda_y))
    = K_u(\lambda_x)\act R_u^{\lambda_x}
      + K_u(\lambda_y\act R_u^{\lambda_x})\act C_u^{-\lambda_x}
\end{equation}
(compare~\cite[Equation~\eqref{KBH-eq:JhProperty}]{KBH}).

Equations~\eqref{eq:Rh} and~\eqref{eq:Kh} hold for any $\lambda$,
and hence for any $\lambda_x$ and $\lambda_y$. Specializing
to $\lambda_x=\tau\gamma$ and $\lambda_y=\epsilon\gamma$, where
$\epsilon$ is some new indeterminate scalar, and using the fact that
$\BCH(\tau\gamma,\epsilon\gamma) = (\tau+\epsilon)\gamma$,
Equations~\eqref{eq:Rh} and~\eqref{eq:Kh} become
\[
  R_u^{(\tau+\epsilon)\gamma}
    = R_u^{\tau\gamma}\act R_u^{\epsilon\gamma\act R_u^{\tau\gamma}}
  \qquad\text{and}\qquad
  K_u((\tau+\epsilon)\gamma)
    = K_u(\tau\gamma)\act R_u^{\tau\gamma}
      + K_u(\epsilon\gamma\act R_u^{\tau\gamma})\act C_u^{-\tau\gamma}.
\]
Now differentiating with respect to $\epsilon$ at $\epsilon=0$ and using
Equation~\eqref{eq:DersAtZero} with $\tau$ replaced with $\epsilon$, we get 
\[
  \frac{d}{d\tau}R_u^{\tau\gamma}
    = R_u^{\tau\gamma}\act\ad_u^{\gamma\act R_u^{\tau\gamma}}
  \qquad\text{and}\qquad
  \frac{d}{d\tau}K_u(\tau\gamma)
    = \atdiv_u(\gamma\act R_u^{\tau\gamma})\act C_u^{-\tau\gamma}.
\]

The first of these equations is the same equation that is satisfied by
$RC_u$ (see \cite[Lemma \ref{KBH-lem:dC}]{KBH}, with $\delta\gamma$
proportional to $\gamma$), and hence $R_u=RC_u$. By a simple change
of variables, $J_u(\tau\gamma)=\int_0^\tau dt\,\atdiv_u\!\left(
\gamma \sslash RC_u^{t\gamma} \right) \sslash C_u^{-t\gamma}$,
and hence $\frac{d}{d\tau}J_u(\tau\gamma) = \atdiv_u(\gamma\act
RC_u^{\tau\gamma})\act C_u^{-\tau\gamma}$ (compare with the
formula for the full differential of $J$, \cite[Proposition
\ref{KBH-prop:dJ}]{KBH}). Comparing with the above formula for the
derivative of $K_u$, we find that $K_u=J_u$. \qed

\subsubsection{The inclusion $\{\calA^w(S)\} \hookrightarrow
\{\calA^w(H;T)\}$.} \label{sssec:Inclusion}

The following definition and proposition imply that there is no
loss in studying the spaces $\calA^w(H;T)$ rather than the spaces
$\calA^w(S)$.

\begin{definition} Let
$\glosm{delta}{\delta}\colon\calA^w(S)\to\calA^w(S;S)$
be the composition of the ``double every strand'' map
$\prod_{a\in S}\Delta^a_{ha,ta} \colon \calA^w(S)
\to \calA^w({hS\sqcup tS})$ with the projection
$\calA^w({hS\sqcup tS})\to\calA^w(S;S)$ (as an exception to
the rule of Footnote~\ref{foot:BruteDisjoint} we temporarily highlight
the distinction between head and tail labels by affixing them with the
prefixes $h$ and $t$).
\end{definition}

\Needspace{1.5in}
{
\makeatletter\def\thm@space@setup{%
  \thm@preskip=0cm plus 0cm minus 0cm %\thm@postskip=\thm@preskip
}\makeatother
\parpic[r]{\input{figs/delta.pstex_t}}
\begin{comment} \label{com:delta} If $D\in\calA^w(S)$ is sorted
``heads below tails'' as in Comment~\ref{com:SortedForm}, then $\delta
D$ is $D$ with its arrow heads placed on the head strands and its arrow
tails placed on the tail strands, as shown on the right.
\end{comment}
}

\begin{proposition} $\delta$ is a (non-multiplicative) vector space
isomorphism. The inverse of $\delta$ on $D\in\calA^w(S;S)$ is given by
the process
\begin{enumerate}
\item Write $D$ with only arrow heads on the head strands and only arrow
tails on the tail strands. By Comment~\ref{com:PureForm} this produces a
well-defined element $D'$ of $\calA^w({hS\sqcup tS})$.
\item Stitch all the head-tail pairs of strands in
$D'$ by putting each head ahead of its corresponding tail: $\delta^{-1}D =
D'\act\prod_a dm^{ha,ta}_a$.
\end{enumerate}
\end{proposition}

\begin{proof} $\delta^{-1}\act\delta=I$ by inspection, and
$\delta\act\delta^{-1}$ is clearly the identity on diagrams sorted to have
heads ahead of tails as in Comment~\ref{com:SortedForm}. \qed
\end{proof}

  \parpic[l]{\begin{picture}(0,0)%
\includegraphics{figs/Topology.pstex}%
\end{picture}%
%
%  pstex_opts: -m 0.3333 
%
\setlength{\unitlength}{1316sp}%
\begingroup\makeatletter\ifx\SetFigFont\undefined%
\gdef\SetFigFont#1#2#3#4#5{%
  \reset@font\fontsize{#1}{#2pt}%
  \fontfamily{#3}\fontseries{#4}\fontshape{#5}%
  \selectfont}%
\fi\endgroup%
\begin{picture}(1052,1180)(0,-190)
\end{picture}%
\,\begin{picture}(0,0)%
\includegraphics{figs/Lie.pstex}%
\end{picture}%
%
%  pstex_opts: -m 0.36 
%
\setlength{\unitlength}{1421sp}%
\begingroup\makeatletter\ifx\SetFigFont\undefined%
\gdef\SetFigFont#1#2#3#4#5{%
  \reset@font\fontsize{#1}{#2pt}%
  \fontfamily{#3}\fontseries{#4}\fontshape{#5}%
  \selectfont}%
\fi\endgroup%
\begin{picture}(960,990)(46,0)
\end{picture}%
  \hspace{-6pt}}%
  \noindent%
 In topology, $\delta$ agrees with the
$\delta$ of~\cite[Section~\ref{KBH-subsec:delta}]{KBH}. In Lie theory,
it agrees with the linear (non-multiplicative) isomorphism
$\calU(I\frakg)\simeq\calU(\frakg)\otimes\calS(\frakg^\ast)$ and with
similar isomorphisms considered by Etingof and Kazhdan within their work
on the quantization of Lie bialgebras~\cite{EtingofKazhdan:BialgebrasI}
(albeit only when the Lie bialgebras in question are cocommutative).

\begin{definition} The product $\#$ of $\calA^w(S;S)$ induces a new
product, also denoted $\glosm{wjail}{\#}$, on $\calA^w(S)$. If
$D_1$ and $D_2$ are in $\calA^w(S)$, set
\begin{equation} \label{eq:FProduct}
  D_1\#D_2\coloneqq(\delta(D_1)\#\delta(D_2))\act\delta^{-1}.
\end{equation}
\end{definition}

\Needspace{1in}
{
\makeatletter\def\thm@space@setup{%
  \thm@preskip=0cm plus 0cm minus 0cm %\thm@postskip=\thm@preskip
}\makeatother
\parpic[r]{\input{figs/jail.pstex_t}}
\begin{comment}
With Comment~\ref{com:delta} in mind, we see that if $D_1$ and $D_2$ are
sorted as in Comment~\ref{com:SortedForm}, then $D_1\#D_2$ is  ``heads of
$D_1$, then of $D_2$, then tails of $D_1$, then of $D_2$''
(with the last two parts interchangeable, by $TC$). The picture is nicer
when rotated, as on the right.
\end{comment}
}

  \parpic[l]{\scalebox{0.6}{%
    \begin{picture}(0,0)%
\includegraphics{figs/Topology.pstex}%
\end{picture}%
%
%  pstex_opts: -m 0.3333 
%
\setlength{\unitlength}{1316sp}%
\begingroup\makeatletter\ifx\SetFigFont\undefined%
\gdef\SetFigFont#1#2#3#4#5{%
  \reset@font\fontsize{#1}{#2pt}%
  \fontfamily{#3}\fontseries{#4}\fontshape{#5}%
  \selectfont}%
\fi\endgroup%
\begin{picture}(1052,1180)(0,-190)
\end{picture}%
\,\begin{picture}(0,0)%
\includegraphics{figs/Lie.pstex}%
\end{picture}%
%
%  pstex_opts: -m 0.36 
%
\setlength{\unitlength}{1421sp}%
\begingroup\makeatletter\ifx\SetFigFont\undefined%
\gdef\SetFigFont#1#2#3#4#5{%
  \reset@font\fontsize{#1}{#2pt}%
  \fontfamily{#3}\fontseries{#4}\fontshape{#5}%
  \selectfont}%
\fi\endgroup%
\begin{picture}(960,990)(46,0)
\end{picture}%
    \hspace{-6pt}}}%
  \noindent%
 See the comments following Discussion~\ref{disc:WhyTwo}.\newline

The next proposition shows how the operations of defined on the
$\calA^w(S)$-spaces in Definition~\ref{def:Operations}
can be written in terms of the ``head and tail'' operations of
Definition~\ref{def:AHTOperations}, thus completing the description of
the $E_s$ presentation.

\begin{proposition} \label{prop:dinht}
\begin{enumerate}[leftmargin=*,labelindent=0pt]

\item \label{it:HTsqcup} If $S_1$ and $S_2$ are disjoint and
$D_1\in\calA^w({S_1})$ and $D_2\in\calA^w({S_2})$,
then $\delta(D_1\sqcup D_2)=\delta(D_1)\sqcup\delta(D_2)$.

\item \label{it:HTStacking}
Let $D_1,D_2\in\calA^w(S)$. Then $\delta(D_1D_2)$ can be
written in terms of $\delta(D_1)$ and $\delta(D_2)$ using its description
in terms of $\sqcup$, $d\sigma$, and $dm$ in Equation~\eqref{eq:multiplem} and
using the formulas for $\sqcup$, $d\sigma$, and $dm$ that appear in
parts (\ref{it:HTsqcup}), (\ref{it:HTsigma}), and (\ref{it:HTdm}) of this
proposition.\footnotemarkC
\addtocounter{footnoteC}{-1}
\begin{multicols}{2}
\item $d\eta^a\act\delta = \delta\act h\eta^a\act t\eta^a$.

\item \label{it:HTdA}
  $\dA^a \act \delta = \delta \act hA^a \act tA^a\act tha^{aa}$.

\item \label{it:HTdS}
  $\dS^a \act \delta = \delta \act hS^a \act tS^a\act tha^{aa}$.

\item \label{it:HTdm}
  $dm^{ab}_c \act \delta = \delta \act tha^{ab} \act hm^{ab}_c \act
tm^{ab}_c$.\footnotemarkC

\item $d\Delta^a_{bc} \act \delta =
  \delta \act h\Delta^a_{bc} \act t\Delta^a_{bc}$.

\item \label{it:HTsigma} $d\sigma^a_b \act \delta = \delta \act
  h\sigma^a_b \act t\sigma^a_b$.

\end{multicols}
\end{enumerate}
\end{proposition}

\footnotetextC{As a sample for the whole proposition, we quote
the implementation of $dm$ and verify its meta-associativity
$dm^{ab}_a \act dm^{ac}_a = dm^{bc}_b \act dm^{ab}_a$
(compare~\cite[Equation~\eqref{KBH-eq:assoc}]{KBH}). We then include our
implementation of the stacking product (item {\it(\ref{it:HTStacking})}
above) without further explanations:
\ACQuote{Esdm}

\shortdialoginclude{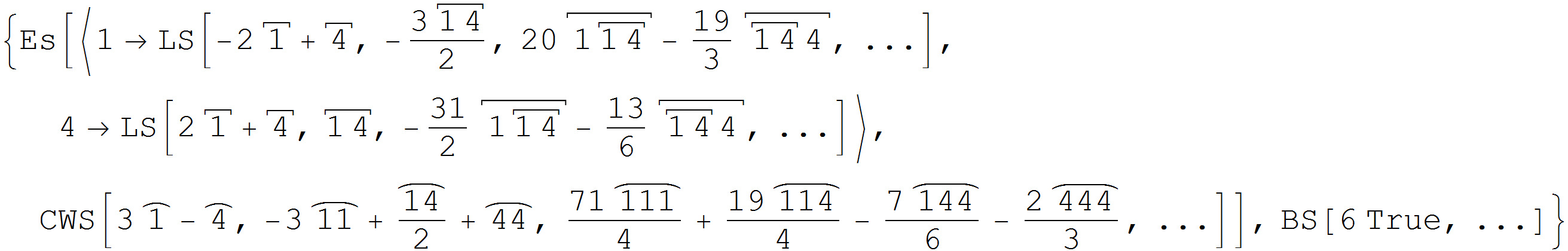}
\ACQuote{EsNCM}
}

\begin{proof} The only difficulty is with items
{\it(\ref{it:HTdA})}--{\it(\ref{it:HTdm})}. Item~{\it(\ref{it:HTdA})}
is easier to understand in the form $\delta^{-1}\act \dA^a= hA^a \act
tA^a\act tha^{aa}\act\delta^{-1}$. Indeed, $\delta^{-1}$ plants heads
ahead of tails on strand $a$. Applying $\dA^a$ reverses that strand
(and adds some signs). This reversal can be achieved by reversing the
head part (with signs), then the tail part (with signs), and then
by swapping the two parts across each other. The first reversal
is $hA^a$, the second is $tA^a$, and the swap is $tha^{aa}$
followed by $\delta^{-1}$. Item~{\it(\ref{it:HTdS})} is proven
in exactly the same way, and item~{\it(\ref{it:HTdm})} is proven
in a similar way, where the right hand side traces the schematics
$(ha\,ta\,hb\,tb)\xrightarrow{tha}(ha\,hb\,ta\,tb)\xrightarrow{hm\act
tm}((ha\,hb)(ta\,tb))$.  \qed
\end{proof}

\begin{discussion} \label{disc:coalg} It is easy to verify
that $\delta\colon \calA^w(S) \to \calA^w(S;\,S)$
is a co-algebra morphism,
and hence it restricts to an isomorphism $\delta\colon
{\calA^w_{\exp}(S)} \to \calA^w_{\exp}(S;\,S)$. Therefore
$E_s\act\delta^{-1}$ is a bijection between $\TW_s(S)\coloneqq\TW_s(S;S)$ and
$\calA^w_{\exp}(S)$. Proposition~\ref{prop:dinht} now tells us
how to write all the ``$d$'' operations of Definition~\ref{def:Operations}
as compositions of ``$h$'' and ``$t$'' operations, and
Definition-Proposition~\ref{dp:EsOps} tells us how to write these as
operations on $\TW_s(H;T)$ (the $H$ and $T$ label sets that occur here
are always $S$ with one or two labels
added or removed). Hence overall $E_s\act\delta^{-1}$, acting on $\TW_s(S)$,
is a complete presentation of $\calA^w_{\exp}(S)$.
\end{discussion}

\Needspace{5cm}
\vskip 1mm
\makeatletter\def\thm@space@setup{%
  \thm@preskip=0cm plus 0cm minus 0cm %\thm@postskip=\thm@preskip
}\makeatother
\parpic[r]{\parbox{1.75in}{
  \centering{\input{figs/Ef.pstex_t}}
  \begin{figcap} \label{fig:Ef} $E_f(\lambda;\,\omega)_s$. \end{figcap}
}}
\begin{definition} The ``factored'' presentation $\glosm{Ef}{E_f}$
of $\calA^w_{\exp}$ is the composition ${E_f\coloneqq
E_s\act\delta^{-1}}$. Namely, for a set $S$ of strands, we define
$
  E_f\colon\TW_s(S)
    \overset{\sim}{\longrightarrow} \calA^w_{\exp}(S)
$
by
$(\lambda;\,\omega)_s\mapsto
  E_s(\lambda;\,\omega)_s\act\delta^{-1}
  = \exp_\#\left(l\lambda + \iota\omega\right)
$. See the illustration on the right.
\end{definition}

\draftcut
\subsection{Converting between the $E_l$ and the $E_f$ presentations.}
\label{subsec:Conversion}

We now have two presentations for elements of $\calA^w_{\exp}(S)$,
and we wish to be able to convert between the two. This turns out to
involve the maps $\Gamma$ and $\Lambda$ of Propositions~\ref{prop:Gamma}
and~\ref{prop:Lambda}.

\begin{definition} Define a pair of inverse maps
$\glosm{TWGamma}{\Gamma}\colon\TW_l(S)\to\TW_s(S)$ and
$\glosm{TWLambda}{\Lambda}\colon\TW_s(S)\to\TW_l(S)$ by
\[ \Gamma\colon(\lambda;\,\omega)_l \mapsto (\Gamma(\lambda);\,\omega)_s
  \quad\text{and}\quad
  \Lambda\colon(\lambda;\,\omega)_s \mapsto (\Lambda(\lambda);\,\omega)_l.
\]
\end{definition}

\begin{theorem} \label{thm:GammaLambda}
The left-most triangle in Figure~\ref{fig:diagram} commutes. Namely,
\begin{equation} \label{eq:convertion}
  E_l = \Gamma \act E_f
  \quad\text{and}\quad
  E_f = \Lambda \act E_l.
\end{equation}
(All other parts of Figure~\ref{fig:diagram} commute by definition).
\end{theorem}

Before we can prove this theorem we need a few preliminaries. For an
element $D\in{\calA^w_{\exp}(S)}$, we can define three associated
quantities:

\begin{myitemize}

\item The projection of $D$ to the degree 1 part of $\calA^w(S)$, and
especially, the projection $\glosm{piA}{\pi_A}(D)$ of the degree $1$
part to its ``framing'' part $A_S$ (consisting of self-arrows,
that begin and end on the same strand and point, say, up).

\item A conjugation automorphism $C_D$ of $\FL(S)$, defined as follows.
First, embed $\FL(S)$ into $\calA^w({S\sqcup\{\infty\}})$
by mapping any generator $a\in S$ to a degree 1 diagram in
${\calA^w({S\sqcup\{\infty\}})}$, the arrow whose tail is on
strand $a$ and whose head is on the new ``$\infty$'' strand and extending
in a bracket-preserving way, using the commutator of the stacking product
as the bracket on $\calA^w({S\sqcup\{\infty\}})$. Then note that
$\FL(S)\subset\calA^w({S\sqcup\{\infty\}})$ is invariant under
conjugation by $D$ and let $C_D$ denote this conjugation action.

\vskip 1mm
  \parpic[l]{\begin{picture}(0,0)%
\includegraphics{figs/Topology.pstex}%
\end{picture}%
%
%  pstex_opts: -m 0.3333 
%
\setlength{\unitlength}{1316sp}%
\begingroup\makeatletter\ifx\SetFigFont\undefined%
\gdef\SetFigFont#1#2#3#4#5{%
  \reset@font\fontsize{#1}{#2pt}%
  \fontfamily{#3}\fontseries{#4}\fontshape{#5}%
  \selectfont}%
\fi\endgroup%
\begin{picture}(1052,1180)(0,-190)
\end{picture}%
  \hspace{-6pt}}%
  \noindent%
 This is a direct analog of the Artin action of the pure braid
groups $\PuB_n$ /
$\PwB_n$ on the free group $\FG(n)$.

\item $\glosm{pidowncap}{\pi_{\!\downcap}}(D)$ is the result of
adding a bullet at the bottom of every strand of $D$, in the
same sense as in Section~\ref{sssec:Family}.  Equivalently,
$\pi_{\!\downcap}=\delta\act\prod_{a\in S}h\eta^a$ is the composition
of $\delta$ with ``delete all head strands''. The target space of
$\pi_{\!\downcap}$ is $\calA^w(\emptyset;S)$, which is the symmetric algebra
$\calS(\CW(S))$ generated by wheels.

\end{myitemize}

\begin{proposition} $D$ is determined by the above three quantities
$\pi_A(D)$, $C_D$, and $\pi_{\!\downcap}(D)$.
\end{proposition}

\begin{proof} As in Section~\ref{subsec:AT}, every
$D\in\calA^w_{\exp}(S)$ can be written uniquely in the
form $D=e^{l\lambda} e^{\iota\omega}$, where $\lambda\in\FL(S)^S$
and $\omega\in\CW(S)$. One may easily verify that $\pi_{\!\downcap}(D)$ is
$\omega$, that $C_D$ is the exponential of the derivation in $\tder_S$
corresponding to $\lambda$, and that $\pi_A(D)$ determines the part of
$\lambda$ lost by the projection $\FL(S)^S\to\tder_S$.\qed
\end{proof}

\noindent{\em Proof of Theorem~\ref{thm:GammaLambda}.} For $\lambda\in\FL(S)^S$ let $\lambda'=\Gamma(\lambda)$.
Comparing Figures~\ref{fig:El} and~\ref{fig:Ef}, we find that the $\omega$
parts drop out and we need to prove, schematically, that in
$\calA^w_{\exp}(S)$,
\[ \input{figs/ElEf.pstex_t} \]
A simple degree 1 calculation shows that $\pi_A(A)=\pi_A(B)=0$.
The CP relation of Section~\ref{sssec:Family} shows that
$\pi_{\!\downcap}(A)=\pi_{\!\downcap}(B)=0$. Finally, it is easy to verify that
$C_A=e^{-\partial_\lambda}$ while $C_B=C^{\lambda'}$, and hence $C_A=C_B$
follows from Proposition~\ref{prop:Gamma}.\qed

\draftcut
\section{Some Computations} \label{sec:Computations}

\subsection{Tangle Invariants} \label{subsec:TangleInvariants}

\subsubsection{The General Framework}
Recall from~\cite{WKO2} that the assignment
$\glosm{Zw}{Z^w}\colon\glosm{overcrossing}{\overcrossing} \mapsto
\exp(\glosm{rightarrowdiagram}{\rightarrowdiagram})
\glosm{virtualcrossing}{\virtualcrossing}$ defined on
$S$-component tangles and taking values in $\calA_{\exp}^w(S)$
(where $\rightarrowdiagram$ denotes an arrow connecting the upper strand
to the lower strand and exponentiation is in a formal sense) defines an
invariant of tangles with values in $\calA_{\exp}^w(S)$. We'd
like to compute $Z^w$ (more precisely, its logarithm), in as much as
possible, using both the $\TW_l(S)$-valued \cite{AT}-presentation $E_l$
or using the $\TW_s(S)$-valued factored presentation $E_f$ (recall
Figure~\ref{fig:diagram}).

We let $\glosm{Rl}{R_l^+}(a,b)$ and $\glosm{Rs}{R_s^+}(a,b)$ denote the value
$\glosm{Rab}{R(a,b)} = Z^w\left(\underset{a\ b}{\overcrossing}\right)$
of the positive crossing in $\TW_l$ and $\TW_s$, respectively, and
similarly, let $R_l^-(a,b)$ and $R_s^-(a,b)$ denote the value
$R^{-1}(a,b)=Z^w\left(\underset{b\
a}{\glosm{undercrossing}{\undercrossing}}\right)$ of the negative crossing
in $\TW_l$ and $\TW_s$, respectively (for both signs we label the upper
strand $a$ and the lower strand $b$). That is,
\[
  Z^w\left(\underset{a\ b}{\overcrossing}\right)
    = R^+_l(a,b)\act E_l = R^+_s(a,b)\act E_s
\qquad\text{and}\qquad
  Z^w\left(\underset{b\ a}{\undercrossing}\right)
    = R^-_l(a,b)\act E_l = R^-_s(a,b)\act E_s.
\]
One may easily verify that $R^\pm_{l,s}(a,b)=(a\to 0,\,b\to\pm
a;\,0)_{l,s}$\footnoteC{In computer talk, this is
\mathinclude{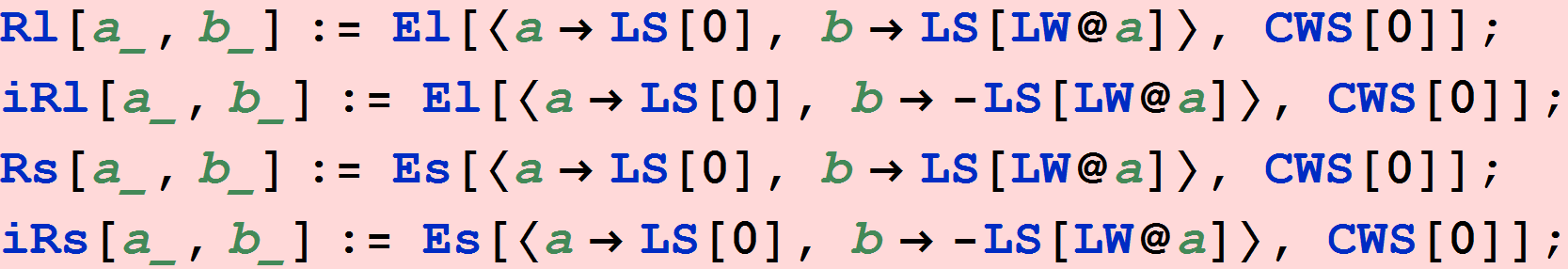}
}, and it is a simple exercise to verify that $R$ satisfies the Yang-Baxter
/ Reidemeister 3 relation $R_{l,s}^+(1,2)\ast R^+_{l,s}(1,3)\ast
R^+_{l,s}(2,3) = R^+_{l,s}(2,3)\ast R^+_{l,s}(1,3)\ast
R_{l,s}^+(1,2)$\footnoteC{
Indeed, here's a computer verification in $E_l$, to degree 5:

\shortdialoginclude{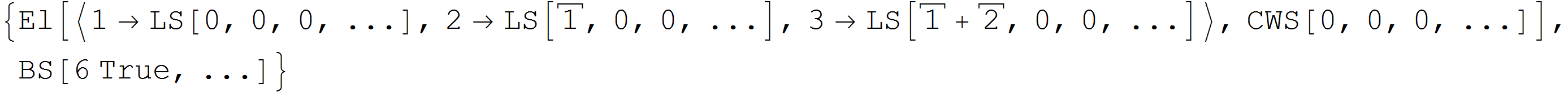}
}.

\subsubsection{The Knot $8_{17}$ and the Borromean Tangle} In this short
section we evaluate $Z^w$ on the knot $8_{17}$ and on the Borromean tangle,
both shown in Figure~\ref{fig:817AndBorromean}. An expanded version of this
section appears as \cite[Sections \ref{KBH-subsec:Demo1} and
\ref{KBH-subsec:Demo2}]{KBH}.

\begin{figure}
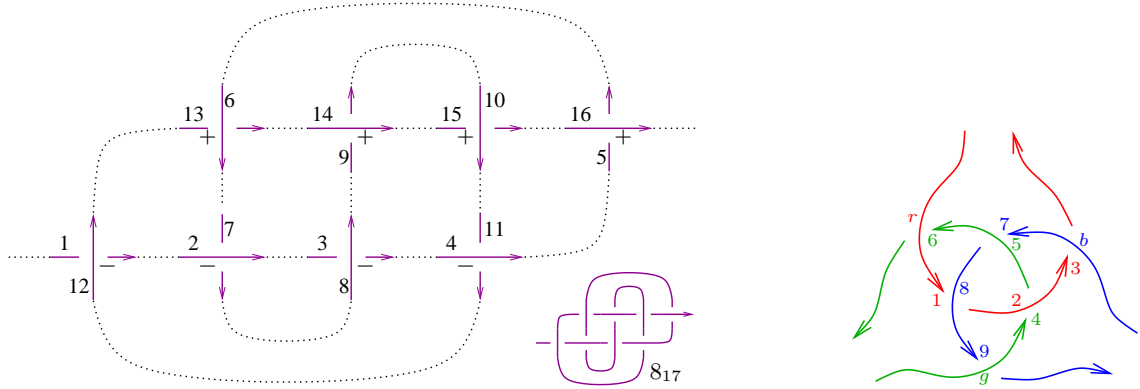

\[ \input{figs/817.pstex_t}\hspace{0.8in}\input{figs/BorromeanTangle.pstex_t} \]
\caption{The knot $8_{17}$ and the Borromean tangle.} \label{fig:817AndBorromean}
\end{figure}

For the 8-crossing knot $8_{17}$ we need to take 8 copies of $R^\pm_s$ with
strands labelled 1 through 16 as in Figure~\ref{fig:817AndBorromean}, and
then stitch strands 1 to 2, 2 to 3, etc\footnoteC{Here it is, to degree 6:

\dialogincludewithlink{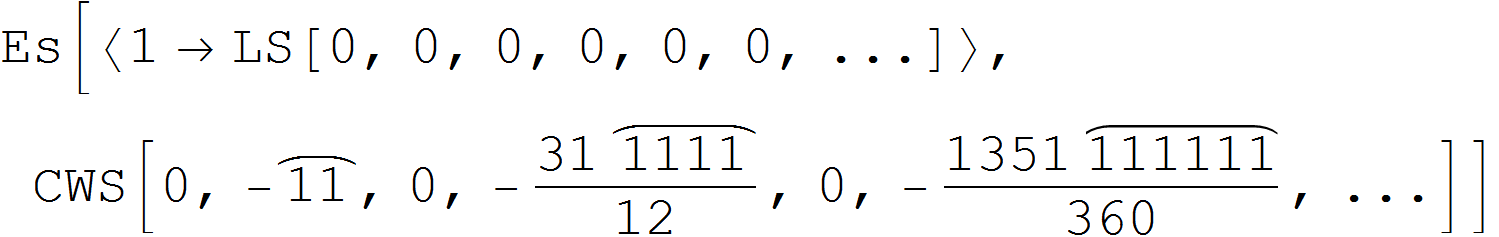}
}.
This is done using $dm$ operations, and hence we cannot use the $E_l$
presentation.

Similarly for the 6-crossings Borromean tangle we need 6 copies of
$R^\pm_s$ followed by some stitching\footnotemarkC. A colourful
evaluation of the Borromean tangle appears in \cite[Section
\ref{KBH-subsec:Demo2}]{KBH}.

\footnotetextC{To degree 4, we get

\dialogincludewithlink{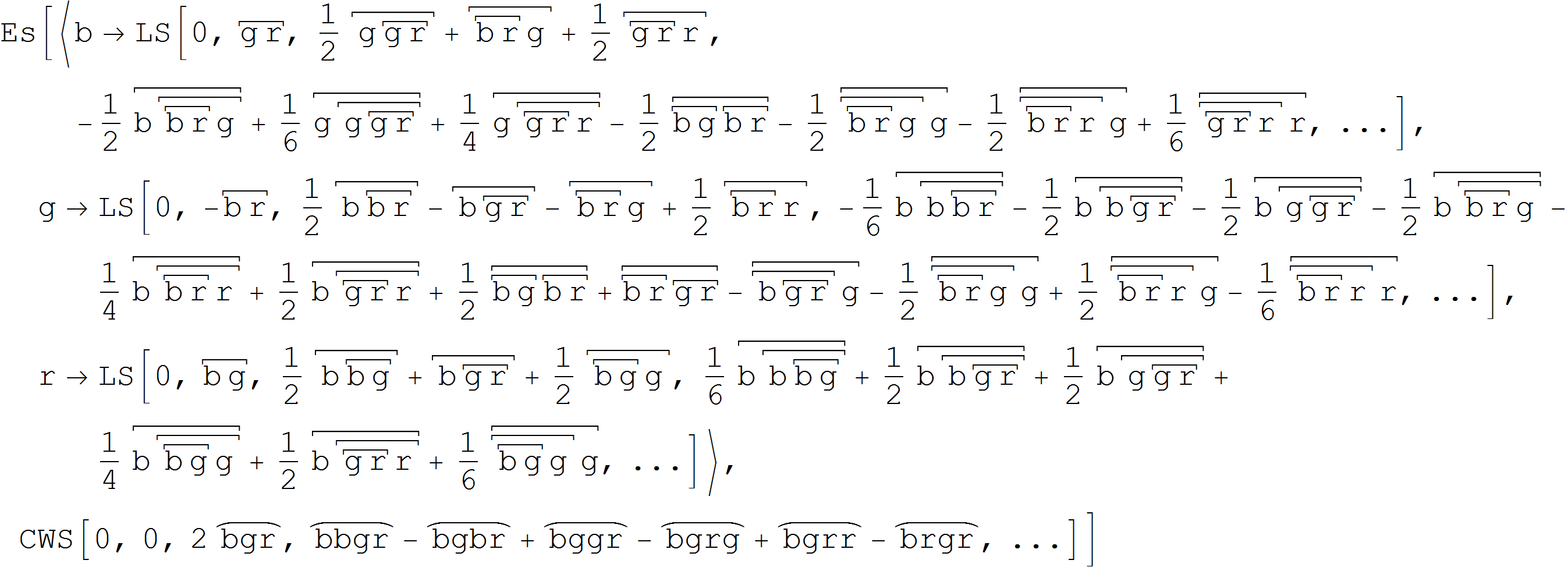}
}

\draftcut
\subsection{Solutions of the Kashiwara-Vergne Equations}
\label{subsec:SolKV} In \cite[Section \ref{2-subsec:wTFo}]{WKO2} we
found that in order to construct a homomorphic expansion $Z^w$ for the
class $\wTFo$ of orientable w-tangled foams, defined there, we need to
find elements $\glosm{V}{V} =
Z^w(\glosm{PlusVertex}{\raisebox{-1mm}{\begin{picture}(0,0)%
\includegraphics{figs/PlusVertex.pstex}%
\end{picture}%
%
%  pstex_opts: -m 0.8 
%
\setlength{\unitlength}{3947sp}%
\begingroup\makeatletter\ifx\SetFigFont\undefined%
\gdef\SetFigFont#1#2#3#4#5{%
  \reset@font\fontsize{#1}{#2pt}%
  \fontfamily{#3}\fontseries{#4}\fontshape{#5}%
  \selectfont}%
\fi\endgroup%
\begin{picture}(211,211)(3439,-1160)
\end{picture}%
}}) \in
\calA^w_{\exp}({x,y})$\footnotemarkC\ and $\glosm{Cap}{\Cap}
= Z^w(\glosm{upcap}{{\raisebox{-1mm}{\begin{picture}(0,0)%
\includegraphics{figs/SmallCap.pstex}%
\end{picture}%
%
%  pstex_opts: -m 0.8 
%
\setlength{\unitlength}{3947sp}%
\begingroup\makeatletter\ifx\SetFigFont\undefined%
\gdef\SetFigFont#1#2#3#4#5{%
  \reset@font\fontsize{#1}{#2pt}%
  \fontfamily{#3}\fontseries{#4}\fontshape{#5}%
  \selectfont}%
\fi\endgroup%
\begin{picture}(64,201)(151,-1123)
\end{picture}%
}}})
\in \calA^w_{\exp}({\raisebox{-1mm}{\begin{picture}(0,0)%
\includegraphics{figs/SmallCap.pstex}%
\end{picture}%
%
%  pstex_opts: -m 0.8 
%
\setlength{\unitlength}{3947sp}%
\begingroup\makeatletter\ifx\SetFigFont\undefined%
\gdef\SetFigFont#1#2#3#4#5{%
  \reset@font\fontsize{#1}{#2pt}%
  \fontfamily{#3}\fontseries{#4}\fontshape{#5}%
  \selectfont}%
\fi\endgroup%
\begin{picture}(64,201)(151,-1123)
\end{picture}%
}}_x)$\footnotemarkT~\footnotemarkC\ that are
required to satisfy the three equations in \eqref{eq:HardR4}
and \eqref{eq:UnzipEquations} below. Recall from \cite[Section
\ref{2-subsec:EqWithAT}]{WKO2} that these equations are equivalent
to equations considered by Alekseev and Torossian in~\cite{AT}
(see \cite[Equation \ref{2-eq:ATKVEqns}]{WKO2} and \cite[Section
5.3]{AT}), and that the latter equations were shown in \cite[Section
5.2]{AT} to be equivalent to the Kashiwara-Vergne equations
of~\cite{KashiwaraVergne:Conjecture}.
\footnotetextT{$\Cap$ is called $C$ in~\cite{WKO2} and we trust
that the other minor notational differences with~\cite{WKO2} will
cause no difficulty to the reader. Note that $\calA^w({\raisebox{-1mm}{\begin{picture}(0,0)%
\includegraphics{figs/SmallCap.pstex}%
\end{picture}%
%
%  pstex_opts: -m 0.8 
%
\setlength{\unitlength}{3947sp}%
\begingroup\makeatletter\ifx\SetFigFont\undefined%
\gdef\SetFigFont#1#2#3#4#5{%
  \reset@font\fontsize{#1}{#2pt}%
  \fontfamily{#3}\fontseries{#4}\fontshape{#5}%
  \selectfont}%
\fi\endgroup%
\begin{picture}(64,201)(151,-1123)
\end{picture}%
}}_S)$ is
$\calA^w(S)$ with $CP$ relations imposed at the tops of the
strands; compare with Section~\ref{sssec:Family}.}
\addtocounter{footnoteC}{-1}
\footnotetextC{For computations, we use the $E_s$ presentation
for $V$. As $V$ is presented in $\TW_s(\{x,y\})$, it is of the form
$V=((x\to\alpha,\,y\to\beta);\,\gamma)_s$, where $\alpha,\beta\in\FL(x,y)$
and $\gamma\in\CW(x,y)$, and where the coefficients of $\alpha$, $\beta$,
and $\gamma$, what we call the $\alpha$s, the $\beta$s, and the $\gamma$s,
will be determined later. The first line below sets $\alpha$, $\beta$,
and $\gamma$ to be series with yet-unknown coefficients, and the second
line sets $V$ to be the appropriate combination of $\alpha$, $\beta$,
and $\gamma$:
\mathinclude{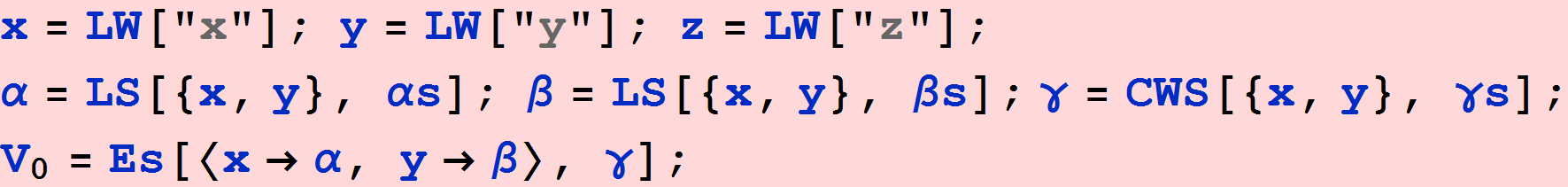}
(for a technical reason, in computations we use the symbol ${\mathtt
V}_{\mathtt 0}$ to denote $V$).
}
\addtocounter{footnoteC}{1}
\footnotetextC{Similarly, $\Cap$ is presented in $\TW_s(x)$. As it is made
only of wheels, its tree part is $0$, or the Lie series \verb$LS[0]$. The
wheels part of $\Cap$ is a series $\kappa\in\CW(x)$ whose
coefficients are the yet-unknown $\kappa$s:
\shortmathinclude{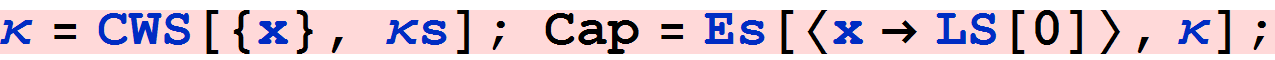}
}

The purpose of this section is to trace through all that at the level
of actual computations. Let us start by recalling from~\cite{WKO2}
the equations for $V$ and for $\Cap$. The first of those is the $R4$
equation \cite[\eqref{2-eq:HardR4}]{WKO2}, $V^{12}\glosm{R}{R}^{(12)3} =
R^{23}R^{13}V^{12}$, coming from the picture
\[ \input{figs/R4ToEquation.pstex_t}. \]
In the language of this paper, and denoting the three strands $x$, $y$, and
$z$, this equation becomes
\begin{equation}\label{eq:HardR4}
  V\ast (R(x,z)\act d\Delta^x_{xy}) = R(y,z)\ast R(x,z)\ast V\footnotemarkC
\end{equation}

The second and the third, ``unitarity'' and the ``cap equation'',
\cite[\eqref{2-eq:unitarity}]{WKO2} and \cite[\eqref{2-eq:CapEqn}]{WKO2},
are the equations
\addtocounter{footnoteC}{-1}
\begin{equation} \label{eq:UnzipEquations}
  V\ast(V\act\dA) = 1\quad\text{in }\calA^w({x,y})
  \quad\text{and}\quad
  V\ast(\Cap\act d\Delta^x_{xy}) = \Cap(\Cap\act d\sigma^x_y)
  \quad\text{in }
  \calA^w(\raisebox{-1mm}{\begin{picture}(0,0)%
\includegraphics{figs/SmallCap.pstex}%
\end{picture}%
%
%  pstex_opts: -m 0.8 
%
\setlength{\unitlength}{3947sp}%
\begingroup\makeatletter\ifx\SetFigFont\undefined%
\gdef\SetFigFont#1#2#3#4#5{%
  \reset@font\fontsize{#1}{#2pt}%
  \fontfamily{#3}\fontseries{#4}\fontshape{#5}%
  \selectfont}%
\fi\endgroup%
\begin{picture}(64,201)(151,-1123)
\end{picture}%
}_{x,y}),\footnotemarkC
\end{equation}
which come from the two unzip operations,
\[ {
  \def\u{{\glosm{unzip}{u}}}
  \def\yu{{\yellowm{u}}}
  \input{figs/Unzips.pstex_t}.
} \]

\footnotetextC{\label{footC:VCapEqns} The three equations in~\eqref{eq:HardR4}
and~\eqref{eq:UnzipEquations} are coded as follows:
\mathinclude{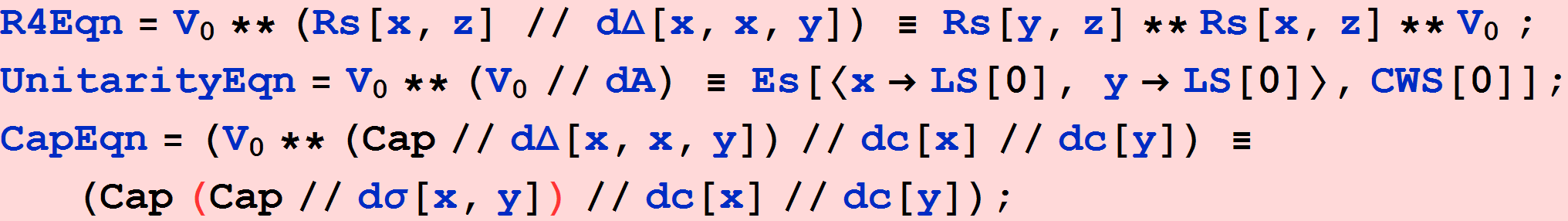}
}

Solving Equations~\eqref{eq:HardR4} and~\eqref{eq:UnzipEquations} degree
by degree with the initial condition $\alpha=-y/2+\ldots$ we find that
one possible solution, given in the factored presentation, is
\begin{multline*} \small
  V=E_f\left(
      x\to -\frac{\ob{xy}}{24}
        + \frac{7\ob{x\ob{x\ob{xy}}}}{5760} - \frac{7\ob{x\ob{\ob{xy}y}}}{5760}
        + \frac{\ob{\ob{\ob{xy}y}y}}{1440} +\ldots,
        \right. \\
      y\to \frac{\ob{x}}{2} - \frac{\ob{xy}}{12}
        + \frac{\ob{x\ob{x\ob{xy}}}}{5760} + \frac{\ob{x\ob{\ob{xy}y}}}{720}
        - \frac{\ob{\ob{\ob{xy}y}y}}{720} +\ldots; \\
    \left. \vphantom{\frac{\ob{x\ob{x\ob{xy}}}}{720}}
    - \frac{\wideparen{xy}}{48} + \frac{\wideparen{xxxy}}{2880}
      + \frac{\wideparen{xxyy}}{2880} + \frac{\wideparen{xyxy}}{5760}
      + \frac{\wideparen{xyyy}}{2880} + \ldots
  \right)_{\!\!s},
\end{multline*}
and $\small\Cap = - \wideparen{xx}/96 + \wideparen{xxxx}/11,520
- \wideparen{xxxxxx}/725,760 + \ldots$\footnoteC{We set the initial
condition for $\alpha$ in degree 1, then
declare that $\alpha$, $\beta$, $\gamma$, and $\kappa$ are the series which
solve equations \verb$R4Eqn$, \verb$UnitarityEqn$, and \verb$CapEqn$, and
then print the values of $V$ and $\kappa$ (note
the $\hbar^{-1}$ that comes with \verb$R4Eqn$ --- it indicates a degree
shift --- \verb$R4Eqn$ in degree $k$ only puts conditions on our unknowns
at degree $k-1$):

\dialogincludewithlink{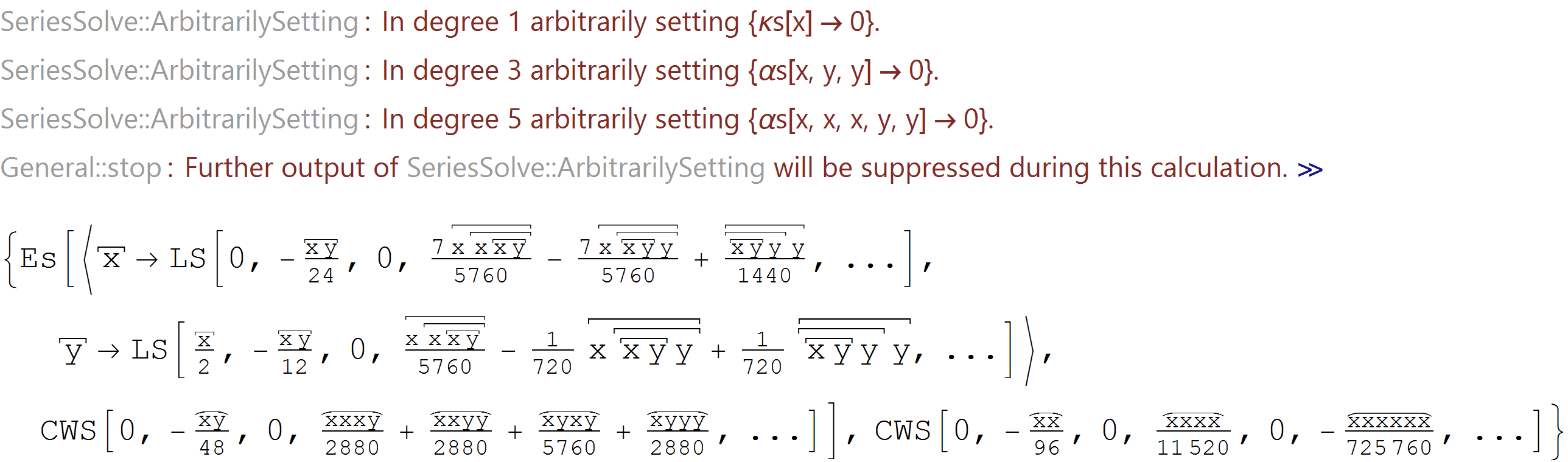}

The solutions of~\eqref{eq:HardR4} and~\eqref{eq:UnzipEquations} are
not unique, and hence occasionally \verb$SeriesSolve$ encounters a
coefficient whose value is not determined by the equations. When this
happens its default action is to set the missing coefficient to $0$. In
the computation this happened to the coefficient of $\wideparen{x}$
in $\kappa$ and to the coefficient of $\ob{\ob{xy}y}$ in $\alpha$.
}.
Note that according
to~\cite{WKO3}, $\Cap$ is always $\sum a_n\wideparen{x^n}$, where
$\sum a_n\hbar^n = \frac14\log\left( \frac{\hbar/2}{\sinh\hbar/2}
\right)$\footnoteC{Indeed, the series below matches with the computation
of $\kappa$, above.

\shortdialoginclude{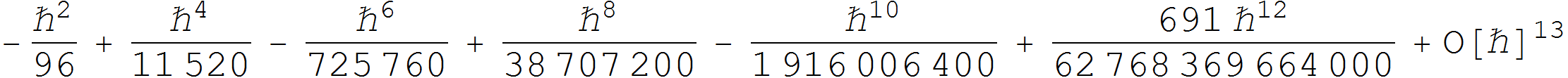}
}.

We can also write $V$ in the lower-interlaced presentation:
\begin{multline*} \small
  V=E_l\left(
      x\to -\frac{\ob{xy}}{24} + \frac{\ob{x\ob{xy}}}{96}
        + \frac{\ob{x\ob{x\ob{xy}}}}{2880} - \frac{\ob{x\ob{\ob{xy}y}}}{480}
        + \frac{\ob{\ob{\ob{xy}y}y}}{1440} +\ldots,
        \right. \\
      y\to \frac{\ob{x}}{2} - \frac{\ob{xy}}{12} + \frac{\ob{x\ob{xy}}}{96}
        + \frac{\ob{x\ob{x\ob{xy}}}}{960} - \frac{\ob{x\ob{\ob{xy}y}}}{320}
        + \frac{\ob{\ob{\ob{xy}y}y}}{720} +\ldots;
    \\ \left. \vphantom{\frac{\ob{x\ob{x\ob{xy}}}}{720}}
    - \frac{\wideparen{xy}}{48} + \frac{\wideparen{xxxy}}{2880}
      + \frac{\wideparen{xxyy}}{2880} + \frac{\wideparen{xyxy}}{5760}
      + \frac{\wideparen{xyyy}}{2880} + \ldots
  \right)_{\!\!s},\footnotemarkC
\end{multline*}
($\Cap$ is the same in both presentations).

\footnotetextC{We could re-compute $V$ in $E_l$ by making some simple
modifications to the input lines in~\ref{footC:VCapEqns}, but it is easier
to use our tools and convert between the two presentations:

\shortdialoginclude{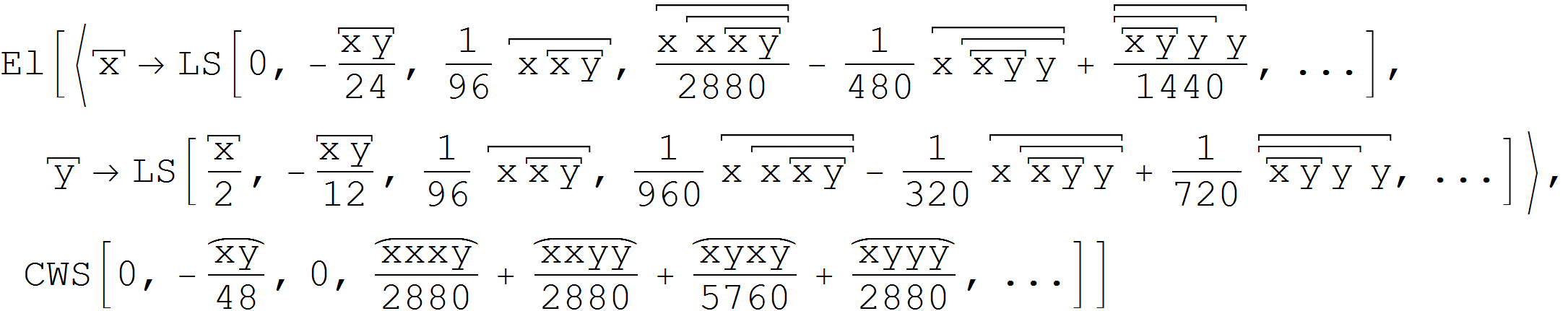}
}

Recall from \cite[Section~\ref{2-subsec:EqWithAT}]{WKO2} and from
Comment~\ref{com:ElAT} that the tree part of ``our'' $V$, taken in
the lower-interlaced presentation, is $\log F^{21}$, where $\glosm{F}{F}$ is
the solution of ``generalized KV problem'' of \cite[Section~5.3]{AT}
and where the superscript $21$ means ``interchange the role of $x$
and $y$''. Thus using the notation of \cite{AT} a solution to degree 4 of the
generalized KV problem is\footnotemarkC
\[ \small
  \log F \!=\! \left(\!
    \frac{\ob{y}}{2} \!+\! \frac{\ob{xy}}{12} \!+\! \frac{\ob{\ob{xy}y}}{96}
      \!-\! \frac{\ob{x\ob{x\ob{xy}}}}{720}
      \!+\! \frac{\ob{x\ob{\ob{xy}y}}}{320}
      \!-\!\frac{\ob{\ob{\ob{xy}y}y}}{960},\,
    \frac{\ob{xy}}{24} \!+\! \frac{\ob{\ob{xy}y}}{96}
      \!-\! \frac{\ob{x\ob{x\ob{xy}}}}{1440}
      \!+\! \frac{\ob{x\ob{\ob{xy}y}}}{480}
      \!-\! \frac{\ob{\ob{\ob{xy}y}y}}{2880}
  \right).
\]

\footnotetextC{The more authoritative version, of course, is the one printed
directly by the computer:

\shortdialoginclude{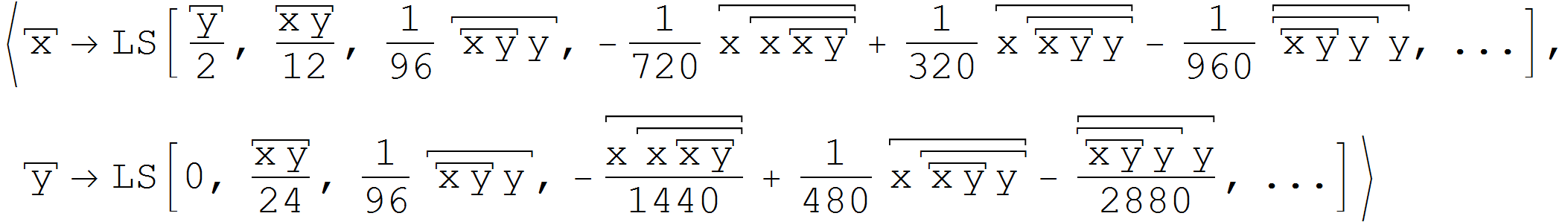}
}

Next, we'd like to compute a solution of the original Kashiwara-Vergne
equations of~\cite{KashiwaraVergne:Conjecture}. These are the two equations
below, written for unknowns $\glosm{fg}{f,g}\in\FL(x,y)$:
\begin{equation} \label{eq:KV1}
  x + y - \log e^ye^x = (1-e^{-\ad x})f + (e^{\ad y}-1)g,
\end{equation}
\begin{equation} \label{eq:KV2}
  \atdiv_xf + \atdiv_y g
  = \frac12\tr_u\left(\left(
    \frac{\ad x}{e^{\ad x}-1} + \frac{\ad x}{e^{\ad x}-1}
    - \frac{\ad \BCH(x,y)}{e^{\ad \BCH(x,y)}-1}
  \right)(u)\right).
\end{equation}

By tracing the definitions of the comparison map $\kappa$ which appears in
\cite[Theorem~5.8]{AT}, we find that a solution $(f,g)$ of the
Kashiwara-Vergne equations can be computed from $\log F$ via the formula
\[ (f,g) = \frac{e^{\ad(\log F)}-1}{\ad(\log F)}(\calE(\log F)), \]
where $\glosm{calE}{\calE}$ denotes the Euler operator, which multiplies every
homogeneous element by its degree. To degree 4, we find\footnotemarkC\ that
\[ \small
  (f,g) \!=\! \left(\!
    \frac{\ob{y}}{2} \!+\! \frac{\ob{xy}}{6} \!+\! \frac{\ob{\ob{xy}y}}{24}
      \!-\! \frac{\ob{x\ob{x\ob{xy}}}}{180}
      \!+\! \frac{\ob{x\ob{\ob{xy}y}}}{80}
      \!+\!\frac{\ob{\ob{\ob{xy}y}y}}{360},\,
    \frac{\ob{xy}}{12} \!+\! \frac{\ob{\ob{xy}y}}{24}
      \!-\! \frac{\ob{x\ob{x\ob{xy}}}}{360}
      \!+\! \frac{\ob{x\ob{\ob{xy}y}}}{120}
      \!+\! \frac{\ob{\ob{\ob{xy}y}y}}{180}
  \right).
\]

\footnotetextC{With higher authority:

\shortdialoginclude{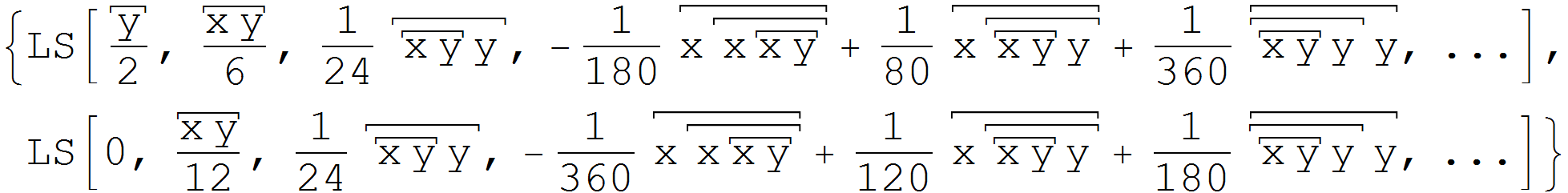}

We can then verify that $(f,g)$ indeed satisfy Equations~\eqref{eq:KV1}
and~\eqref{eq:KV2}, at least to degree 9:

\dialoginclude{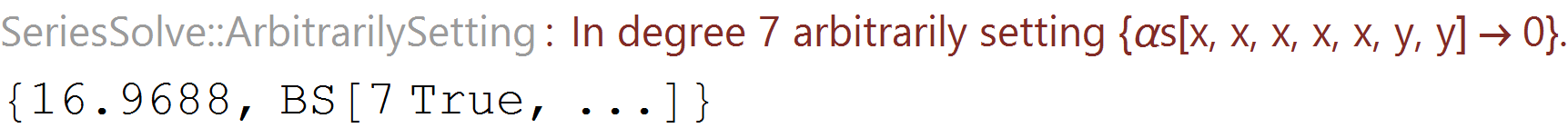}

Of course, we could have simply solved Equations~\eqref{eq:KV1}
and~\eqref{eq:KV2} directly:

\dialogincludewithlink{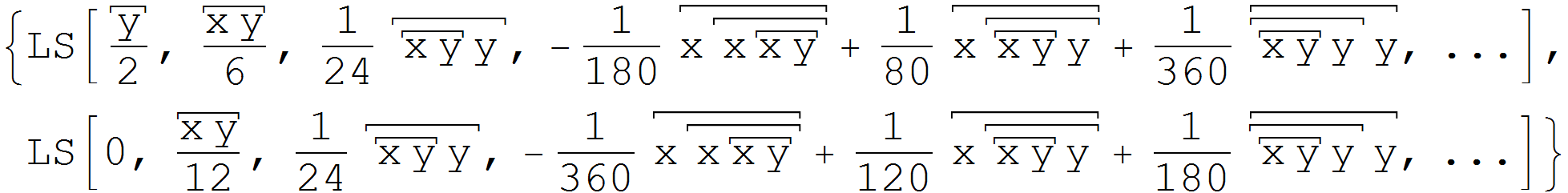}

(To the degree shown, the results are the same. But starting at degree 8
they diverge as the solutions are non-unique.)
}

\draftcut
\subsection{The involution $\tau$ and the Twist Equation} \label{subsec:tau}
Alekseev and Torossian~\cite[Section 8.2]{AT} construct an
involution $\glosm{tau}{\tau}$ on the set $\SolKV$ of solutions of the
Kashiwara-Vergne equations. Phrased using the language
of~\cite{WKO2}, Alekseev and Torossian define a map
$\tau\colon\calA^w(\uparrow_2)\to\calA^w(\uparrow_2)$ by
$\tau(V)\coloneqq R(1,2)V^{21}\Theta^{-1/2}$,
where $\glosm{Theta}{\Theta}^s=e^{st}$ and
$t=\rightarrowdiagram+\leftarrowdiagram\in\calA^w(\uparrow_2)$. They
then prove that $\tau$ restricts to an involution of the set of solutions
Equations~\eqref{eq:HardR4} and~\eqref{eq:UnzipEquations}. It is not known
if $\tau$ is different from the identity; in other words, it is not known
if every $V$ satisfying~\eqref{eq:HardR4} and~\eqref{eq:UnzipEquations}
also satisfies the ``Twist Equation''
\begin{equation} \label{eq:Twist} V=\tau(V). \end{equation}

  \parpic[l]{\begin{picture}(0,0)%
\includegraphics{figs/Topology.pstex}%
\end{picture}%
%
%  pstex_opts: -m 0.3333 
%
\setlength{\unitlength}{1316sp}%
\begingroup\makeatletter\ifx\SetFigFont\undefined%
\gdef\SetFigFont#1#2#3#4#5{%
  \reset@font\fontsize{#1}{#2pt}%
  \fontfamily{#3}\fontseries{#4}\fontshape{#5}%
  \selectfont}%
\fi\endgroup%
\begin{picture}(1052,1180)(0,-190)
\end{picture}%
  \hspace{-6pt}}%
  \noindent%
 In topology, the Twist Equation is essential for the
compatibility between $\glosm{Zu}{Z^u}$ and $Z^w$;
see~\cite[Section~4.7]{WKO2}. So it is not known if ``every $Z^w$ is
compatible with some $Z^u$''.

Below the dark line we verify that to degree 6, ``our'' $V$ satisfies
the Twist Equation~\eqref{eq:Twist}\footnoteC{We define $\Theta${\tt
l[x,y,s]} to be $e^{st}$ in the $E_l$ presentation in a straightforward
manner, then convert it to the $E_s$ presentation, and then print its
value in both the $E_l$ and $E_s$ presentations:

\dialoginclude{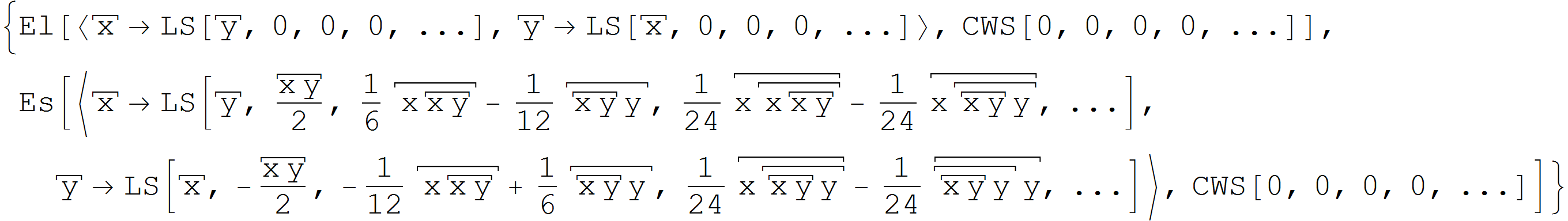}

This done, the computation of $\tau(V_0)$ and the verification that it
is equal to $V_0$ to degree 6 s routine:

\shortdialoginclude{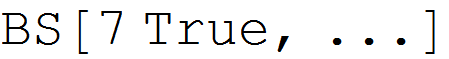}
}.

Following that, we reproduce the results of Albert, Harinck, and
Torossian~\cite{AlbertHarinckTorossian:KV78}, who studied the
linearizations
\begin{equation} \label{eq:LinearizedKV}
  [x,A]+[y,B] = 0
  \qquad\text{and}\qquad
  \atdiv_xA+\atdiv_yB = 0
  \qquad\text{with }A,B\in\FL(x,y)
\end{equation}
of Equations~\eqref{eq:KV1} and~\eqref{eq:KV2} (which are
equivalent to~\eqref{eq:HardR4} and~\eqref{eq:UnzipEquations}), and the
linearization of Equation~\eqref{eq:Twist},
\begin{equation} \label{eq:LinearizedTwist}
  A(x,y) = B(y,x).
\end{equation}

We find\footnotemarkC\ that up to degree 16, the dimensions
of the spaces of solutions of~\eqref{eq:LinearizedKV} and of
\eqref{eq:LinearizedKV}$\wedge$\eqref{eq:LinearizedTwist} are the same
and are given by the following table:

\footnotetextC{We solve for series $A$ and $B$
satisfying~\eqref{eq:LinearizedKV}. These equations are linear, so the
printed solution is $0$. Yet we store messages produced by {\tt
LinearSolve} in a stream called {\tt msgs}. As {\tt LinearSolve}
progresses, it outputs messages detailing which coefficients were set in an
arbitrary manner in each degree, and the dimension of the space of
solutions in each degree can be read from that information:

\dialoginclude{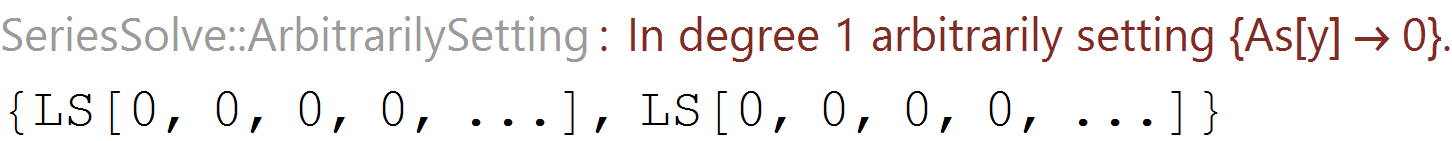}

Next, we read the stream {\tt msgs}, just to explore its format:

\shortdialoginclude{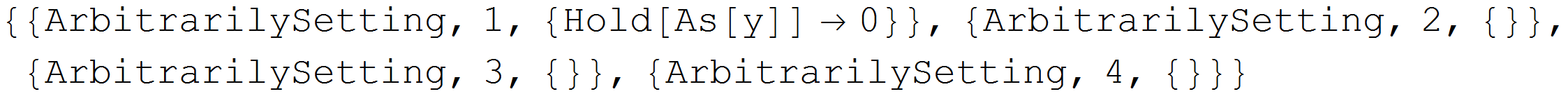}

Next we compute $A$ to degree 12, and read only the dimensions information
contained in {\tt msgs}:

\shortdialogincludewithlink{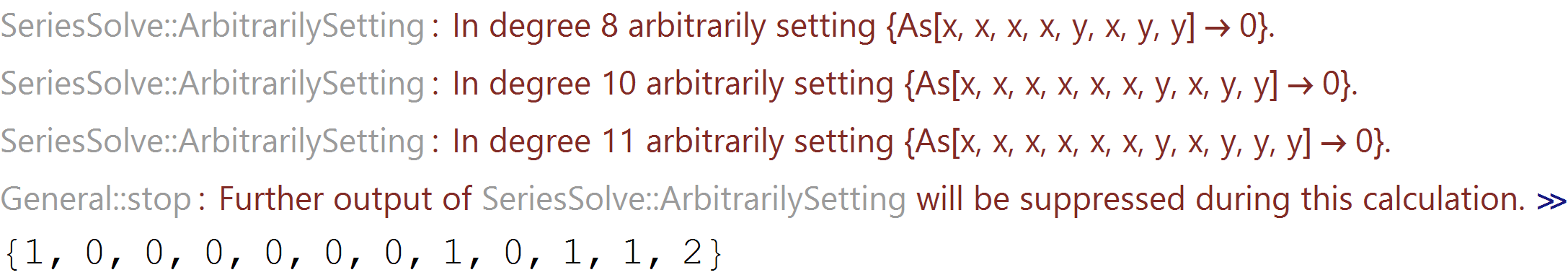}

Finally we do the same, but now adding Equation~\eqref{eq:LinearizedTwist}:

\dialogincludewithlink{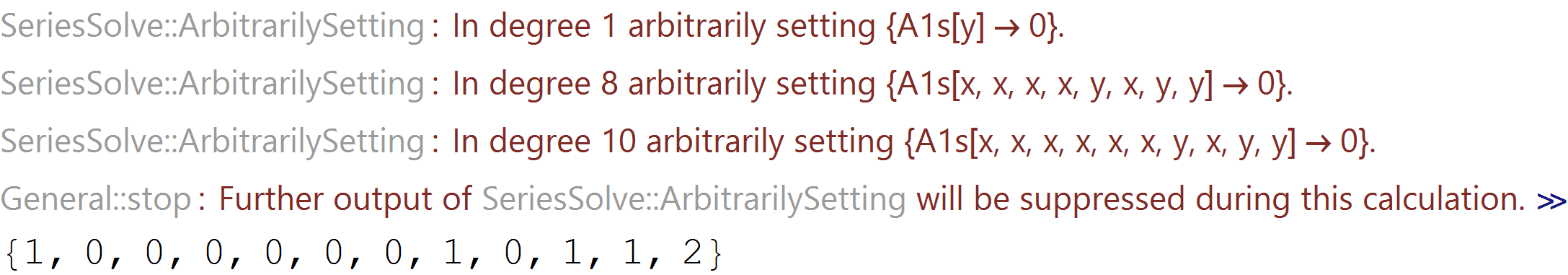}
}

\begin{equation} \label{tab:dims}
\begin{tabular}{|c|c|c|c|c|c|c|c|c|c|c|c|c|c|c|c|c|}
\hline
$\deg A,B$ & 1 & 2 & 3 & 4 & 5 & 6 & 7 & 8 & 9 & 10 & 11 & 12 & 13 & 14 & 15 & 16 \\
\hline
dimension  & 1 & 0 & 0 & 0 & 0 & 0 & 0 & 1 & 0 &  1 &  1 &  2 &  2 &  3 &  3 &  5 \\
\hline
\end{tabular}
\end{equation}

Assuming that every solution of the KV equations to degree $k$
can be extended to a solution at all degrees (and similarly for
KV$\wedge$Twist)\footnoteT{I am not aware that this was ever proven for
KV (and/or KV$\wedge$Twist), yet a similar result holds for Drinfel'd
associators; see~\cite{Drinfeld:QuasiHopf, Drinfeld:GalQQ, Bar-Natan:NAT,
Bar-Natan:Associators}.}, the above table shows the number of degrees of
freedom for the solutions of KV (and/or KV$\wedge$Twist), in each degree.

\draftcut
\subsection{Drinfel'd Associators} \label{subsec:Associators}
It pains me to say so little
about Drinfel'd associators, but this is a computational paper and
everything we need about associators was already said elsewhere; e.g.,
in Drinfel'd's original papers \cite{Drinfeld:QuasiHopf, Drinfeld:GalQQ},
in my \cite{Bar-Natan:NAT, Bar-Natan:Associators}, and in earlier papers in
this series \cite{WKO2,WKO3}. Hence here I will only recall the few things
that are necessary in order to understand the computations below.

Recall that the Drinfel'd-Kohno algebra $\glosm{fraktn}{\frakt_n}$
is the completed graded Lie algebra with degree $1$ generators
$\{\glosm{tij}{t_{ij}} = t_{ji}\colon 1\leq i\neq j\leq n\}$ and relations
$[t_{ij},t_{kl}]=0$ when $i,j,k,l$ are distinct (``locality relations'')
and $[t_{ij}+t_{ik},t_{jk}]=0$ when $i,j,k$ are distinct (``4T
relations'')\footnoteC{We verify these relations, using obvious notation:

\shortdialoginclude{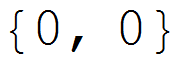}}. For any fixed $2\leq k\leq n$ the $k-1$
elements $\{t_{ik}\colon 1\leq i<k\}$ form a free subalgebra $\FL_{k-1}$
of $\frakt_n$, and $\frakt_n$ is an iterated semi-direct product of
these subalgebras:
\[ \frakt_n \cong
  ((\ldots(\FL_1\ltimes\FL_2)\ltimes\ldots)\ltimes\FL_{n-2})\ltimes\FL_{n-1}.
\]
Hence as a vector space, $\frakt_n$ has a basis with elements ordered
pairs $(k,w)$, where $2\leq k\leq n$ and $w$ is a Lyndon word in the
letters $\{1,\ldots,k-1\}$ (which really stand for
$\{t_{1k},\ldots,t_{k-1,k}\}$)\footnoteC{Hence for example,
$[t_{13},t_{12}]=-[t_{13},t_{23}]$ (the bracket of a generator of $\FL_3$
with the generator of $\FL_2$ is an element of $\FL_3$). In computer speak,
this is

\shortdialoginclude{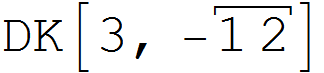}

Note that the head \verb$DK$ represents ``a basis element in a
Drinfel'd-Kohno algebra'', and that the Lyndon word $12$ becomes
$[t_{13},t_{23}]$ when interpreted in $\FL_3\subset\frakt_3$.

We could make the last output a bit friendlier by turning it into a
``Drinfel'd-Kohno Series'' (\verb$DKS$):

\shortdialoginclude{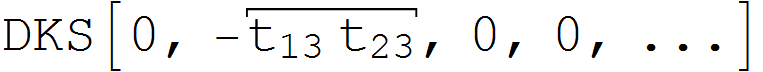}
}. 

The collection $\{\frakt_n\}$ of all Drinfel'd-Kohno algebras forms an
``operad'' (e.g.~\cite{Fresse:OperadsAndGT}). We only need to mention a
part of that structure here: that for any $n$ and $m$, there are many maps
$\frakt_n\to\frakt_m$. Namely, whenever $\left\{s_i\right\}_{i=1}^n$ is a
collection of disjoint subsets of $\{1,\ldots,m\}$ (some of which may be
empty), we have a morphism of Lie algebras
$\Psi\mapsto\Psi^{s_1,\ldots,s_n}$ mapping $\frakt_n$ to $\frakt_m$,
and defined by its values on the generators of $\frakt_n$ as follows:
\[ \left(t_{ij}\right)^{s_1,\ldots,s_n} \coloneqq
  \sum_{\alpha\in s_i,\,\beta\in s_j} t_{\alpha\beta}.\footnotemarkC
\]

\footnotetextC{As an example we repeat a single evaluation of a map
$\frakt_4\to\frakt_9$ twice. First using a complete and somewhat cumbersome
notation, and then using a shortened notation that works only if all
indices are single-digit:

\shortdialoginclude{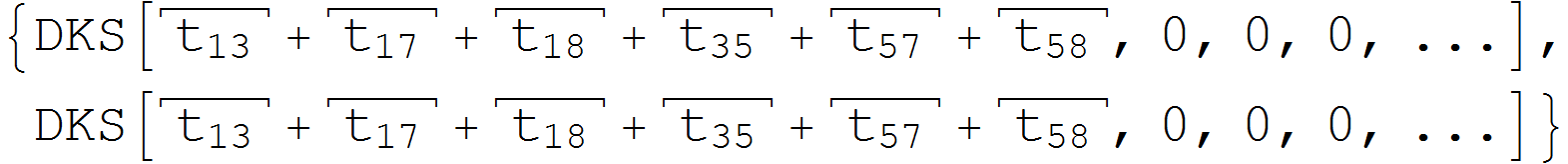}
}

Note also that by regarding elements of $\frakt_n$ as formal
exponentials and using the BCH product each $\frakt_n$ also acquires a
(non-commutative) group structure.\footnoteC{For example, in $\frakt_3$ the
elements $t_{12}$ and $t_{23}$ do not commute, and hence the product
$e^{t_{12}/2}e^{t_{23}/2}$ is messy. Yet by a 4T relation the elements
$t_{12}$ and $(t_{12})^{12,3}=t_{13}+t_{23}$ do commute, and hence the
product $e^{t_{12}/2}\left(e^{t_{12}/2}\right)^{12,3}$ is much simpler:

\shortdialoginclude{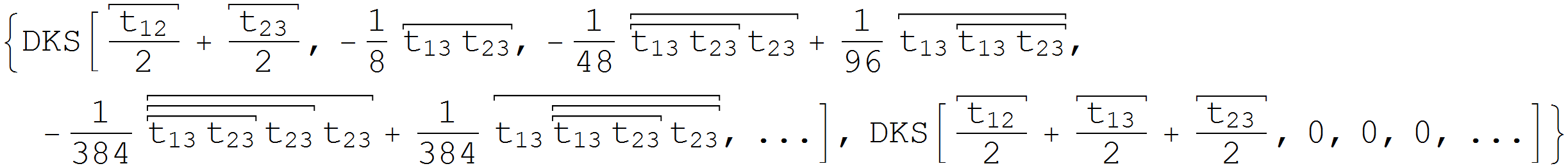}
}
By convention, when we think of $\frakt_n$ as a group, we refer to it as
``$\glosm{exptn}{\exp\frakt_n}$''.

We are finally in position to recall the definition of a
Drinfel'd associator. With $R=e^{t_{12}/2}\in\exp\frakt_2$, a
Drinfel'd associator is an element $\glosm{Phi}{\Phi}\in\exp\frakt_3$ which
satisfies the ``unitarity condition''~\eqref{eq:PhiUnitarity},
the pentagon equation~\eqref{eq:PhiPentagon}, and the hexagon
equations~\eqref{eq:PhiHexagons}:

\begin{align}
  \text{Unitarity}&\colon& \Phi^{321}
    &= \Phi^{-1}, \label{eq:PhiUnitarity} \\
  \pentagon&\colon& \Phi\cdot\Phi^{1,23,4}\cdot\Phi^{2,3,4}
    &= \Phi^{12,3,4}\cdot\Phi^{1,2,34}, \label{eq:PhiPentagon} \\
  \hexagon_\pm&\colon& (R^{\pm 1})^{12,3}
    &= \Phi\cdot(R^{\pm 1})^{2,3}\cdot(\Phi^{-1})^{1,3,2}\cdot
       (R^{\pm 1})^{1,3}\cdot\Phi^{3,1,2}. \label{eq:PhiHexagons}
\end{align}

A surprising result by Furusho~\cite{Furusho:Pentagon} (see
also~\cite{Bar-NatanDancso:Furusho}) states that in the context of
$\exp\frakt_n$ the hexagon equations follow from unitarity and the
pentagon, provided $\Phi$ is initialized to degree $2$ by
$\Phi=\exp\left([t_{13},t_{23}]/24+\text{higher
terms}\right)$.\footnoteC{Here's an associator $\Phi_0$, computed to
degree $6$. The data file~\citeweb{Phi.nb} contains a computation
of an associator to degree 10, higher than was previously
computed~\cite{Bar-Natan:NAT, Brochier:DrinfeldAssociators}.

\dialogincludewithlink{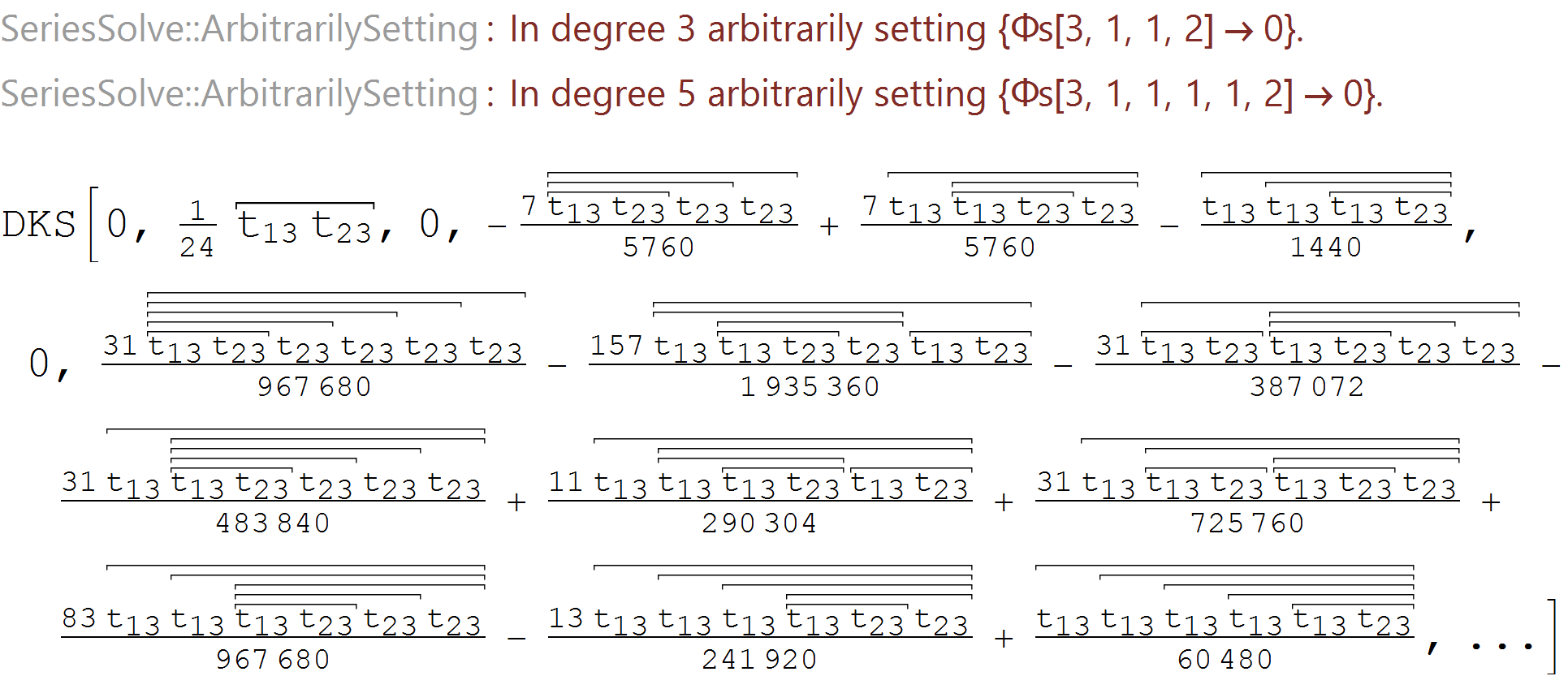}

To be on the safe side, we verify that $\Phi_0$ satisfies the hexagon
equations to degree $6$:

\dialoginclude{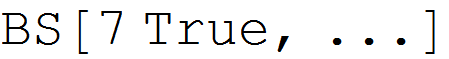}}

\draftcut
\parpic[r]{\input{figs/PhiV.pstex_t}}
\subsection{Associators in $\calA^w$} \label{subsec:wAssociators}
We know from~\cite[Section~1]{AT}
that a certain combination of four copies of $V$ makes a solution
of the pentagon equation, with values in $\tder_3$. In the language
of~\cite{WKO2}, this is the statement that $V$ is the $Z^w$-value
of a vertex, that four vertices can make a tetrahedron, and that the
$Z^w$-value $\glosm{PhiV}{\Phi_V}$ of a tetrahedron is an associator in
$\calA^w$ (see the figure on the right). Specifically,
\[ \Phi_V = (V\act\dA)^{12,3}(V\act\dA)^{1,2}V^{2,3}V^{1,23},\footnotemarkC \]
where we use standard notation: $V^{2,3}$, for example, means ``$V$ with
its $x$ strand renamed $2$ and its $y$ strand renamed $3$'' and $V^{1,23}$
means ``V with its $x$ strand renamed $1$ and its $y$ strand doubled to
become strands $2$ and $3$''. With the language of
Definition~\ref{def:Operations}, this is $V^{2,3}=V\act d\sigma^x_2\act
d\sigma^y_3$ and $V^{1,23} = V\act d\sigma^x_1\act d\Delta^y_{23}$.

\footnotetextC{And here is $\Phi_V$, to degree 4:

\shortdialoginclude{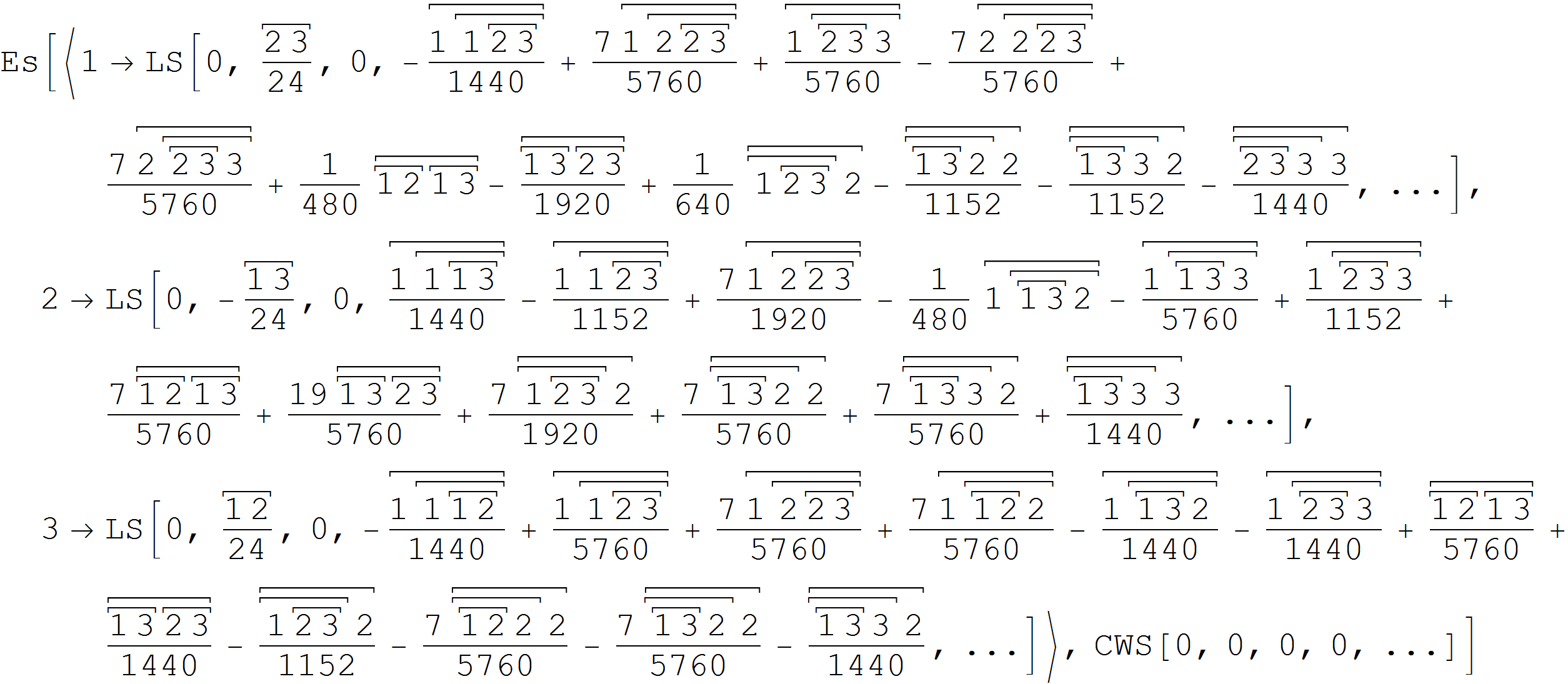}}

$\Phi_V$ satisfies the pentagon equation.\footnoteC{Indeed,

\shortdialoginclude{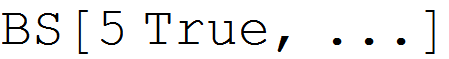}}
If our $V$ also satisfies the Twist Equation, then $\Phi_V$ also satisfies
the hexagon equations (though we do not test that here). Finally, Alekseev
and Torossian~\cite{AT} prove that if the tree part of $\Phi_V$ is written
as an exponential $\exp(l\phi)$ of an element $\phi$ of $\tder_3$,
then in fact $\phi\in\sder_3$, where as in~\cite{AT},
$\glosm{sder}{\sder}_n$ is the space of ``special derivations in
$\tder_n$'', the derivations which annihilate the sum of all
generators on $\FL_n$\footnotemarkC.

\footnotetextC{We convert $\Phi_V$ to the $E_l$ presentation and take
its first (tree) part and call it $\phi$, and then we verify that
$[x_1,\phi_1] + [x_2,\phi_2] + [x_3,\phi_3] = 0$:

\shortdialoginclude{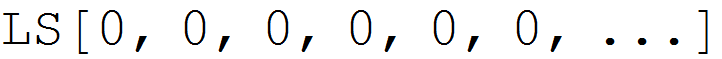}}

\vskip -2mm
\noindent\parbox[b]{\textwidth-2.0in}{
  \parpic[l]{\begin{picture}(0,0)%
\includegraphics{figs/Topology.pstex}%
\end{picture}%
%
%  pstex_opts: -m 0.3333 
%
\setlength{\unitlength}{1316sp}%
\begingroup\makeatletter\ifx\SetFigFont\undefined%
\gdef\SetFigFont#1#2#3#4#5{%
  \reset@font\fontsize{#1}{#2pt}%
  \fontfamily{#3}\fontseries{#4}\fontshape{#5}%
  \selectfont}%
\fi\endgroup%
\begin{picture}(1052,1180)(0,-190)
\end{picture}%
  \hspace{-6pt}}%
  \noindent%
 The topological meaning of ``$\phi\in\sder_3$'' is that one
may perform a sequence of four $R4$ moves to slide a strand underneath
a tetrahedron, as shown on the right.
}\hfill\begin{picture}(0,0)%
\includegraphics{figs/Phisder.pstex}%
\end{picture}%
%
%  pstex_opts: -m 0.7 
%
\setlength{\unitlength}{2763sp}%
\begingroup\makeatletter\ifx\SetFigFont\undefined%
\gdef\SetFigFont#1#2#3#4#5{%
  \reset@font\fontsize{#1}{#2pt}%
  \fontfamily{#3}\fontseries{#4}\fontshape{#5}%
  \selectfont}%
\fi\endgroup%
\begin{picture}(3024,1074)(1639,-373)
\put(3151,164){\makebox(0,0)[b]{\smash{{\SetFigFont{8}{9.6}{\rmdefault}{\mddefault}{\updefault}{\color[rgb]{0,0,0}$=$}%
}}}}
\end{picture}%

\vskip 2mm

Recall that there is a map
$\glosm{alpha}{\alpha}\colon\frakt_n\to\calA_{\text{prim}}^w(\uparrow_n)$
(equivalently, $\alpha\colon\calU(\frakt_n)\to\calA^w(\uparrow_n)$),
defined by its values on the generators by sending $t_{ij}$ to a sum of a
single arrow from strand $i$ to strand $j$ plus a single arrow from strand
$j$ to strand $i$: $t_{ij}\mapsto\tensor[_i]{\rightarrowdiagram}{_j}
+ \tensor[_i]{\leftarrowdiagram}{_j}$. Using the map $\alpha$, every
Drinfel'd associator becomes an associator in $\calA^w$.\footnoteC{Indeed,
we define a map {\tt DK2Es} which takes Drinfel'd-Kohno series to elements
of $\calA^w$ given in the $E_s$ presentation by applying the built-in
$\alpha${\tt Map}, adding $0$ wheels, and applying the $E_l$ to $E_s$
conversion $\Gamma$. Applying this map to the Drinfel'd associator
$\Phi_0$ computed before, we get and associators in $\calA^w$:

\shortdialoginclude{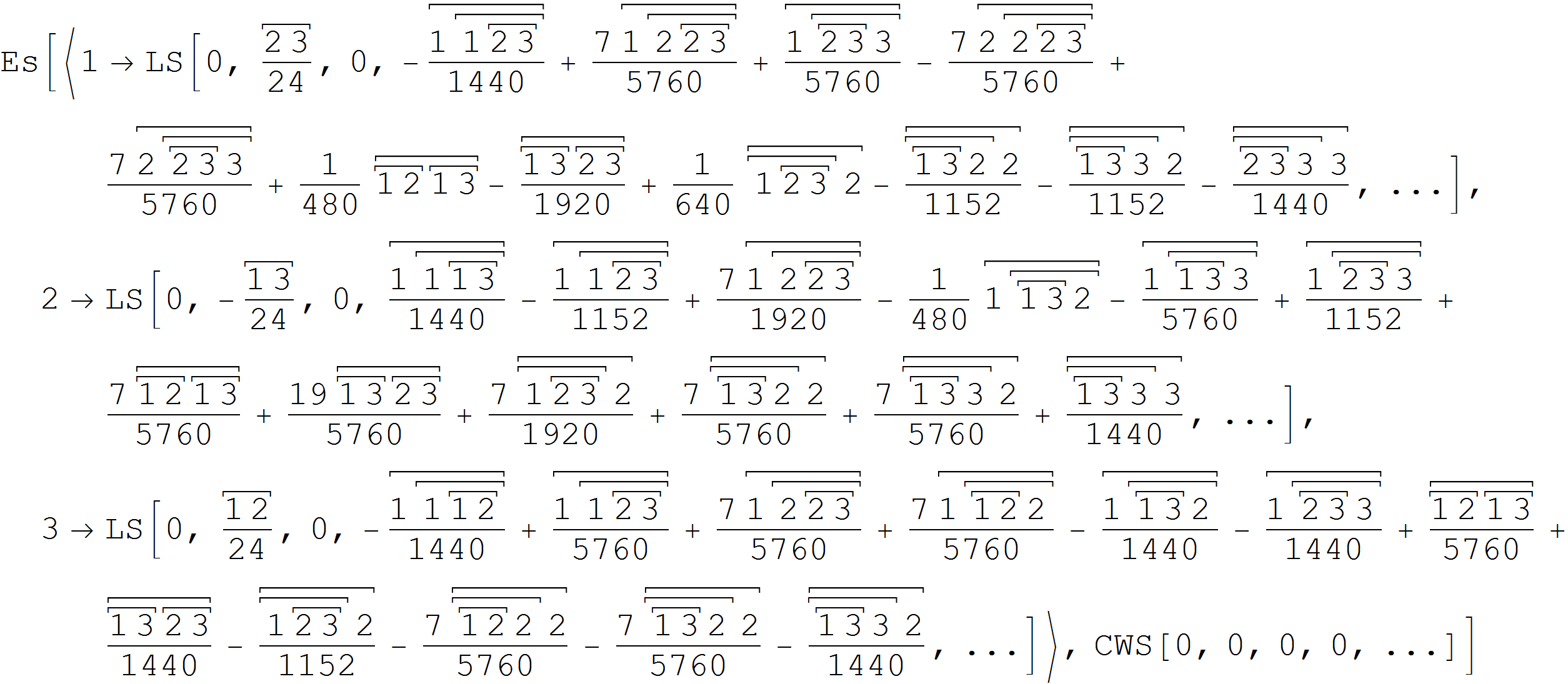}

The result matches $\Phi_V$, computed before, to the degree shown. But this
is only because both associators are supported in even degrees, and there's
a unique even associator in $\calA^w$ up to degree $4$. In degree $8$
these two associators diverge.}

  \parpic[l]{\begin{picture}(0,0)%
\includegraphics{figs/Topology.pstex}%
\end{picture}%
%
%  pstex_opts: -m 0.3333 
%
\setlength{\unitlength}{1316sp}%
\begingroup\makeatletter\ifx\SetFigFont\undefined%
\gdef\SetFigFont#1#2#3#4#5{%
  \reset@font\fontsize{#1}{#2pt}%
  \fontfamily{#3}\fontseries{#4}\fontshape{#5}%
  \selectfont}%
\fi\endgroup%
\begin{picture}(1052,1180)(0,-190)
\end{picture}%
  \hspace{-6pt}}%
  \noindent%
 In topology, $\alpha$ is the associated graded of the ``do
nothing'' map $\glosm{uva}{a}$ which maps ordinary knots to virtual knots.
$\horizontalchord\mapsto\rightarrowdiagram + \leftarrowdiagram$ because
$\horizontalchord\sim\doublepoint\sim\overcrossing-\undercrossing \mapsto
(\overcrossing-\virtualcrossing) + (\virtualcrossing-\undercrossing) \sim
\semivirtualover + \semivirtualunder \sim
\rightarrowdiagram + \leftarrowdiagram$. See
\cite[Section~\ref{1-subsubsec:RelWithu}]{WKO1} and
\cite[Section~\ref{2-subsec:sder}]{WKO2}.

  \parpic[l]{\begin{picture}(0,0)%
\includegraphics{figs/Lie.pstex}%
\end{picture}%
%
%  pstex_opts: -m 0.36 
%
\setlength{\unitlength}{1421sp}%
\begingroup\makeatletter\ifx\SetFigFont\undefined%
\gdef\SetFigFont#1#2#3#4#5{%
  \reset@font\fontsize{#1}{#2pt}%
  \fontfamily{#3}\fontseries{#4}\fontshape{#5}%
  \selectfont}%
\fi\endgroup%
\begin{picture}(960,990)(46,0)
\end{picture}%
  \hspace{-6pt}}%
  \noindent%
 In Lie theory, the existence of $\alpha$ corresponds to the
fact that the invariant metric on $I\frakg=\frakg\ltimes\frakg^\ast$
(represented by an undirected chord) is the sum of the two possible
contractions of a space with its dual in
$(\frakg\ltimes\frakg^\ast)\otimes(\frakg\ltimes\frakg^\ast)$ (the two
arrows).

  \parpic[l]{\begin{picture}(0,0)%
\includegraphics{figs/AT.pstex}%
\end{picture}%
%
%  pstex_opts: -m 0.36 
%
\setlength{\unitlength}{1421sp}%
\begingroup\makeatletter\ifx\SetFigFont\undefined%
\gdef\SetFigFont#1#2#3#4#5{%
  \reset@font\fontsize{#1}{#2pt}%
  \fontfamily{#3}\fontseries{#4}\fontshape{#5}%
  \selectfont}%
\fi\endgroup%
\begin{picture}(1041,1009)(118,-19)
\end{picture}%
  \hspace{-8pt}}%
  \noindent%
 The~\cite[Proposition~3.11]{AT} version of $\alpha$ is the
map $\frakt_n\to\atsder_n \subset \attder_n$ taking $t_{ij}$ to
$\partial\left(i\to x_j,\, j\to x_i,\, (k\neq i,j)\to 0\right)$.

\draftcut
\subsection{Solving the Kashiwara-Vergne Equations Using a Drinfel'd Associator}
\label{subsec:KVandAssoc} Following~\cite{WKO3} (in a deeper
sense, following~\cite{AlekseevEnriquezTorossian:ExplicitSolutions}),
we know that an element $V$ solving the KV equations \eqref{eq:HardR4}
and \eqref{eq:UnzipEquations} can be computed from a Drinfel'd associator
$\Phi$ by first computing the invariant $\glosm{ZB}{Z_B} = Z^u(B)$ 
of the ``buckle'' $\glosm{B}{B}$, shown below
both as a knotted trivalent graph and as a product of associators,
then puncturing strands 1 and 3 and capping strands 2 and 4 from below,
and then regarding the result in $\calA^w(\uparrow_2)$ by applying an
``Etingof-Kazhdan (EK) isomorphism'':\footnotemarkC
\[
  B = \begin{array}{c}\input{figs/Buckle.pstex_t}\end{array}
  \mapsto Z_B = (\Phi^{-1})^{13,2,4}\Phi^{1,3,2}R^{23}\Phi^{-1}\Phi^{12,3,4}
  \xrightarrow{\text{puncture, cap, EK}} V.
\]

\footnotetextC{We start with a straightforward computation of $Z_B$:

\shortdialoginclude{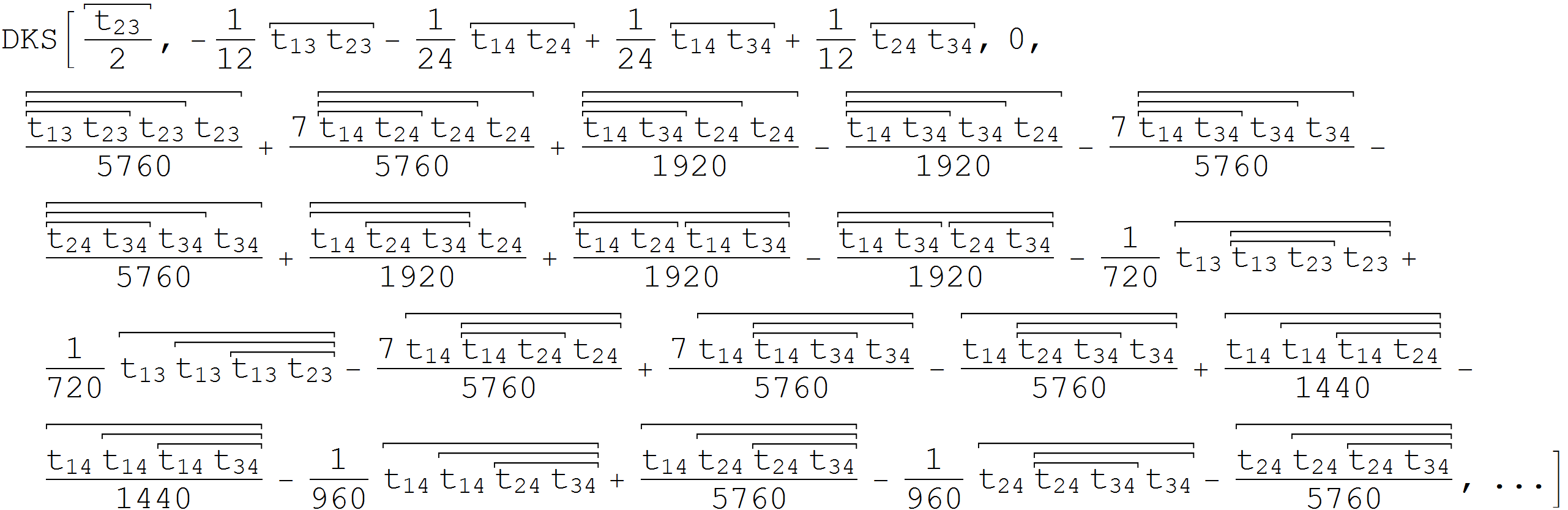}

In the $E_s$ presentation, ``puncture'' is $t\eta$. So we puncture strands
1 and 3:

\shortdialoginclude{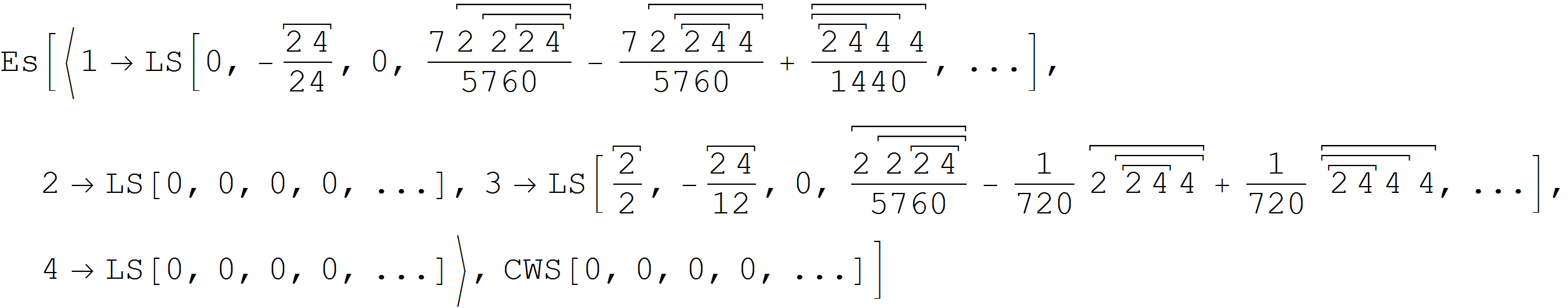}

At this point we would normally need to cap and apply EK. But fortunately,
strands 2 and 4 carry no arrow heads (as can be seen in the above output),
so there is no need to cap them and the EK isomorphisms act by doing
nothing. Hence apart from some obvious renaming, the above is already a
solution of the KV equations. It matches with the previously-computed $V$
to degree 4 but diverges from it in degree 8 (not shown here). This is
consistent with the result in~\eqref{tab:dims}, which shows that
non-uniqueness starts only in degree $8$.}

\Needspace{1in}
\parpic[r]{\input{figs/nu.pstex_t}}
Likewise following~\cite{WKO3}, we know that $\Cap=\alpha(\nu^{1/4})$, where
$\glosm{nu}{\nu}$ is the Kontsevich integral of the unknot, or the
inverse of the associator-combination shown on the right and given by
the formula $\alpha(\nu^{-1})=\Phi\act\alpha\act \dS^2\act dm^{32}_2\act
dm^{21}_1$.\footnotemarkC\ (Note that this computation uses the operation
$\dS^a$, which is not easily available in the $E_l$ presentation).

\footnotetextC{Indeed here is $\nu^{-1}$, followed by a verification that
$\nu^{-1}\Cap^4$ is trivial:

\shortdialoginclude{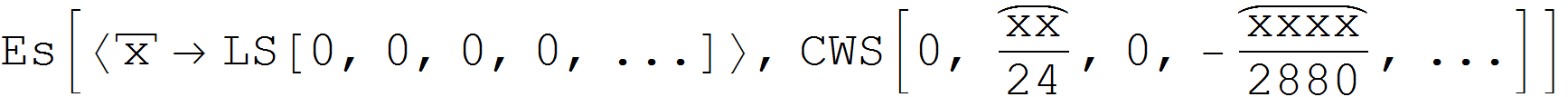}

\shortdialoginclude{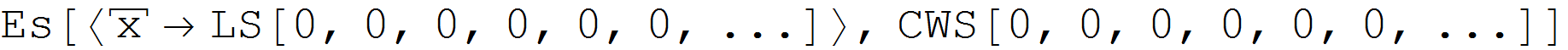}}

\draftcut
% The "minipage" here is to prevent a very funny bad interaction of the
% section title with the "computations below" rule.
\noindent\begin{minipage}{\textwidth}
\subsection{A Potential $S_4$ Action on Solutions of KV}
\label{subsec:Trivolution} In~\cite{Bar-NatanDancso:KTG},
Z.~Dancso and I discussed how ``the expansion of a tetrahedron''
can be interpreted as an associator valued in the appropriate
space $\calA^u(\glosm{tetrahedron}{\tetrahedron})\cong\calA^u(\uparrow_3)$ (see
also~\cite{Thurston:Shadow}). The symmetry group of an oriented
tetrahedron is the alternating group $A_4$, and hence $A_4$
acts on the set of all associators in $\calA^u(\uparrow_3)$
(note that while the action of the permutation group $S_3$ on
$\calA^u(\uparrow_3)$ is obvious, its extension to an action of
$S_4$ is non-obvious and is best understood using the isomorphism
$\calA^u(\tetrahedron)\cong\calA^u(\uparrow_3)$). The unitarity
equation~\eqref{eq:PhiUnitarity} means that odd permutations map
associators to objects whose inverses are associators; with some abuse
of language we simply say that ``$S_4$ acts on the set of associators''
(really, it acts on ``associators and inverse-associators''). As there
are bi-directional relations between associators and solutions of the
KV equations, we can expect an action of $S_4$ on the set of solutions
of the KV equations and their inverses.
\end{minipage}

\vskip 1mm
As mathematicians, Z.~Dancso and I only lightly explored this
potential action of $S_4$; we wrote down what we think are the formulas
inherited from the action on associators, but on the formal level, we've
verified almost nothing. Yet computer experiments, described below,
suggest that our formulas are correct and that they have the properties
described below.

\noindent{\bf The first $\bbZ/2$ action} is the involution $\tau$ discussed
in Section~\ref{subsec:tau}. We have nothing further to add.

\noindent{\bf The second $\bbZ/2$ action} is the involution
$\glosm{rho2}{\rho_2}$ of $\calA^w$ which multiplies every degree $d$
element by $(-1)^d$. Solutions $V$ of the KV equations are not invariant
under $\rho_2$. Yet if $V_0$ is the solution computed in this paper
then $V_1\coloneqq R^{-1/2}V_0$ is invariant under $\rho_2$, at least
experimentally. Alternatively, $V_0$ is (experimentally) invariant under
$\rho_2'\coloneqq R\rho_2$.\footnoteC{Indeed,

\dialoginclude{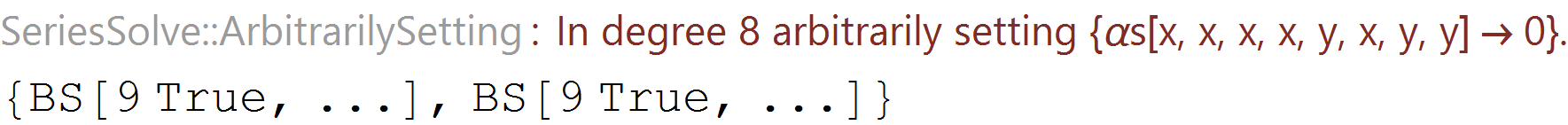}
}

\noindent{\bf A $\bbZ/3$ action.} For $\xi\in\calA^w({x,y})$
let $\glosm{rho3}{\rho_3}(\xi)\coloneqq \xi\act \dS^y \act d\Delta^y_{yz} \act
dm^{xz}_x \act d\sigma^{xy}_{yx}$, where $d\sigma^{xy}_{yx}$ simply
means ``swap the labels $x$ and $y$''. Then $\rho_3$ is a trivolution
($(\rho_3)^3=1$)\footnoteC{Indeed for a random $\xi_c$,
$\xi_c\act\rho_3\act\rho_3\act\rho_3=\xi_c$:

\dialoginclude{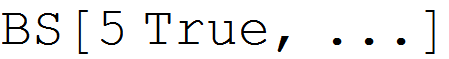}
}, and a renormalized version of $V_0$, namely $V_2\coloneqq V_0 \ast
\Theta^{-1/4} \ast \exp\left(\frac{\wideparen{x}-\wideparen{y}}{12}\right)
\ast d\Delta^x_{xy}(\Cap^2)$ is, at least experimentally, invariant
under the action of $\rho_3$.\footnoteC{Indeed,

\begin{minipage}{\textwidth}\vskip 1mm\dialoginclude{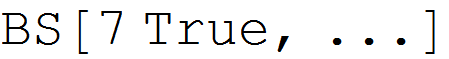}\end{minipage}}

\draftcut\section{Glossary of notation} \label{sec:glossary}

Icons, then Greek letters, then Latin, and then symbols:

\noindent
{\small\begin{multicols}{2}
  \parpic[l]{\begin{picture}(0,0)%
\includegraphics{figs/Topology.pstex}%
\end{picture}%
%
%  pstex_opts: -m 0.3333 
%
\setlength{\unitlength}{1316sp}%
\begingroup\makeatletter\ifx\SetFigFont\undefined%
\gdef\SetFigFont#1#2#3#4#5{%
  \reset@font\fontsize{#1}{#2pt}%
  \fontfamily{#3}\fontseries{#4}\fontshape{#5}%
  \selectfont}%
\fi\endgroup%
\begin{picture}(1052,1180)(0,-190)
\end{picture}%
\,\begin{picture}(0,0)%
\includegraphics{figs/Lie.pstex}%
\end{picture}%
%
%  pstex_opts: -m 0.36 
%
\setlength{\unitlength}{1421sp}%
\begingroup\makeatletter\ifx\SetFigFont\undefined%
\gdef\SetFigFont#1#2#3#4#5{%
  \reset@font\fontsize{#1}{#2pt}%
  \fontfamily{#3}\fontseries{#4}\fontshape{#5}%
  \selectfont}%
\fi\endgroup%
\begin{picture}(960,990)(46,0)
\end{picture}%
  \,\,\begin{picture}(0,0)%
\includegraphics{figs/AT.pstex}%
\end{picture}%
%
%  pstex_opts: -m 0.36 
%
\setlength{\unitlength}{1421sp}%
\begingroup\makeatletter\ifx\SetFigFont\undefined%
\gdef\SetFigFont#1#2#3#4#5{%
  \reset@font\fontsize{#1}{#2pt}%
  \fontfamily{#3}\fontseries{#4}\fontshape{#5}%
  \selectfont}%
\fi\endgroup%
\begin{picture}(1041,1009)(118,-19)
\end{picture}%

  \hspace{-6pt}}%
  \noindent%
 Links with topology, finite-dimensional Lie theory, and the
Alekseev-Torossian paper~\cite{AT}.

\parpic[l]{
  \imagetop{\begin{picture}(0,0)%
\includegraphics{figs/face.pstex}%
\end{picture}%
%
%  pstex_opts: -m 0.2 
%
\setlength{\unitlength}{789sp}%
\begingroup\makeatletter\ifx\SetFigFont\undefined%
\gdef\SetFigFont#1#2#3#4#5{%
  \reset@font\fontsize{#1}{#2pt}%
  \fontfamily{#3}\fontseries{#4}\fontshape{#5}%
  \selectfont}%
\fi\endgroup%
\begin{picture}(1240,1240)(-19,-381)
\end{picture}%
}%
  \,\,\imagetop{\begin{picture}(0,0)%
\includegraphics{figs/human.pstex}%
\end{picture}%
%
%  pstex_opts: -m 0.2 
%
\setlength{\unitlength}{789sp}%
\begingroup\makeatletter\ifx\SetFigFont\undefined%
\gdef\SetFigFont#1#2#3#4#5{%
  \reset@font\fontsize{#1}{#2pt}%
  \fontfamily{#3}\fontseries{#4}\fontshape{#5}%
  \selectfont}%
\fi\endgroup%
\begin{picture}(1244,2442)(-21,-1583)
\end{picture}%
}%
  \,\,\imagetop{\begin{picture}(0,0)%
\includegraphics{figs/machine.pstex}%
\end{picture}%
%
%  pstex_opts: -m 0.2 
%
\setlength{\unitlength}{789sp}%
\begingroup\makeatletter\ifx\SetFigFont\undefined%
\gdef\SetFigFont#1#2#3#4#5{%
  \reset@font\fontsize{#1}{#2pt}%
  \fontfamily{#3}\fontseries{#4}\fontshape{#5}%
  \selectfont}%
\fi\endgroup%
\begin{picture}(1244,1394)(-21,-608)
\end{picture}%
}%
  \hspace{-6pt}
}%
\noindent Human input, multi-line human input, and computer output.

\vskip 2mm
\parpic[l]{
  \imagetop{\includegraphics[width=6mm]{figs/NotebookIcon.eps}}%
    \raisebox{-2mm}{\href{\web/AwCalculus.m}{\tiny\tt FL}}%
  \,\,\imagetop{\includegraphics[width=6mm]{figs/NotebookIcon.eps}}%
    \raisebox{-2mm}{\href{\web/AwCalculus.m}{\tiny\tt AC}}%
  \hspace{-6pt}
}%
\noindent Source code quotes from the Mathematica packages
\href{\web/FreeLie.m}{\tt FreeLie.m} and \href{\web/AwCalculus.m}{\tt
AwCalculus.m} \cite{WKO4}.

\raggedright\begin{list}{}{
  \renewcommand{\makelabel}[1]{#1\hfil}
}
\item

% \alpha
\glosi{alpha}{$\alpha$}{a map $\frakt_n\to\calA_{\text{prim}}^w$ / $\calA^u\to\calA^w$}
% \beta

% \gamma
\glosi{TWGamma}{$\Gamma$}{the conversion $\TW_l\to\TW_s$}
\glosi{Gamma}{$\Gamma(\lambda)$}{$\Gamma_1(\lambda)$}
\glosi{Gammat}{$\Gamma_t(\lambda)$}{solution of $e^{-t\partial_\lambda}=C^{\Gamma_t(\lambda)}$}

% \delta
\glosi{Delta}{$\Delta$}{a co-product}
\glosi{delta}{$\delta$}{double all strands $\calA^w(S)\!\to\!\calA^w(S;S)$}

% \epsilon
% \zeta

% \eta
\glosi{eta}{$\eta$}{a co-unit}

% \theta
\glosi{Theta}{$\Theta$}{$\exp(\rightarrowdiagram+\leftarrowdiagram)$}

% \iota
\glosi{ElDef}{$\iota$}{the embedding $\CW\to\calA^w$}

% \kappa

% \lambda
\glosi{expectation}{$\lambda$}{generic element of $\FL(S)^S$}
\glosi{TWLambda}{$\Lambda$}{the conversion $\TW_s\to\TW_l$}
\glosi{Lambda}{$\Lambda(\lambda)$}{$\Lambda_1(\lambda)$}
\glosi{Lambdat}{$\Lambda_t(\lambda)$}{solution of $C^{t\lambda}=e^{-\partial_{\Lambda_t(\lambda)}}$}

% \mu

% \nu
\glosi{nu}{$\nu$}{Kontsevich integral of the unknot}

% \xi
% \omicron

% \pi
\glosi{piA}{$\pi_A$}{projection on ``framing part''}
\glosi{piT}{$\pi_T$}{projection on trees}
\glosi{pidowncap}{$\pi_{\!\downcap}$}{a projection on wheels}

% \rho
\glosi{rho2}{$\rho_2$}{an involution on $\calA^w$}
\glosi{rho3}{$\rho_3$}{a trivolution on $\calA^w(x,y)$}
% \sigma

% \tau
\glosi{tau}{$\tau$}{an involution on $\SolKV$}

% \upsilon

% \phi
\glosi{Phi}{$\Phi$}{a Drinfel'd associator}
\glosi{PhiV}{$\Phi_V$}{an associator in $\calA^w$}

% \chi
% \psi
% \omega
\glosi{expectation}{$\omega$}{generic element of $\CW(S)$}

\item

%a
\glosi{a}{$a,\bar{a},a_i,b,\ldots$}{generic strand labels}
\glosi{uva}{$a$}{the inclusion usual$\hookrightarrow$virtual}
\glosi{A}{$A$}{Abelian lie algebra}
\glosi{fraka}{$\fraka$}{\cite{AT} notation for $A$}
\glosi{adugamma}{$\ad_u^\gamma$}{a derivation on $\FL(T)$}
\glosi{wRels}{$\aAS$}{the directed AS relation}
\glosi{calAw}{$\calA^w$}{arrow-diagram spaces}
\glosi{calAwexp}{$\calA^w_{\exp}$}{exponentials in $\calA^w$}
\glosi{calAwHT}{$\calA^w(H;T)$}{arrow-diagram space on heads-tails skeleton}
\glosi{AwHTexp}{$\calA^w_{\exp}(H;T)$}{exponentials in $\calA^w(H;T)$}

%b
\glosi{B}{$B$}{the ``buckle'' KTG}
\glosi{BCH}{$\BCH$}{the Baker-Campbell-Hausdorff series}
\glosi{BCHtb}{$\BCH_{tb}$}{$\BCH$ relative to $tb$}

%c
\glosi{C}{$C^\lambda$}{conjugating generators by exponentials}
\glosi{Cap}{$\Cap$}{$Z^w$ of a knot-theoretic cap}
\glosi{CP}{$CP$}{the $CP$ relation}
\glosi{Cu}{$C_u^{\gamma}$}{$C^{(u\to\gamma)}$}
\glosi{CW}{$\CW$}{cyclic words}

%d
\glosi{D}{$D$}{a diagram in $\calA^w$}
\glosi{dDelta}{$d\Delta$}{strand doubling in $\calA^w(S)$}
\glosi{ldDelta}{$d\Delta$}{``strand doubling'' in $\TW_l$}
\glosi{deta}{$d\eta$}{strand deletion in $\calA^w(S)$}
\glosi{ldeta}{$d\eta$}{``strand deletion'' in $\TW_l$}
\glosi{dsigma}{$d\sigma$}{strand renaming in $\calA^w(S)$}
\glosi{ldsigma}{$d\sigma$}{``strand renaming'' in $\TW_l$}
\glosi{dA}{$\dA$}{strand adjoint in $\calA^w(S)$}
\glosi{ldA}{$\dA$, $dA^S$}{``strand adjoint'' in $\TW_l$}
\glosi{der}{$\der$}{derivations of $\FL$}
\glosi{atder}{$\atder$}{\cite{AT} notation for $\der$}
\glosi{atdiv}{$\atdiv$}{$\sum_u\atdiv_u$}
\glosi{atdivu}{$\atdiv_u$}{a ``self-action'' map $\FL(S)\to\CW(S)$}
\glosi{dm}{$dm$}{strand stitching in $\calA^w(S)$}
\glosi{dS}{$\dS$}{strand antipode in $\calA^w(S)$}
\glosi{ldS}{$\dS$, $dS^S$}{``strand antipode'' in $\TW_l$}

%e
\glosi{calE}{$\calE$}{the Euler operator}
\glosi{Ef}{$E_f$}{the factored presentation}
\glosi{El}{$E_l$}{the lower-interlaced presentation}
\glosi{Es}{$E_s$}{the split presentation}
\glosi{Eu}{$E_u$}{the upper-interlaced presentation}
\glosi{es}{$e_s$}{a map $\FL(T)^H\to\calA^w_{\exp}(H;T)$}
\glosi{exptn}{$\exp\frakt_n$}{the exponential group of $\frakt_n$}

%f
\glosi{F}{$F$}{solution of the generalized KV equations}
\glosi{fg}{$f,g$}{solution of the original KV equations}
\glosi{FL}{$\FL$}{free Lie algebra}

%g
\glosi{frakg}{$\frakg$}{a finite-dimensional Lie algebra}

%h
\glosi{H}{$H$}{a set of head labels}
\glosi{h}{$h_i$}{head labels}
\glosi{hdeg}{$h^{\deg}$}{degree-scaling}
\glosi{hDelta}{$h\Delta$}{head-strand doubling in $\calA^w(H;T)$}
\glosi{shDelta}{$h\Delta$}{``head-strand doubling'' in $\TW_s$}
\glosi{heta}{$h\eta$}{deleting a head-strand in $\calA^w(H;T)$}
\glosi{sheta}{$h\eta$}{``deleting a head-strand'' in $\TW_s$}
\glosi{hsigma}{$h\sigma$}{head-strand renaming in $\calA^w(H;T)$}
\glosi{shsigma}{$h\sigma$}{``head-strand renaming'' in $\TW_s$}
\glosi{hA}{$hA$}{head-strand adjoint in $\calA^w(H;T)$}
\glosi{shA}{$hA$}{``head-strand adjoint'' in $\TW_s$}
\glosi{hm}{$hm$}{head-strand stitching in $\calA^w(H;T)$}
\glosi{shm}{$hm$}{``head-strand stitching'' in $\TW_s$}
\glosi{hS}{$hS$}{head-strand antipode in $\calA^w(H;T)$}
\glosi{shS}{$hS$}{``head-strand antipode'' in $\TW_s$}

%i
\glosi{Ifrakg}{$I\frakg$}{$\frakg\ltimes\frakg^\ast$}
\glosi{wRels}{$\aIHX$}{the directed IHX relation}

%j
\glosi{j}{$j$}{a ``$\log$-Jacobian'' $\FL\to\CW$}
\glosi{Ju}{$J_u$}{a ``partial Jacobian'' $\FL\to\CW$}
%k

%l
\glosi{ElDef}{$l$}{the lower embedding $\FL(S)^S\to\calA^w$}
\glosi{lie}{$\lie$}{\cite{AT} notation for $\FL$}

%m
%n
%o

%p
\glosi{Aprimw}{$\calA_{\text{prim}}^w$}{the primitives in $\calA^w$}
\glosi{AprimwHT}{$\calA_{\text{prim}}^w(H;T)$}{the primitives in $\calA^w(H;T)$}

%q

%r
\glosi{R}{$R$}{$R(1,2)$}
\glosi{Rab}{$R^{\pm 1}(a,b)$}{$Z^w$ of a single $\pm$ crossing}
\glosi{Rl}{$R_l^\pm$}{$R^{\pm 1}$ in $\TW_l$}
\glosi{Rs}{$R_s^\pm$}{$R^{\pm 1}$ in $\TW_s$}
\glosi{RC}{$RC^{-\lambda}$}{inverse of $C^\lambda$}
\glosi{RCu}{$RC_u^{\gamma}$}{$RC^{(u\to\gamma)}$}

%s
\glosi{S}{$S$}{a set of strands}
\glosi{calS}{$\calS$}{a symmetric algebra}
\glosi{sder}{$\sder$}{``special'' derivations}
\glosi{wRels}{$\aSTU$}{a directed STU relation}

%t
\glosi{T}{$T$}{a set of tail labels}
\glosi{t}{$t_i$}{head labels}
\glosi{tij}{$t_{ij}$}{generators of $t_{ij}$}
\glosi{fraktn}{$\frakt_n$}{the Drinfel'd-Kohno algebra}
\glosi{tDelta}{$t\Delta$}{tail-strand doubling in $\calA^w(H;T)$}
\glosi{stDelta}{$t\Delta$}{``tail-strand doubling'' in $\TW_s$}
\glosi{teta}{$t\eta$}{deleting a tail-strand in $\calA^w(H;T)$}
\glosi{steta}{$t\eta$}{``deleting a tail-strand'' in $\TW_s$}
\glosi{tsigma}{$t\sigma$}{tail-strand renaming in $\calA^w(H;T)$}
\glosi{stsigma}{$t\sigma$}{``tail-strand renaming'' in $\TW_s$}
\glosi{tA}{$tA$}{tail-strand adjoint in $\calA^w(H;T)$}
\glosi{stA}{$tA$}{``tail-strand adjoint'' in $\TW_s$}
\glosi{TAut}{$\TAut$}{the exponential group of $\tder$}
\glosi{tb}{$tb$}{tangential bracket}
\glosi{wRels}{$TC$}{the tails-commute relation}
\glosi{tder}{$\tder$}{tangential derivations}
\glosi{attder}{$\attder$}{\cite{AT} notation for $\tder$}
\glosi{tha}{$tha$}{tail-head action in $\calA^w(H;T)$}
\glosi{stha}{$tha$}{``tail-head action'' in $\TW_s$}
\glosi{thm}{$thm$}{tail-head stitching in $\calA^w(H;T)$}
\glosi{tm}{$tm$}{tail-strand stitching in $\calA^w(H;T)$}
\glosi{stm}{$tm$}{``tail-strand stitching'' in $\TW_s$}
\glosi{tru}{$\tr_u$}{a trace map $\FL(S)\to\CW(S)$}
\glosi{attr}{$\attr$}{\cite{AT} notation for $\CW$}
\glosi{tS}{$tS$}{tail-strand antipode in $\calA^w(H;T)$}
\glosi{stS}{$tS$}{``tail-strand antipode'' in $\TW_s$}
\glosi{expectation}{$\TW$}{trees and wheels}
\glosi{TWl}{$\TW_l$}{domain of $E_l$}
\glosi{Es}{$\TW_s$}{domain of $E_s$}

%u
\glosi{u}{$u$}{the upper embedding $\FL(S)^S\to\calA^w$}
\glosi{unzip}{$u$}{unzip operations}
\glosi{uinT}{$u,v,w$}{tail labels}
\glosi{calU}{$\calU$}{universal enveloping algebra}

%v
\glosi{V}{$V$}{$Z^w$ of a knot-theoretic vertex}
%w

%x
\glosi{xinH}{$x,y,z$}{head labels}

%y

%z
\glosi{ZB}{$Z_B$}{$Z^u$ of the buckle $B$}
\glosi{Zu}{$Z^u$}{the $\calA^u$ counterpart of $Z^w$}
\glosi{Zw}{$Z^w$}{a (universal) $\calA^w_{\exp}$-valued invariant}

\item
\glosi{act}{$\act$}{postfix operator application, ``composition done right''}
\glosi{rightarrowdiagram}{$\rightarrowdiagram$}{a single-arrow diagram}
\glosi{ast}{$\ast$}{the stacking product in $\calA^w(S)$}
\glosi{last}{$\ast$}{the ``stacking product'' in $\TW_l$}
\glosi{jail}{$\#$}{the stacking product in $\calA^w(H;T)$}
\glosi{sjail}{$\#$}{the ``stacking product'' in $\TW_s$}
\glosi{wjail}{$\#$}{a product on $\calA^w(S)$}
\glosi{Box}{$\Box$}{the co-product in $\calA^w(S)$}
\glosi{htBox}{$\Box$}{the co-product in $\calA^w(H;T)$}
\glosi{hdeg}{${-1}^{\deg}$}{degree-scaling with $h=-1$}
\glosi{ob}{$\ob{xy}$}{top-bracket notation}
\glosi{partial}{$\partial$}{the map $\FL(S)^S\to\der_S$}
\glosi{ldeta}{$\remove$}{set minus, array key removal}
\glosi{sqcup}{$\sqcup$}{a disjoint union in $\calA^w(S)$}
\glosi{lsqcup}{$\sqcup$}{``disjoint union'' in $\TW_l$}
\glosi{ssqcup}{$\sqcup$}{``disjoint union'' in $\TW_s$}
\glosi{Brutesqcup}{$\sqcup$}{a union made disjoint}
\glosi{htsqcup}{$\sqcup$}{a disjoint union in $\calA^w(H;T)$}
\glosi{uparrow}{$\uparrow_n$}{a skeleton labelled $S=\{1,\ldots,n\}$}
\glosi{wideparen}{$\wideparen{uvw}$}{a cyclic word}
\glosi{parenl}{$(\lambda;\,\omega)_l$}{generic element in $\TW_l$}
\glosi{parens}{$(\lambda;\,\omega)_s$}{generic element in $\TW_s$}
\glosi{parenu}{$(\lambda;\,\omega)_u$}{element in the domain of $E_u$}
\glosi{tb}{$[\cdot,\cdot]_{tb}$}{tangential bracket}

\item
\glosi{overcrossing}{$\overcrossing$}{an over-crossing}
\glosi{undercrossing}{$\undercrossing$}{an under-crossing}
\glosi{virtualcrossing}{$\virtualcrossing$}{a ``virtual'' crossing}
\glosi{PlusVertex}{$\raisebox{-1mm}{\begin{picture}(0,0)%
\includegraphics{figs/PlusVertex.pstex}%
\end{picture}%
%
%  pstex_opts: -m 0.8 
%
\setlength{\unitlength}{3947sp}%
\begingroup\makeatletter\ifx\SetFigFont\undefined%
\gdef\SetFigFont#1#2#3#4#5{%
  \reset@font\fontsize{#1}{#2pt}%
  \fontfamily{#3}\fontseries{#4}\fontshape{#5}%
  \selectfont}%
\fi\endgroup%
\begin{picture}(211,211)(3439,-1160)
\end{picture}%
}$}{the
  knot-theoretic ``vertex''}
\glosi{upcap}{${\raisebox{-1mm}{\begin{picture}(0,0)%
\includegraphics{figs/SmallCap.pstex}%
\end{picture}%
%
%  pstex_opts: -m 0.8 
%
\setlength{\unitlength}{3947sp}%
\begingroup\makeatletter\ifx\SetFigFont\undefined%
\gdef\SetFigFont#1#2#3#4#5{%
  \reset@font\fontsize{#1}{#2pt}%
  \fontfamily{#3}\fontseries{#4}\fontshape{#5}%
  \selectfont}%
\fi\endgroup%
\begin{picture}(64,201)(151,-1123)
\end{picture}%
}}$}{a knot-theoretic ``cap''}
\glosi{tetrahedron}{$\tetrahedron$}{unknotted tetrahedron}

\end{list}
\end{multicols}}

\draftcut

\if\draft y
  \newpage
  Everything below is to be blanked out before the completion of this paper.
  
\begin{discussion} \label{disc:onlie} For use in the next section, note
that both $\calA^w(\uparrow_S)$ and $\calA(S;\,S)$ are associative algebras
(the former using the stacking product of Equation~\eqref{eq:TubeProduct} and
the latter using that of Equation~\eqref{eq:AHTStacking}), yet $\delta$ is
not multiplicative and hence it does not restrict to a Lie morphism on
primitives. Instead, on primitives
$(\lambda_1;\omega_1),(\lambda_2;\omega_2)\in\TW(S)$ we have

\begin{equation} \label{eq:BracketComparissons}
  \delta[l\lambda_1+\iota\omega_1,\,l\lambda_2+\iota\omega_2]
  = [\delta(l\lambda_1+\iota\omega_1),\,\delta(l\lambda_2+\iota\omega_2)]
  + e_s(\partial_{\lambda_1}\lambda_2;\,\partial_{\lambda_1}\omega_2)
  - e_s(\partial_{\lambda_2}\lambda_1;\,\partial_{\lambda_2}\omega_1).
\end{equation}
\end{discussion}

\noindent{\em Proof of Equation~\eqref{eq:BracketComparissons}.}

\begin{proof} $E_l$ and $E_f$ both plant wheels at the top, and as tails
commute, they do so in the same manner. So $\omega'=\omega$ and we only
need to show Equation~\eqref{eq:convertion} at tree level (meaning, modulo
wheels). We will show that for every scalar $t$,
\begin{equation} \label{eq:treelevel}
  \exp(l(t\lambda))=\exp_\#(e_s(\Gamma(\lambda,t)))\act\delta^{-1};
\end{equation}
the desired result is the specialization of Equation~\eqref{eq:treelevel}
to $t=1$. It is clear that Equation~\eqref{eq:treelevel} holds for some
unique $\Gamma_0=\{a\to\gamma_{0a}(t)\}$, that $\gamma_{0a}(0)=0$, and that
each coefficient of each $\gamma_{0a}(t)$ depends polynomialy on $t$, and
hence it is enough to show that $\Gamma_0$ satisfies the differential
equation in~\eqref{eq:GammaODE}.

MORE.

\end{proof}
 
  \section*{To Do}

{\bf Comments by Zsuzsi.}

The word ``certain'' is used too many times in the abstract. I think it
would sell better to be more specific there, even at the cost of a longer
abstract.

Similarly, I think you could make the first paragraph of the introduction
sound less enigmatic.

In the first topology paragraph on page 5, you could say ``for S, T, H some
sets'', or something along those lines to make it less cryptic?

In general I think the introduction and abstract should cater more to
readers not initiated, I think in general if someone doesn't understand
the introduction to a paper they are much less likely to bother reading
the rest. (Ps: I do love Figure 1.)

Occasionally I think it might make sense to group the math on one or two
full pages and then the corresponding computation on one or two full
pages.

{\bf {\tt FreeLie.m} to do.}
\FLQuote{ToDo}

\fi

\end{document}